\documentclass{amsart}
\usepackage{amssymb, amsmath, latexsym, amscd, graphicx, tabularx, color}
\usepackage{rotating}
\usepackage[all]{xy}
\usepackage[dvipdfm]{hyperref}
\usepackage[all]{hypcap}

\newtheorem{theorem}{Theorem}[section]
\newtheorem{lemma}[theorem]{Lemma}
\newtheorem{cor}[theorem]{Corollary}
\newtheorem{definition}[theorem]{Definition}
\newtheorem{proposition}[theorem]{Proposition}
\newtheorem{remark}[theorem]{Remark}
\newtheorem{example}[theorem]{Example}

\def\pagenumber{1}
\begin{document}
\setcounter{page}{\pagenumber}
\newcommand{\T}{\mathbb{T}}
\newcommand{\R}{\mathbb{R}}
\newcommand{\Q}{\mathbb{Q}}
\newcommand{\N}{\mathbb{N}}
\newcommand{\Z}{\mathbb{Z}}
\newcommand{\tx}[1]{\quad\mbox{#1}\quad}
\def\IM{\hbox{\rm im}\hskip 2pt}
\def\RANKP{\hbox{\rm rank}}
\parindent=0pt
\def\SRA{\hskip 2pt\hbox{$\joinrel\mathrel\circ\joinrel\to$}}

\title[Quantum Exotic PDE's]{\mbox{}\\[1cm] QUANTUM EXOTIC PDE's}
\author{Agostino Pr\'astaro}
\maketitle
\vspace{-.5cm}
{\footnotesize
\begin{center}
Department SBAI - Mathematics, University of Rome ''La Sapienza'', Via A.Scarpa 16,
00161 Rome, Italy. \\
E-mail: {\tt Agostino.Prastaro@uniroma1.it; Prastaro@dmmm.uniroma1.it}
\end{center}
}

\noindent
\begin{abstract}
Following the previous works on the A. Pr\'astaro's formulation of algebraic topology of quantum (super) PDE's, it is proved that a canonical Heyting algebra ({\em integral Heyting algebra}) can be associated to any quantum PDE. This is directly related to the structure of its global solutions. This allows us to recognize a new inside in the concept of quantum logic for microworlds. Furthermore, the Pr\'astaro's geometric theory of quantum PDE's is applied to the new category of {\em quantum hypercomplex manifolds}, related to the well-known Cayley-Dickson construction for algebras. Theorems of existence for local and global solutions are obtained for (singular) PDE's in this new category of noncommutative manifolds. Finally the extension of the concept of exotic PDE's, recently introduced by A.Pr\'astaro, has been extended to quantum PDE's. Then a smooth quantum version of the quantum (generalized) Poincar\'e conjecture is given too. These results extend ones for quantum (generalized) Poincar\'e conjecture, previously given by A. Pr\'astaro.
\end{abstract}


\vspace{.08in} \noindent \textbf{Keywords}: Integral bordisms in quantum PDE's;
Existence of local and global solutions in quantum PDE's; Conservation laws;
Crystallographic groups; Singular quantum PDE's; PDE's on quantum hypercomplex manifolds category; PDE's on quantum hypercomplex (supermanifolds) category; Quantum exotic spheres; Quantum exotic PDE's; Quantum smooth (generalized) Poincar\'e conjecture.

\section{\bf INTRODUCTION}

This paper aims to further develop the A. Pr\'astaro's geometric theory of quantum PDE's, by considering three different (even if related) subjects in this theory. The first is a way to characterize quantum PDE's by means of suitable Heyting algebras. Nowadays these algebraic structures are considered important in order to characterize quantum logics and quantum topoi. (See, e.g. \cite{HEUNEN-LANDSMAN-SPITTERS, ISHAM-BUTTERFIELD, JOHNSTONE, MACLANE-MOERDIJK}.) Really we prove that to any quantum PDE can be associated a Heyting algebra, naturally arising from the algebraic topologic structure of the PDE's and that encodes its integral bordism group.

Another aspect that we shall consider is the extension of the category $\mathfrak{Q}$ of quantum manifolds,  or $\mathfrak{Q}_S$ of quantum supermanifolds, to the ones $\mathfrak{Q}_{hyper}$ of {\em quantum hypercomplex manifolds}. These generalizations are obtained by extending a quantum algebra $A$, in the sense of A. Pr\'astaro, by means of Cayley-Dickson algebras. In this way one obtains a new category of noncommutative manifolds, that are useful in some geometric and physical applications. In fact there are some fashioned research lines, concerning classical superstrings and classical super-2-branes, where one handles with algebras belonging just to some term in the Cayley-Dickson construction. Thus it is interesting to emphasize that the Pr\'astaro's geometric theory of quantum PDE's can be directly applied also to PDE's for such quantum hypercomplex manifolds. This allows us to encode quantum micro-worlds, by a general theory that goes beyond the classical simple description of classical extended objects, and solves also the problem of their quantization.

Let us emphasize that in some previous works we have formulated a geometric theory of quantum manifolds that are noncommutative manifolds, where the fundamental algebra is a suitable associative noncommutative topological algebra, there called {\em quantum algebra}. Extensions to quantum supermanifolds and quantum-quaternionic manifolds are also considered too. Furthermore, we have built a geometric theory of quantum PDE's in these categories of noncommutative manifolds, that allows us to obtain theorems of existence of local and global solutions, and constructive methods to build such solutions too.\cite{PRA0300, PRA0400, PRA0500, PRA01, PRA1, PRA2, PRA2-3, PRA3, PRA12, PRA13, PRA14, PRA15, PRA16, PRA18-0, PRA18}.

On the other hand it is well known that the sequences $\mathbb{R}\subset\mathbb{C}\subset\mathbb{H}\subset\mathbb{O}\subset\mathbb{S}$, where $\mathbb{S}$ is the sedenionic algebra, fit in the so-called {\em Cayley-Dikson construction}, $\mathbb{R}\subset\mathbb{C}\subset\mathbb{H}\subset\mathbb{O}\subset\mathbb{S}\subset\cdots\subset\mathbb{A}_r\subset\cdots$, where $\mathbb{A}_r$ is a Cayley algebra of dimension $2^r$: $\mathbb{A}_r\cong\mathbb{R}^{2^r}$, $r\ge 0$, $\mathbb{A}_0=\mathbb{R}$. Thus we can also consider {\em quantum hypercomplex algebras} $A\bigotimes\mathbb{A}_r$, $1\le r\le\infty$, where $A$ is a quantum algebra in the sense of Pr\'astaro, obtaining the natural inclusions $A\bigotimes_{\mathbb{R}}\mathbb{A}_r\hookrightarrow A\bigotimes_{\mathbb{R}}\mathbb{A}_{r+1}$. It is important to note that the Cayley algebras $\mathbb{A}_r$ are not associative for $r\ge 3$, (even if they are exponential associative\footnote{i.e., $z^{n+m}=z^nz^m$, $\forall z\in\mathbb{A}_r$ and $n,m\in\mathbb{N}$.}) hence also the corresponding quantum hypercomplex algebras are non-associative for $r\ge 3$, despite the associativity of the quantum algebra $A$. This fact introduces some particularity in the theory of PDE's on such algebras. The purpose of this paper is just to study which new behaviours have PDE's on the category $\mathfrak{Q}_{hyper}$ of {\em quantum hypercomplex manifolds}.

Let us emphasize that a first justification to use the category of quantum (super) manifolds to formulate PDE's that encode quantum physical phenomena, arises from the fact that quantized PDE's can be identified just with quantum (super) PDE's, i.e., PDE's for such noncommutative manifolds. However, to formulate equations just in the category $\mathfrak{Q}$ (or $\mathfrak{Q}_S$) allows us to go beyond the point of view of quantization of classical systems, and capture more general nonlinear phenomena in quantum worlds, that should be impossible to characterize by some quantization process. (See Refs. \cite{PRA12, PRA13, PRA14, PRA15, PRA16, PRA18-0, PRA18}.) In fact the concept of quantum algebra (or quantum superalgebra) is the first important brick to put in order to build a theory on quantum physical phenomena. In other words it is necessary to extend the fundamental algebra of numbers, $\mathbb{R}$, to a noncommutative algebra $A$, just called {\em quantum algebra}. The general request on such a type of algebra can be obtained on the ground of the mathematical logic. (See \cite{PRA0500, PRA3}.) In fact, we have shown that the meaning of quantization of a classical theory, encoded by a PDE $E_k$, in the category of commutative manifolds, is a representation of the logic $\mathcal{L}(E_k)$ of the classic theory, into a quantum logic $\mathcal{L}_q$. More precisely, $\mathcal{L}(E_k)$ is the Boolean algebra of subsets of the set $\Omega(E_k)_c$ of solutions of $E_k$: $\mathcal{L}(E_k)\equiv \mathcal{P}(\Omega(E_k)_c)$. (The infinite dimensional manifold $\Omega(E_k)_c$ is called also the {\em classic limit of the quantum situs} of $E_k$.) Furthermore, $\mathcal{L}_q$ is an algebra $A$ of (self-adjoint) operators on a locally convex (or Hilbert) space $\mathcal{H}$: $\mathcal{L}_q\equiv A\subset L(\mathcal{H})$. Then to quantize a PDE $E_k$, means to define a map $\mathcal{L}(E_k)\to \mathcal{L}_q$, or an homomorphism of Boolean algebras $q:\mathcal{P}(\Omega(E_k)_c)\to\mathcal{P}_r(\mathcal{H})$, where $\mathcal{P}_r(\mathcal{H})$ is a Boolean algebra of projections on $\mathcal{H}$. (For details see \cite{PRA0500, PRA3}.\footnote{Compare with the meaning of quantum logic given by G. Birkoff and J. Von Neumann \cite{BIRKHOFF-NEUMANN}.}) This construction allowed us also to prove that a quantization of a classical theory, can be identified by a functor relating the category of differential equations for commutative (super)manifolds, with the category of quantum (super) PDE's. (See Refs. \cite{PRA0-100, PRA12, PRA13, PRA14, PRA15, PRA16, PRA18-0, PRA18}.)

We can also extend a quantum algebra $A$, when the particular mathematical (or physical) problem requires it useful. Then the extended algebra does not necessitate to be associative. This is, for example, the case when the extension is made by means of some algebra in the Cayley-Dickson construction, obtaining a {\em quantum-Cayley-Dickson construction}:
$$\xymatrix{\mathcal{Q}_0\ar@{^{(}->}[r]&\mathcal{Q}_1\ar@{^{(}->}[r]&\mathcal{Q}_2\ar@{^{(}->}[r]&
\cdots\ar@{^{(}->}[r]&\mathcal{Q}_r\ar@{^{(}->}[r]&\cdots}$$
where $\mathcal{Q}_r\equiv A\bigotimes_{\mathbb{R}}\mathbb{A}_r$, hence $\mathcal{Q}_0=A$, $\mathcal{Q}_1=A\bigotimes_{\mathbb{R}}\mathbb{C}$ and $\mathcal{Q}_2=A\bigotimes_{\mathbb{R}}\mathbb{H}$, etc.
The main of this paper is just to show that the Pr\'astaro's algebraic topologic theory of quantum PDE's, formulated starting from $80s$, directly applies to these non-associative quantum algebras arising in the above quantum-Cayley-Dickson construction.

Finally, the last purpose of this paper is to extend the concept of {\em exotic PDE's}, recently introduced by A. Pr\'astaro, for PDE's in the category of commutative manifolds \cite{PRA19, PRA20, PRA21, PRA21-1}, also for the ones in the category of quantum PDE's. In particular, we prove also smooth versions of quantum generalized Poincar\'e conjectures.

In the following we list the main results of this paper, assembled for sections. 2. Theorem \ref{integral-spectrum-quantum-pdes} and Theorem \ref{heyting-algebra-quantum-pde} show that to any quantum hypercomplex PDE $\hat E_k\subset\hat J^k_n(W)$, can be associated a topological spectrum ({\em integral spectrum}) and a Heyting algebra ({\em integral Heyting algebra}) encoding some algebraic topologic properties of such a PDE. This allows us to give a new constructive point of view to the actual approach to consider quantum logic in field theory by means of topoi.
3. Theorem \ref{delta-poincare-lemma-pdes-category-quantum-hypercomplex-manifolds} and Theorem \ref{criterion-formal-quantum-integrability-pdes-category-quantum-hypercomplex-manifolds} extend our formal geometric theory of PDE's from the category of quantum (super)manifolds to the ones for quantum hypercomplex manifolds. Theorem \ref{integral-bordism-groups-quantum-hypercomplex-pdes} and Theorem \ref{theorem-quantum-sedenionic-heat-equation} characterize global solutions of PDE's in the category $\mathfrak{Q}_{hyper}$, by means of suitable bordism groups. 4. Theorem \ref{main-quantum-singular1} gives an algebraic topologic characterization of singular PDE's in the category $\mathfrak{Q}_{hyper}$.  5. Here we extend to the category $\mathfrak{Q}_{hyper}$ the concept of exotic PDE's, perviously introduced by A. Pr\'astaro for PDE's in the category of commutative manifolds. Theorem \ref{bordism-groups-in-quantum-hypercomplex-exotic-pdes-stability} and Theorem \ref{integral-h-cobordism-in-Ricci-flow-pdes} give characterizations of global solutions for exotic PDE's in the category $\mathfrak{Q}_{hyper}$, that allow to classify smooth solutions starting from quantum homotopy spheres. In particular, an integral h-cobordism theorem in quantum Ricci flow PDE's is proved.

\section{Spectra in quantum PDE's}\label{spectra-quantum-pdes-section}

In this section we prove that to any quantum PDE can be canonically associated a Heyting algebra directly related to its integral bordism group. This is made since an algebraic topologic spectrum is recognized to characterize such an integral bordism group, and topologic spectra can be encoded in Heyting algebras. This result allows us to look to the meaning of quantum logic from a more general point of view, since Heyting algebras contain Boolean algebras. In order to make this paper as self-contained as possible, let us first recall some fundamental definitions and results about Heyting algebras. (There propositions are generally given without proof since they are standard. See, e.g., Refs \cite{BIRKHOFF, RUTHERFORD} for more informations about.)

\begin{definition} {\em 1)} A {\em partially ordered system}\index{partially ordered system ! order relation ! totally ordered subset} $(E,\le  )$ is a
non-empty set $E$, together with a relation $\le$   on $E $, such
that: {\em(a)} if $a\le  b$ and $b\le  c \Rightarrow  a\le  c $; {\em(b)}
$a\le a $. The relation $\le$   is called an {\em order
relation} in $E$. The notation $y\ge x$ is sometimes used in
place of $\le$.

{\em 2)} A {\em totally ordered subset} $F$ of partially ordered
system  $(E,\le  )$ is a subset of $E$ such that for every pair
$x,y\in F$ either $x\le  y$ or $y\le  x $.

{\em 3)} If $F$ is a subset of a partially ordered system  $(E,\le  )$
then an element $x$ in $E$ is said to be an {\em upper bound}
for $F$ if every $f$ in $F$ has the property $f\le  x $. An upper
bound for $F$ is said to be a {\em least upper bound} of $F$
($\sup F$) if every upper bound $y$ of $F$ has the property $x\le
y$.

{\em 4)} In a similar fashion, the terms {\em lower bound} and {\em greatest lower bound} (in $F$) may be defined.

{\em 5)} An element $x\in E$ is said to be {\em maximal} if $x\le y$
implies $y\le  x $.
\end{definition}

\begin{example} Let $X$ be a set and $\mathcal{P}(X)$ be the set of all subsets
of $X $. The couple $(\mathcal{P}(X),\subseteq )$ is a partially
ordered system where the order relation is the inclusion relation
$\subseteq$  between the sets contained in $X$. An upper bound for
a subfamily $B\subseteq  \mathcal{P}(X)$ is any set containing $\cup
B $, and $\cup B$ is the only least upper bound of $B$. Similarly
$\cap B$ is the only greatest lower bound of $B$. The only maximal
element of $\mathcal{P}(X)$ is $X$.\end{example}

\begin{definition} Let $(E,\le  )$ be a non-void partially ordered system. A
subset $B\subset E$ with the following three properties will be
called {\em admissible} with respect to a function
$f:E\rightarrow E$ such that $f(x)\ge x;$ and with respect to an
element $a\in E$:

{\em(a)} $a\in B $;

{\em(b)} $f(B)\subseteq B$;

{\em(c)} Every least upper bound of a totally ordered subset
of $B$ is in $B$.\end{definition}

\begin{theorem} {\em 1) (Hausdorff maximality theorem).}\index{Hausdorff maximality theorem} Every partially ordered system contains a maximal
totally ordered subsystem.

{\em 2) (Zorn's lemma).}\index{Zorn lemma} A
partially ordered system has a maximal element if every totally
ordered subsystem has an upper bound.

{\em 3) (Zermelo well-ordering theorem).}\index{Zermelo well-ordering theorem} Every set $E$ may be {\em well-ordered} that is a partially ordered system  $(E,\le)$ such that:

{\em(i)} $a\le b$, $b\le  a \Rightarrow a=b $;

{\em(ii)} any non-void subset of $E$ contains a lower bound for
itself.\end{theorem}

\begin{example}  The set $\mathbb{N}$ of positive integers
in the usual order is a familiar example of well-ordered
system.\end{example}

\begin{definition} A partially ordered system  $(E,\le  )$ is said to be {\em complete} if:

{\em(i)} $a\le  b$ and $b\le  a \Rightarrow  a=b $;

{\em(ii)} Every non-void subset has a least upper bound and a greatest
lower bound.
\end{definition}

\begin{proposition}[Tarski]\index{Tarski} If $(E,\le  )$ is a complete partially
ordered system , $f:E\rightarrow E$, and $x\le  y$ implies
$f(x)\le  f(y) $, then $f$ has a fixed element $x_{o , }$, i.e.,
$f(x_0)=x_0 $, and the set of all fixed elements contains its
least upper bound and its greatest lower bound.\end{proposition}

\begin{definition} An
element $x$ in a ring $R$ is said to be {\em idempotent}
 if $x^2=x $, and to be {\em nilpotent} if $x^n=0$
for some positive integer $n$. A {\em Boolean ring}\index{Boolean ring} is one in
which every element is idempotent.\end{definition}

\begin{proposition} {\em 1)} Every Boolean
ring is a commutative ring where the identity $x+x=0$ (or
equivalently $x=-x$) holds.

{\em 2)} In a Boolean ring, any
prime ideal is maximal.\end{proposition}

\begin{example} The smallest Boolean ring with
unit is ${\mathbb{Z}}_2\equiv \{ 0,1\} $, i.e., the set of integers
modulo $2$. ${\mathbb{Z}}_2$ is actually a field. Conversely every
Boolean ring with unit which is also a field is necessarily
isomorphic to the field ${\mathbb{Z}}_2$.\footnote{Let us recall that $\mathbb{Z}_n$, has nonempty set $Z_{ero}(\mathbb{Z}_n)$ of zero divisors iff $n$ is composite, i.e., it is of the type $n=pq$. If $n$ is prime $Z_{ero}(\mathbb{Z}_n)=\varnothing$. In this last case $\mathbb{Z}_n$ is a field.} If a Boolean ring $A$ does
not contain zero divisor it is reduced to 0 or it is isomorphic to
${\mathbb{Z}}_2 $. The ring $\mathcal{P}(X)$ of subsets of a fixed set $X$
is a Boolean ring with respect to the following definitions of
multiplication and addition of arbitrary subsets of $X$: {\em (a) multiplication}: $EF=E\cap F, E,F\in \mathcal{P}(X) ) $;  {\em (b) addition}: $E+F = (E\cap CF)\cup (CE\cap F) = (E-F)\cup (F-E)
\equiv  E\triangle F$ ( $\equiv$ {\em symmetric difference}).
$X$ is the unity element and $\varnothing$ is the zero element.
The Boolean ring $\mathcal{P}(X)$ is isomorphic to the ring
${{\mathbb{Z}}_2}^X$ of the applications $X\to{\mathbb{Z}}_2$:
$h : \mathcal{P}(X) \rightarrow  {{\mathbb{Z}}_2}^X$,
 $h : A \mapsto  h(A)\equiv \chi _{A}$ where $\chi_A$ is the characteristic function of $A$.
Every Boolean ring is isomorphic to a sub-ring of the ring
${{\mathbb{Z}}_2}^X $. Every finite Boolean ring is of the form
${{\mathbb{Z}}_2}^n $.\end{example}

\begin{definition} A topological space is said to be
{\em totally disconnected} if its topology has a base
consisting of sets which are simultaneously open and closed.\end{definition}

\begin{example} A set $X$ with the discrete topology is a totally disconnected
topological space.\end{example}

\begin{theorem}[Stone representation theorem] Every Boolean ring with unit is isomorphic to the
Boolean ring of all open and closed subsets of a totally
disconnected compact Hausdorff space.\end{theorem}

\begin{definition} A {\em lattice}\index{lattice}
is a partially ordered set $L$ such that every pair $x , y$ of
elements in $L$ has a least upper bound and a greatest lower
bound, denoted by $x\vee y$ and $x\wedge y$ respectively. The
lattice $L$ has a {\em unit} if there exists an element $1$
such that $x\le  1 $, $x\in L $. The lattice $L$ has a {\em zero} if there exists an element 0 such that $0\le  x $, $x\in
L $. The lattice is called {\em distributive lattice} if
$x\wedge (y\vee z)=(x\wedge y)\vee (x\wedge z) $, $x,y,z\in L $.
The lattice is called {\em complemented lattice} if for every
$x\in L $, there exists an $x' \in L$ such that $x\vee x' =1 $,
$x\wedge x' =0 $. A lattice $L$ is said to be {\em complete
lattice} if every subset bounded above has a least upper bound,
or equivalently, every subset bounded below has a greatest lower
bound. A lattice $L$ is said to be {\em $\sigma$-complete
lattice} if this is true for denumerable subsets of $L$.\end{definition}

\begin{definition} A {\em Boolean algebra}\index{Boolean algebra} is a lattice with unit and zero
which is distributive and complemented.\end{definition}

\begin{proposition} The concepts of
Boolean algebra and Boolean ring with unit are equivalent.\end{proposition}

\begin{proof} In fact let $B$ be a Boolean algebra and define
multiplication and addition as: $xy= x\wedge y $, $x+y = (x\wedge
y' )\vee (x' \wedge y) $. Then, it may be verified that $B$ is a
Boolean ring with $1$ as unit. On the other hand, if $B$ is a
Boolean ring with unit denoted by $1$, then if $x\le  y$ is
defined to mean $x=\lambda y$ and $x' =1+x $, then $B$ is a
Boolean algebra and $x\vee y =x+y $, $x\wedge y = xy $. \end{proof}

\begin{definition} If $B$ and $C$ are Boolean algebras and $h:B\rightarrow C $,
then $h$ is said to be a {\em Boolean algebra homomorphism},
if $h(x\wedge y) = h(x)\wedge h(y) $, $h(x\vee y) = h(x)\vee h(y)
$, $h(x' ) =h(x)'  $. If $h$ is one-to-one, it is called an {\em isomorphism Boolean algebra}. If $h$ is an isomorphism and
$h(B)=C $, then we say that $B$ is {\em isomorphic} to $C $.\end{definition}

\begin{example} The following are examples of Boolean algebras:

{\em 1)} ${\mathbb{Z}}_2\equiv \{ 0,1\}$.

{\em 2)} $\mathcal{P}(X) $, where $\le$   is
taken to be the inclusion, and $\wedge$ and $\vee$  are taken as
intersection and union respectively.\end{example}

\begin{theorem} Every Boolean algebra is isomorphic with the Boolean algebra of
all open and closed subsets of a totally disconnected compact
Hausdorff space.\end{theorem}

\begin{definition} {\em 1)} A {\em projection} in a vector space $V$ means
a linear operator $E\in L(V)$ with $E^2=E $.

{\em 2)} The {\em intersection} $A\wedge B$ and the {\em union}
 $A\vee B$ of two commutating projections $A$ and $B$ in $V$ are
projections:

{\em(a) (intersection)} $A\wedge B=A\circ B$;

{\em(b) (union)} $A\vee B=A+B-(A\circ B) $.

The ranges of the intersection and union of two commutating projections are given by:
 $(A\wedge B)(V)=A(V)\cap B(V) $; $(A\vee B)(V)=A(V)\oplus B(V) =
S\bar p(A(V),B(V))$, where $S\bar p(A(V),B(V))$ means the closed
linear manifold spanned by the sets $A(V)$ and $B(V)$.

{\em 3)} The {\em natural ordering} $A\le  B$ between two commutating
projections $A$ and $B$ has the geometrical significance that
$A\le  B$ is equivalent to $A(V)\subseteq B(V)$.

{\em 4)} A {\em Boolean algebra of projections}\index{Boolean algebra of projections} in $V$ is a set of projections
in $V$ which is a Boolean algebra under the operations $A\vee B$
and $A\wedge B$ which has for its zero and unit elements the
operators 0 and $id_{V}\equiv 1\in L(V)$.

{\em 5)} The notions of
abstract $\sigma$-completeness, and completeness can be extended.
\end{definition}

Boolean algebras can be considered particular cases of more general algebras (Heyting algebras) that play important roles in Algebraic Topology.
\begin{definition}\label{heyting-algebra}

A {\em Heyting algebra} $H$ is a bounded lattice such that for all $a,b\in H$ there is a greatest element $x\in H$ such that $a\wedge x\le b$. This element $x$ is called the {\em relative pseudo-complement} of $a$ with respect to $b$, and denoted $x\equiv a\to b$. The largest (resp. smallest) element in $H$ is denoted by $1$ (resp. $0$).

$\bullet$\hskip 3pt The {\em pseudo-complement} of $x\in H$ is the element $\neg x=x\to 0$. \footnote{One has $a\wedge \neg a=0$, and $\neg a$ is the largest element having this property. For a Boolean algebra, $\neg a$ is a true complement, but for a Heyting algebra this is not generally assured. In fact, in general $a\vee\neg a\not= 1$.}

$\bullet$\hskip 3pt A {\em complete Heyting algebra} is a Heyting algebra that is a complete lattice.\footnote{Complete Heyting algebras are also called {\em frames}, or {\em locales}, or {\em complete Brouwerian lattices}. In a frame the meet distributes over infinite joins: $a\, \wedge\, \bigvee b_i=\bigvee(a\wedge\, b_i)$.}

$\bullet$\hskip 3pt A {\em subalgebra} of a Heyting algebra  $H$ is a subset $K\subset H$, containing $0$ and $1$ and closed under the operations $\wedge$, $\vee$ and $\to$.\footnote{Then it follows that $K$ is closed also under $\neg$, hence $K$ is necessarily a Heyting algebra under the same operations of $H$.}

$\bullet$\hskip 3pt If there exists an element $a\in H$, such that, $\neg a=a$, i.e., negation has a fixed point, then $H=\{a\}$.
\end{definition}

\begin{proposition}
The following propositions are equivalent.

{\em(i)} $H$ is a Heyting algebra.

{\em(ii) (Lattice theoretic definition-a).}  $H$ is a bounded lattice and the mappings $f_a:H\to H$, $f_a(x)=a\wedge x$, $a\in H$, are the lower adjoint of a monotone Galois connection.\footnote{Then, the relative upper adjoints $g_a$ are given by $g_a(x)=a\to x$.}

{\em(iii) (Lattice theoretic definition-b).}  $H$ is a residual lattice whose monoid operation is $\wedge$.\footnote{The monoid unit must be top element $1$. Commutativity of this monoid implies that the two residuals coincide as $a\to b$.}

{\em(iv) (Bounded lattice with an implication).} Let $H$ be a bounded lattice with largest and smallest elements $1$ and $0$ respectively, and a binary operation $\to$. such that the following hold:

{\em 1.} $a\to a=1$.

{\em 2.} $a\wedge (a\to b)=a\wedge b$.

{\em 3.} $b\wedge (a\to b)=a\wedge b$.

{\em 4.} $a\to(b\wedge c)=(a\to b)\wedge(a\to c)$.

{\em(v) (Axioms of intuitionistic propositional logic).}  $H$ is a set with three binary operations $\to$, $\wedge$ and $\vee$, and two distinguished elements $0$ and $1$, such that the following hold for any $x, y, z\in H$.\footnote{The relation $\le$ is defined by the condition $a\le b$ when $a\to b=1$. Furthermore, $\neg x=x\to 0$.}

{\em 1.} If $x\to y=1$ and $y\to x=1$ then $x=y$.

{\em 2.} If $1\to y=1$ then $y=1$.

{\em 3.} $x\to (y\to x)=1$.

{\em 4.} $(x\to (y\to z))\to((x\to y)\to(x\to z))=1$.

{\em 5.} $x\wedge y\to x=1$.

{\em 6.} $x\wedge y\to y=1$.

{\em 7.} $x\to(y\to(x\wedge y))=1$.

{\em 8.} $x\to x\vee y=1$.

{\em 9.} $y\to x\vee y=1$.

{\em 10.} $(x\to z)\to((y\to z)\to(x\vee y\to z))=1$.

{\em 11.} $0\to x=1$.
\end{proposition}

\begin{example}
Every Boolean algebra is a Heyting algebra with $p\to q\equiv \neg p\vee q$.
\end{example}

\begin{example}
Every totally ordered set that is a bounded lattice is also a Heyting algebra with $\to$ defined in {\em(\ref{arrow-definition-totally-ordered-set})}.
\begin{equation}\label{arrow-definition-totally-ordered-set}
    p\to q\equiv\left\{
\begin{array}{l}
  q\, ,\, p>q\\
  1\, ,\, \hbox{\rm otherwise.}
\end{array}\right.
\end{equation}

For example the set $H\equiv\{0,\frac{1}{2},1\}$, with $\to$ defined in {\em(\ref{arrow-definition-totally-ordered-set})} is a Heyting algebra that is not a Boolean algebra. It is important to emphasize that in this Heyting algebra does not hold the law of excluded middle. In fact one has: $\frac{1}{2}\vee\neg\frac{1}{2}=\frac{1}{2}\vee(\frac{1}{2}\to 0)=\frac{1}{2}\vee 0=\frac{1}{2}$.
\end{example}

\begin{example}[Topological Heyting algebra]\label{topological-heyting-algebra}
Let $(X,\mathcal{T})$ be a topological space, where $\mathcal{T}$ is its sets of open sets. $\mathcal{T}$ has a natural structure of lattice that is a complete Heyting algebra with binary operations given in {\em(\ref{binary-operations-topological-heyting-algebra})}.\footnote{Complete Heyting algebras are seen in categorical topology as generalized topological spaces.}

\begin{equation}\label{binary-operations-topological-heyting-algebra}
\left\{\begin{array}{l}
 A\to B\equiv(\complement A\bigcup B)^{\circ}=\complement(\overline{A\setminus B})\\
  A\wedge B=A\cap B\\
 A\vee B=A\cup B \\
\end{array}\right\}\hskip 3pt A\le B\,  \Leftrightarrow\, A\subseteq B.
\end{equation}
\end{example}

\begin{example}[Topological Heyting algebras induced from partially ordered sets]
Let $(X,\le)$ be a partially ordered set. Then the increasing sets (resp. decreasing sets) form a topology $\mathcal{T}^{\uparrow}$ (resp. $\mathcal{T}^{\downarrow}$) on $X$, hence are identified topological Heyting algebras, say $H^{\uparrow}$ (resp. $H^{\downarrow}$). In such algebras infinite distributive laws hold.
\end{example}

\begin{proposition}[Properties of Heyting algebra]

{\em 1) (Provable identities).} To prove true a formula $F(A_1,\cdots,A_n)$ of the intuitionistic propositional calculus, by means $\wedge$, $\vee$, $\neg$, $\to$ and the constants $0$ and $1$, is equivalent to state $F(a_1,\cdots,a_n)=1$ for any $a_1,\cdots,a_n\in H$, where $H$ is a Heyting algebra generated by $n$ variables.

{\em 2) (Distributivity).} Heyting algebras are always distributive, i.e., relations {\em(\ref{distributivity-relations-heyting-algebras})} hold.

\begin{equation}\label{distributivity-relations-heyting-algebras}
    \left\{\begin{array}{l}
     a\wedge(b\vee c)=(a\wedge b)\vee(a\wedge c)\\
     a\vee(b\wedge c)=(a\vee b)\wedge(a\vee c).\\
   \end{array}
    \right.
\end{equation}

$\bullet$\hskip 3pt Furthermore in complete Heyting algebras one has relation given in {\em(\ref{distributivity-relations-complete-heyting-algebras})}
\begin{equation}\label{distributivity-relations-complete-heyting-algebras}
x\wedge(\bigwedge Y)=\bigwedge\{ x\wedge y\, :\, y\in Y\}\, \forall x\in H,\, Y\subset H.
\end{equation}
$\bullet$\hskip 3pt Vice versa, any complete lattice satisfying the infinite distributive law {\em(\ref{distributivity-relations-complete-heyting-algebras})} is a complete Heyting algebra, with relative pseudo-complement operation defined by $a\to b\equiv \bigwedge\{ c\, |\, a\wedge c\le b\}$.
\end{proposition}

\begin{definition}
Two {\em complements} elements $x,y\in H$ of a Heyting algebra $H$ are characterized by the following  conditions: $x\wedge y=0$ and $x\vee y=1$. If $x$ admits a complement, we say that it is {\em complemented}.
\end{definition}

\begin{proposition}[Properties of complement elements in Heyting algebra]
{\em 1)} If there exists a complement $y$ of $x\in H$, it is unique and one has $y=\neg x$.

{\em 2)} If $x\in H$ is complemented, then so is $\neg x$, and both are to each other complement.\footnote{Even if $x$ is not complemented, $\neg x$ can be complemented, with complement different from $x$.}

{\em 3)} In any Heyting algebra, $0$ and $1$ are complements to each other.

\end{proposition}

\begin{definition}
A {\em regular} element $x$ of a Heyting algebra $H$ is characterized by either of the following equivalent conditions.\footnote{The equivalence of conditions (i) and (ii) follows from the fact that $\neg\neg\neg x=\neg x$.}

{\em (i)} $x=\neg\neg x$.

{\em (ii)} $x=\neg y$ for some $y\in H$.
\end{definition}

\begin{proposition}[Properties of regular elements in Heyting algebra]
{\em 1)} Any complemented element of a Heyting algebra is regular. (The converse is not true in general.) Therefore, $0$ and $1$ are regular elements.

{\em 2)} For any Heyting algebra $H$ the following conditions are equivalent.

{\em (i)} $H$ is a Boolean algebra.

{\em (ii)} Every $x\in H$ is regular.

{\em (iii)} Every $x\in H$ is complemented.

{\em 3)} Any Heyting algebra, $H$, contains two Boolean algebras $H_{reg}$ and $H_{comp}$ made respectively by the regular and complemented elements. In both Boolean algebras the operations $\wedge$, $\neg$ and $\to$ are the same than in $H$. For $H_{comp}$ also the operation $\vee$ is the same than in $H$, so that this Boolean algebra is a sub-algebra of $H$. Instead, for $H_{reg}$ we get the following different ''or'' operation: $x\vee_{reg}y=\neg(\vee x\wedge\neg y)$.

{\em 4) (The De Morgan laws in a Heyting algebra).} In any Heyting algebra one has:
\begin{equation}\label{de-morgan-laws-heyting-algebra}
    \left\{
   \begin{array}{ll}
 \hbox{\rm(regular De Morgan law):}&    \neg (x\vee y)=\neg x\wedge \neg y\\
     \hbox{\rm(weak De Morgan law):}&    \neg (x\wedge y)=\neg\neg(\neg x\vee \neg y).\\
   \end{array}
    \right\}_{\forall x,y\in G}.
\end{equation}
{\em 5)} The following propositions are equivalent for all Heyting algebras $H$.

{\em 1.} $H$ satisfies both De Morgan laws.

{\em 2.} $\neg(x\wedge y)=\neg x\vee \neg y$, $\forall x,y\in H$.

{\em 3.} $\neg(x\wedge y)=\neg x\vee \neg y$, for all regular $x,y\in H$.

{\em 4.} $\neg\neg(x\vee y)=\neg\neg x\vee\neg\neg y$, $\forall x,y\in H$.

{\em 5.} $\neg\neg(x\vee y)=x\vee y$, for all regular $x,y\in H$.

{\em 6.} $\neg(\neg x\wedge \neg y)=x\vee y$, for all regular $x,y\in H$.

{\em 7.} $\neg x\vee\neg\neg x=1$, $\forall x\in H$.
\end{proposition}

\begin{definition}
A {\em morphism} $f:H_1\to H_2$ between Heyting algebras is characterized by the following equivalent relations.

{\em (i)} $f(0)=0$.

{\em (ii)} $f(1)=1$.

{\em (iii)} $f(x\wedge y)=f(x)\wedge f(y)$.

{\em (iv)} $f(x\vee y)=f(x)\vee f(y)$.

{\em (v)} $f(x\to y)=f(x)\to f(y)$.

{\em (vi)} $f(\neg x)=\neg f(x)$.
\end{definition}

\begin{proposition}[Properties of Heyting morphisms]

{\em 1.} Let $H_1$ and $H_2$ be structure with operations $\to$, $\wedge$, $\vee$ (and possibly $\neg$) and constants $0$ and $1$. Let $f:H_1\to H_2$ be a surjective mapping satisfying properties {\em(i)--(vi)} in above definition. Then if $H_1$ is a Heyting algebra so is $H_2$ too.

{\em 2.} Heyting algebras form a category $\mathfrak{C}heyting$.\footnote{Complete Heyting algebras form three different categories having all the same objects, but having different morphism. In Tab. \ref{complete-heyting-algebra-categories} are resumed these categories. The category $\mathfrak{T}op$ of topologic spaces admits a representation in the category $\mathfrak{L}ocales$ of locales. However, many important theorems in point-set topology require axiom of choice, that has not an analogue in $\mathfrak{L}ocales$. Therefore, not all propositions in $\mathfrak{T}op$ can be translated in $\mathfrak{L}ocales$.}

\begin{table}[t]\centering
\caption{Complete Heyting algebras categories.}
\label{complete-heyting-algebra-categories}
\scalebox{0.8}{$\begin{tabular}{|l|l|l|}
\hline
{\rm{\footnotesize Symbol}}&{\rm{\footnotesize Definition}}&{\rm{\footnotesize Note}}\\
\hline
{\rm{\footnotesize $\mathfrak{C}heyting$}}&{\rm{\footnotesize $Ob(\mathfrak{C}heyting)=\{\hbox{complete Heyting algebras}\}$}}&\\
&{\rm{\footnotesize $Hom(\mathfrak{C}heyting)=\{\hbox{homomorphism of complete Heyting algebras}\}$}}&{\rm{\footnotesize }}\\
\hline
{\rm{\footnotesize $\mathfrak{F}rames$}}&{\rm{\footnotesize $Ob(\mathfrak{F}rames)=\{\hbox{\footnotesize lattices $L$, where every subset $(a_i)\subset L$ has a supremum $\bigvee a_i$}$}}&\\
&{\rm{\footnotesize such that $b\wedge(\bigvee a_i)=\bigvee (a_i\wedge b)\}$.}}&\\
&{\rm{\footnotesize $ Hom(\mathfrak{F}rames)=\{\hbox{\footnotesize lattices homomorphisms respecting arbitrary suprema}\}$}}&\\
\hline
{\rm{\footnotesize $ \mathfrak{L}ocales$}}&{\rm{\footnotesize $\mathfrak{L}ocales=(\mathfrak{F}rames)^{(op)}$}}&{\rm{\footnotesize $ \hbox{\footnotesize There exists}$}}\\
&&{\rm{\footnotesize $ \hbox{\footnotesize a natural functor}$}}\\
&&{\rm{\footnotesize $\mathfrak{T}op\to \mathfrak{L}ocales$}}\\
\hline
\multicolumn{3}{l}{\rm{\footnotesize $\mathfrak{T}op$ denotes the category of topological spaces. The axiom of choice is required for many important theorems}}\\
\multicolumn{3}{l}{\rm{\footnotesize  in point-set topology, but it does not exist for locales.}}\\
\end{tabular}$}
\end{table}

{\em 3.} Let $K\subset H$ be a subalgebra of a Heyting algebra $H$. Then the inclusion $i:K\hookrightarrow H$ is a morphism.

{\em 4.} One has a canonical morphism $H\to H_{reg}$, given by $x\mapsto \neg\neg x$.\footnote{Note that the composition $H\to H_{reg}\hookrightarrow H$ is not in general a morphism, since the join operation of $H_{reg}$ can be different from that of $H$.}
\end{proposition}

\begin{definition}
A {\em filter} on a Heyting algebra $H$ is a subset $F\subseteq H$ such tha the following conditions hold.

{\em(i)} $1\in F$.

{\em(ii)} $x,y\in F$ then $x\wedge y\in F$.

{\em(iii)} $x\in F$, $y\in H$, $x\leq y$ then $y\in F$.
\end{definition}

\begin{proposition}[Properties of filter on Heyting algebra]

{\em 1)} The intersection of filters on a Heyting algebra $H$ is again a filter.

{\em 2)} We call {\em filter generated by a subset} $S\subseteq H$ of a Heyting algebra $H$, the smallest filter $F$ on $H$ containing $S$.  If $S=\varnothing$ then $F=\{1\}$. If $S\not=\varnothing$ then $F=\{x\in H\, |\, y_1\wedge y_2\wedge\cdots\wedge y_n\leq x\, y_i\in S\}$.

{\em 3)} If $F$ is a filter on a Heyting algebra $H$, there is a Heyting algebra $H/F$, called {\em quotient} of $H$ by $F$, such that $p_F:H\to H/F$ is a morphism. More precisely $H/F=H/\thicksim$, where $\thicksim$ is the equivalence relation {\em(\ref{equivalence-relation-induced-filter-on-heyting-algebra})} in $H$ induced by $F$.
\begin{equation}\label{equivalence-relation-induced-filter-on-heyting-algebra}
    x\thicksim y \Leftrightarrow \left\{
    \begin{array}{l}
      x\to y\in F\\
      y\to x\in F.\\
    \end{array}
    \right.
\end{equation}
{\em 4) (Universal property).} Let $\subseteq H$ be a subset of a Heyting algebra $H$, and let $F$ the corresponding filter generated by $S$. Given any morphism $f:H\to H'$ of Heyting algebras, such that $f(y)=1$, $\forall y\in S$, there exists a unique morphism $f':H/F\to H'$, such that the diagram {\em(\ref{commutative-diagram-universal-relation-induced-subset-tfilter-on-heyting-algebra})} is commutative.
\begin{equation}\label{commutative-diagram-universal-relation-induced-subset-tfilter-on-heyting-algebra}
    \xymatrix{H\ar[d]_{p_F}\ar[r]^{f}&H'\\
    H/F\ar[ur]_{f'}&}
\end{equation}
{\em 5)} The {\em kernel} of a morphism $f:H_1\to H_2$ of Heyting algebras is the filter $\ker f\equiv f^{-1}(\{1\}) \subseteq H_1$. The morphism $f': H_1/(\ker f)\to H_2$, for universal property, is an isomorphism: $H_1/(\ker f)\cong f(H_1)\triangleleft H_2$.
\end{proposition}

\begin{theorem}[Heyting algebra of propositional formulas in $n$ variables up to intuitionistic equivalence]
To the set of propositional formulas in the variables $A_1,\dots,A_n$, one can canonically associate a Heyting algebra. {\em(Similar properties hold also for any set of variables $\{A_i\}_{i\in I}$ that are conditioned to some theory $T$.)}
\end{theorem}

\begin{proof}
Let us introduce in the set $L$ of propositional formulas in the variables $A_1,\dots,A_n$, a preorder $\leq$ defined by $F\leq G$ if $G$ is an (intuitionistic) local consequence of $F$. This preorder induces an equivalence relation $\thicksim$ in $L$: $ F\thicksim G \Leftrightarrow F\leq G,\,  G\leq F$. Then $H_0\equiv L/\thicksim$ is a Heyting algebra. Furthermore, the preorder 0n $L$ induces on $H_0$ an order relation $\leq$.
\end{proof}

\begin{theorem}[Spectra in Algebraic Topology and Heyting algebras]\label{spectra-algebraic-topology-heyting-algebras}
 Spectra in algebraic topology identify Heyting algebras.
\end{theorem}

\begin{proof}
We say that the spectrum $X$ is {\em acyclic} with respect to certain theory $E$ if $E\wedge X$ is contractible: $E\wedge X\simeq pt$. Two spectra are {\em Bousfield equivalent} if they have the same acyclic spectra:
$$E\thicksim F \Leftrightarrow \forall X:\, E\wedge X\simeq pt\, \Leftrightarrow F\wedge X\simeq pt. $$
Let us denote by $<E>$ the Bousfield class of the spectrum $E$. One has the following lemmas.

\begin{lemma}[\cite{OHKAWA}]
Bousfield classes form a set $B$ of cardinality at most $\beth_2>\mathfrak{c}$.\footnote{$\beth_2$ is the cardinality of all subsets $\mathcal{P}(\mathbb{R})$ of $\mathbb{R}$ and it is greater than the cardinality $\mathfrak{c}$ of the continuum, i.e., that of $\mathbb{R}$. (Recall that $\beth_i=2^{\beth_{i-1}}$, and $\beth_0=\aleph_0$.)}
\end{lemma}

\begin{lemma}[\cite{HOVEY-PALMIERI}]
One can define partial ordering in the set of Bousfield classes by
$$<E>\, \geq\, <F> \, \Leftrightarrow\, \forall X\, E\wedge X\simeq pt \, \Rightarrow\, F\wedge X\simeq\, pt.$$
The set of Bousfield classes form a complete lattice. The join is given by the wedge $\vee$. The smallest element is $<pt>$ and the largest element is $<E(S^0)>$. The meet $\curlywedge$ is given by the join of all lower bounds.\footnote{But, this meet does not distribute over infinite joins and the smash does. So it is interesting to consider a structure where the meet coincide with the smash.}  Let us denote by $DL\subset B$ the subset of the Bousfield lattice such that $<E>\wedge<E>\equiv<E\wedge E>=<E>$. $DL$ is a distributive lattice, ({\em distributive Bousfield lattice}), and it is just a complete Heyting algebra.\footnote{This is endowed with three operations: $wedge$, $\vee$ and $\Rightarrow$, where $a\Rightarrow b$ is the greatest $x$ such that $a\wedge x\le b$. Therefore $DL$ is a frame. (See Definition \ref{heyting-algebra} and \cite{BOUSFIELD}.)}

The inclusion $DL\hookrightarrow B$ preserves joins but does not preserve meets. Furthermore, there is a retraction $r:B\to DL$, defined by $$r<X>=\bigvee\{<Y>\in DL\, |\, <Y>\, \le\, <X>\}.$$ $r$ can be considered the right adjoint of the functor $i$, by considering any partially ordered set a category with a unique map from $x$ to $y$  iff $x\le y$.\footnote{A functor between complete lattices preserves colimits iff it preserves arbitray joins. For maps $f:A\to B$, $g:B\to A$, between partially ordered sets $A$ and $B$, $g$ is right adjoint to $f$ iff $f(x)\le y$ is equivalent to $x\le g(y)$.} $r$ preserves smash product: $r(<X>\, \wedge\, <Y>)=r<X>\, \wedge\, r<Y>$.
\end{lemma}

After the above two lemma the proof of the theorem is done.
\end{proof}

\begin{example}
All ring spectra and all finite spectra are in $DL$.
\end{example}

\begin{proposition}
The set of $<E>\in DL$ such that the pseudocomplement $<E>\to <pt>$ \footnote{That is the greatest $<F>$ such that $<E>\wedge<F>\leq pt$, i.e., $<E\wedge F>\wedge X\simeq\, pt\, \Rightarrow\, pt\wedge X\simeq\, pt$.} is really the complement is a Boolean algebra $BA\subset DL$.

Furthermore into $BA$ is contained a Boolean algebra $FBA$ isomorphic to the Boolean algebra of finite and co-finite subsets of $\mathbb{N}$; this is just the subalgebra of all {\em finite $p$-local spectra}.
\end{proposition}

\begin{theorem}[Integral spectrum of quantum PDEs]\label{integral-spectrum-quantum-pdes}

 {\em 1)} Let $\hat E_k\subset
\hat J^k_n(W)$ be a PDE in the category $\mathfrak{Q}$ of quantum manifolds \cite{PRA2, PRA3, PRA12, PRA13, PRA14, PRA15, PRA16}. Then there is a spectrum $\{\Xi_s\}$ {\em(singular integral spectrum of quantum PDEs)}, such that
$\Omega_{p,s}^{\hat E_k}=\mathop{\lim}\limits_{r\to\infty}\pi_{p+r}(\hat E^+_k\wedge
\Xi_r)$,
$\Omega^{p,s}_{\hat E_k}=\mathop{\lim}\limits_{r\to\infty}[S^r\hat E^+_k,\Xi_{p+r}]$,
$p\in\{0,1,\dots,n-1\}$.

{\em 2)} There exists a spectral
sequence $\{\hat E^r_{p,q}\}$, (resp. $\{\hat E_r^{p,q}\}$), with
$\hat E^2_{p,q}=H_p(\hat E_k,E_q(*))$, (resp. $\hat E_2^{p,q}=H^p(\hat E_k,E^q(*))$),
converging to $\Omega^{\hat E_k}_{\bullet,s}$, (resp.
$\Omega_{\hat E_k}^{\bullet,s}$). We call the spectral sequences
$\{\hat E^r_{p,q}\}$ and $\{\hat E_r^{p,q}\}$ the {\em integral singular
spectral sequences}\index{integral singular spectral sequences} of $\hat E_k$.
\end{theorem}

\begin{proof} The proof follows directly from natural extensions of analogous ones for PDE's in the category of commutative manifolds. See \cite{PRA1}.\end{proof}

From above results we are able to associate Heyting algebras to any PDE considered in the category of quantum PDE's. In fact we have the following.
\begin{theorem}[Heyting algebra of a quantum PDE]\label{heyting-algebra-quantum-pde}
Let $\hat E_k\subset\hat J^k_n(W)$ be a PDE in the category $\mathfrak{Q}$ of quantum manifolds. Then $\hat E_k$ identifies a Heyting algebra, $H(\hat E_k)$, that we call {\em spectral Heyting algebra} of $\hat E_k$.
\end{theorem}

\begin{proof}
After Theorem \ref{spectra-algebraic-topology-heyting-algebras} and Theorem \ref{integral-spectrum-quantum-pdes} we can conclude that to the spectrum $\Xi_s$ on can associate a Heyting algebra that is just $H(\hat E_k)$.
\end{proof}

\begin{remark}
Compare this new result with reinterpretations of quantum theory by using topoi. (See, e.g. Refs. \cite{BANASCHEWSKI-MULVEY, ISHAM-BUTTERFIELD, HEUNEN-LANDSMAN-SPITTERS}. Recall that subobjects of any object in a topos form a Heyting algebra. Furthermore, it is possible identify a topos for any non-abelian $C^*$-algebra $B$, as the topos of covariant functors over the category $\mathcal{C}$ of abelian subalgebras of $B$. $C$ induces an internal abelian $C^*$-algebra $C$ in this topos of functors. The Gel'fand spectrum of $C$ is the spectral presheaf.)\footnote{For general informations on topos theory see, e.g., the following \cite{BANASCHEWSKI-MULVEY, JOHNSTONE, MACLANE, MACLANE-MOERDIJK, STONE}.}
\end{remark}
\begin{remark}
Direct extensions of Theorem \ref{integral-spectrum-quantum-pdes} and Theorem \ref{heyting-algebra-quantum-pde} for PDE's, in the category $\mathfrak{Q}_S$ of quantum super manifolds, hold too.
\end{remark}

\section{Quantum hypercomplex PDE's}\label{quantum-hypercomplex-pdes-section}

Following our previous works devoted to quantum PDE's, we consider, now, the category of quantum hypercoplex manifolds $\mathfrak{Q}_{hyper}$ and we recaste our theory of quantum PDE's, considering PDE's in $\mathfrak{Q}_{hyper}$ emphasizing the geometric new mathematical structures and the characterizations that these generate. Then theorems for existence of local and global solutions are obtained for such quantum PDE's that extend analogous previous results for quantum PDE's.

For definitions and fundamental results on quantum manifolds see \cite{PRA12, PRA13, PRA14, PRA15, PRA16, PRA18-0, PRA18}. Let in the following, introduce the new definitions related to quantum hypercomplex manifolds.

\begin{definition}
A {\em quantum hypercomplex $r$-algebra}, $0\le r\in\mathbb{N}$, is the extension $\mathcal{Q}_r\equiv B\bigotimes_{\mathbb{R}}\mathbb{A}_r$, where $B$ is a quantum algebra in the sense of \cite{PRA12, PRA13, PRA14, PRA15, PRA16, PRA18-0, PRA18} and $\mathbb{A}_r$, is an $\mathbb{R}$-algebra in the Cayley-Dikson construction. (We assume that $B$ is a quantum $\mathbb{K}$-algebra, with $\mathbb{K}\equiv\mathbb{R}$ or $\mathbb{K}\equiv\mathbb{C}$.) We call any $q\in\mathcal{Q}_r$ a {\em quantum hypercomplex number}.\footnote{Let us recall that the {\em Cayley-Dikson algebra} $\mathbb{A}_r\cong \mathbb{R}^{2^r}$ is an $\mathbb{R}$-algebra structure on the $\mathbb{R}$-vector space $\mathbb{R}^{2^r}=\mathbb{R}^{2^{r-1}}\times\mathbb{R}^{2^{r-1}}$, given inductively by the formula $(a_1,a_2)(b_1,a_2)=(a_1b_1-\bar b_2a_2,b_2a_1+a_2\bar b_1)$, where $\bar a=(\bar a_1,-a_2)\in\mathbb{R}^{2^r}=\mathbb{R}^{2^{r-1}}\times\mathbb{R}^{2^{r-1}}$, with $\mathbb{A}_0=\mathbb{R}$ and $\mathbb{A}_1=\mathbb{C}$, $\mathbb{A}_2=\mathbb{H}$, $\mathbb{A}_3=\mathbb{O}$. The algebras $\mathbb{A}_r$, $0\le r\le 3$ are called the {\em classical Cayley-Dickson algebras}. For classical Cayley-Dickson algebras holds the {\em Hurewitz's theorem}: $||ab||=||a||\, ||b||$, $\forall a,b\in\mathbb{A}_r$, $r=0,1,2,3$, where $||.||$ denotes the euclidean norm in $\mathbb{R}^{2^r}$. For $r\ge 4$ does not necessitate the product preserves the norm. In fact, for $r\ge 4$ one has Cayley-Dikson algebras $\mathbb{A}_r$ with nonempty set $Z_{ero}(\mathbb{A}_r)$, of zero divisors.}
\end{definition}
\begin{proposition}[Properties of quantum hypercomplex algebras]
{\em 1)} $\mathcal{Q}_r$ is a quantum vector space of dimension $2^r$ with respect to the quantum algebra $B$. Therefore, $\mathcal{Q}_r$ is a metrizable, complete, Hausdorff, locally convex topological $\mathbb{K}$-vector space

{\em 2)} $\mathcal{Q}_r$ has a natural structure of $\mathbb{K}$-algebra.  Furthermore one has the following properties

{\em(i)} $\mathcal{Q}_r$ is also a ring with unit $e$;

{\em(ii)} $\epsilon:\mathbb{K}\to Z(\mathcal{Q}_r)\subset \mathcal{Q}_r$ is a ring homomorphism, where $Z\equiv Z(\mathcal{Q}_r)$ is the centre of $\mathcal{Q}_r$;

{\em(iii)} $c:\mathcal{Q}_r\to\mathbb{K}$ is a $\mathbb{K}$-linear morphism, with $c(e)=1$, $e=$unit of $\mathcal{Q}_r$. For any $a\in \mathcal{Q}_r$ we call $a_C\equiv c(a)\in\mathbb{K}$ the {\em classic limit} of $a$;

{\em(iv)} $\mathcal{Q}_r$ is a non-associative $\mathbb{K}$-algebra, for $r\ge 3$. We say that $\mathcal{Q}_r=B\bigotimes_{\mathbb{R}}\mathbb{A}_r$ is {\em$m$-associative} if $\mathbb{A}_r$ is $m$-associative, i.e., there exists an $m$-dimensional subspace $V\subset\mathbb{A}_r$ such that $(y.x).z=y.(x.z)$, for all $y,\, z\in\mathbb{A}_r$ and $x\in S$. This means that one has also $((a\otimes y).(b\otimes x)).(c\otimes z)=(a\otimes y).((b\otimes x).(c\otimes z))$, $\forall a,\, b\, c\, \in B$.

$\mathcal{Q}_3$ is an {\em alternative $\mathbb{K}$-algebra}, i.e., $a^2b=a(ab)$ and $ab^2=(ab)b$, $\forall a,\, b\in \mathcal{Q}_3$.

{\em(v)} $\mathcal{Q}_r=\mathbb{K}\bigotimes_{\mathbb{R}}\mathbb{A}_r$, is a {\em flexible $\mathbb{K}$-algebra}, for $r\ge 4$, i.e., $a(ba)=(ab)a$, $\forall a,\, b\in \mathcal{Q}_r$.

{\em(vi)} One has the following implications for the quantum hypercomplex algebras $\mathcal{Q}_r$: $Associative\, \Rightarrow\, Alternative\, \Rightarrow\, Flexible$, but the backwards implications are not true. The canonical product in $\mathcal{Q}_r$ is nonassociative for $r\ge 3$, since $\mathbb{A}_r$ is a nonassociative algebra for $r\ge 3$.

{\em 3)} One has the short exact sequences reported in {\em(\ref{split-short-exact-sequence-quantum-ypercomplex-algebra})}.
\begin{equation}\label{split-short-exact-sequence-quantum-ypercomplex-algebra}
\left\{ \begin{array}{l}
   \xymatrix{0\ar[r]&B\ar[r]^a&\mathcal{Q}_r\ar[r]^(.4)b&\mathcal{Q}_r/B\ar[r]&0}\\
 \xymatrix{0\ar[r]&\mathbb{A}_r\ar[r]^c&\mathcal{Q}_r\ar[r]^(.4)d&\mathcal{Q}_r/\mathbb{A}_r\ar[r]&0}.\\
 \end{array}\right.
\end{equation}
This means that any quantum hypercomplex algebra $\mathcal{Q}_r$ is an extension of a quantum algebra $B$ and a Cayley-Dickson algebra $\mathbb{A}_r$, where both can be considered subalgebras of $\mathcal{Q}_r$, and one has the canonical isomorphisms reported in {\em(\ref{quantum-hypercomplex-algebras-isomorphisms})}.
\begin{equation}\label{quantum-hypercomplex-algebras-isomorphisms}
   \mathcal{Q}_r\cong (B\bigotimes 1).(e_B\bigotimes\mathbb{A}_r)\cong(e_B\bigotimes\mathbb{A}_r).(B\bigotimes 1).
\end{equation}

{\em 4)} The {\em nucleus} $N_{ucleus}(\mathcal{Q}_r)$ of $\mathcal{Q}_r$, i.e., the set of elements in $\mathcal{Q}_r$ which associates with every pair of elements $a,\, b\in \mathcal{Q}_r$, is an associative subalgebra of $\mathcal{Q}_r$, containing the center $Z(\mathcal{Q}_r)$ that is a commutative associative subalgebra of $\mathcal{Q}_r$.
\end{proposition}

\begin{proof}
1) In fact for any basis $\{e_p\}_{1\le p\le 2r}$, one has an $\mathbb{R}$-isomorphism $\mathbb{A}_r\cong\mathbb{R}^{2^r}$. This induces an isomorphism $\mathcal{Q}_r\cong A^{2^r}$. This means that $\{1\otimes e_p\}_{1\le p\le 2r}$, $1\in A$, is a $A$-basis for $\mathcal{Q}_r$, hence any vector $q\in\mathcal{Q}_r$, i.e., any quantum hypercomplex $r$-number, admits the linear representation given in (\ref{linear-representation-quantum-hypercomplex-r-number}).
\begin{equation}\label{linear-representation-quantum-hypercomplex-r-number}
    q=\sum_{1\le p\le 2r}a^p(1\otimes e_p),\, a^p\in A.
\end{equation}
In fact, any $q\in\mathcal{Q}_r$ can be written in the form $q=\sum_{i\in I}a^i\otimes b_i$, with $a^i\in B$ and $b_i\in \mathbb{A}_r$. Representing each $b_i$ in a basis $\{e_p\}_{1\le p\le 2r}$ of $\mathbb{A}_r$, we get:

$$
\left\{
\begin{array}{ll}
  q&=\sum_{i\in I}a^i\otimes b_i \\
  \\
  & =\sum_{i\in I}a^i\otimes(\sum_{1\le p\le 2^r}b_i^p e_p),\, b_i^p\in\mathbb{R},\, e_p\in\mathbb{A}_r\\
  \\
  & =\sum_{i\in I; 1\le p\le 2^r}a^ib_i^p\otimes e_p\\
  \\
  & =\sum_{1\le p\le 2^r}c^p\otimes (1\otimes e_p),\, c^p\equiv \sum_{i\in I}a^ib_i^p\in B,\, 1\otimes e_p\in\mathcal{Q}_r.\\
\end{array}
\right.
$$

2) The product in $\mathcal{Q}_r$ is given by (\ref{product-in-quantum-hypercomplex-r-numbers}).
\begin{equation}\label{product-in-quantum-hypercomplex-r-numbers}
\left\{
\begin{array}{ll}
   a.b&=(\sum_{i\in I}a^i\otimes \bar a_i ).(\sum_{j\in J}b^j\otimes \bar b_j ),\, a^i,b^j\in B,\, \bar a_i,\bar b_j\in\mathbb{A}_r\\
   \\
   &=\sum_{(i,j)\in\in I\times J}a^ib^j\otimes\bar a_i\bar b_j\\
   \\
   &=\sum_{(i,j)\in\in I\times J}c^{ij}\otimes d_{ij},\, c^{ij}\equiv a^ib^j\in B,\, d_{ij}\equiv\bar a_i\bar b_j\in\mathbb{A}_r.\\
   \end{array}
   \right.
\end{equation}
By using the linear representations of $a,\, b\in\mathcal{Q}_r$, with respect to a $A$-basis $\{1\otimes e_p\}_{1\le p\le 2r}$, we get the expression (\ref{product-in-quantum-hypercomplex-r-numbers-by-a-components}).\footnote{The $2^{3r}$ numbers $\gamma^k_{pq}\equiv(e_pe_q)^k\in\mathbb{R}$, are called {\em multiplication constants} of $\mathcal{Q}_r$.}
\begin{equation}\label{product-in-quantum-hypercomplex-r-numbers-by-a-components}
\left\{
\begin{array}{ll}
   a.b&=(\sum_{1\le p\le 2^r}a^p(1\otimes e_p).(\sum_{1\le q\le 2^r}b^q(1\otimes e_q),\, a^p,\, b^q\in B\\
   \\
   &=\sum_{1\le p,q\le 2^r}a^pb^q(1\otimes e_pe_q)\\
   \\
   &=\sum_{1\le k\le 2^r}c^{k}(1\otimes e_{k})\\
   \\
   &c^{k}\equiv \sum_{1\le p,q\le 2^r}a^pb^q(e_pe_q)^k\in B,\, (e_pe_q)^k\in\mathbb{R}\\
   \\
   &e_pe_q=\sum_{1\le k\le 2^r}(e_pe_q)^ke_k\in\mathbb{A}_r.\\
    \end{array}
   \right.
\end{equation}
Thus this product is $\mathbb{R}$-bilinear and unital with
unit $e\in\mathcal{Q}_r$ given by $e=e_B\otimes 1$, where $e_B$ is the unity in $B$ and $1$ is the unity in $\mathbb{A}_r$. Then, the properties listed above, follow directly from analogous ones for the Cayley-Dickson algebras $\mathbb{A}_r$.

3) In fact the mappings $a$ and $c$ in (\ref{split-short-exact-sequence-quantum-ypercomplex-algebra}) are the canonical inclusions $b\mapsto b\otimes1$ and $a\mapsto e_B\otimes a$ respectively, for $b\in B$ and $a\in\mathbb{A}_r$. Furthermore, the mappings $b$ and $d$ in (\ref{split-short-exact-sequence-quantum-ypercomplex-algebra}) are the canonical projections.

4) These properties directly follow from definitions. Let us emphasize only that $Z(\mathcal{Q}_r)=\{a\in N_{ucleus}(\mathcal{Q}_r)\, |\, ab=ba,\, \forall b\in \mathcal{Q}_r\}$.
\end{proof}

\begin{definition}
A {\em quantum hypercomplex vector space} of dimension $(m_0,\dots,m_s)\in\mathbb{N}^s$, built on the quantum hypercomplex algebra $A\equiv \mathcal{Q}_0\times\dots\times \mathcal{Q}_s$, is a locally convex topological $\mathbb{K}$-vector space $E$ isomorphic to $\mathcal{Q}_0^{m_0}\times\dots \mathcal{Q}_s^{m_s}$.
\end{definition}

\begin{definition}
A {\em quantum hypercomplex manifold} of dimension $(m_0,\dots,m_s)$ over a quantum algebra $A\equiv \mathcal{Q}_0\times\dots\times \mathcal{Q}_s$ of class $Q^k_w$, $0\le k\le \infty,\omega$, is a locally convex manifold $M$ modelled on $E$ and with a $Q^k_w$-atlas of local coordinate mappings, i.e., the transaction functions $f:U\subset E\to U'\subset E$ define a pseudogroup of local $Q^k_w$-homeomorphisms on $E$, where $Q^k_w$ means $C^k_w$, i.e., weak differentiability \cite{KELLER, PRA0500, PRA2, PRA3}, and derivatives $Z$-linaires, with $Z\equiv Z(A)$ the centre of $A$. So for each open coordinate set $U\subset M$ we have a set of $m_0+\dots+m_s$ coordinate functions $x^A:U\to A$, ({\em quantum hypercomplex coordinates}).
\end{definition}

\begin{definition}
The {\em tangent space} $T_pM$ at $p\in M$, where $M$ is a quantum hypercomplex manifold of dimension $(m_0,\dots,m_s)$, over a quantum hypercomplex algebra $A\equiv \mathcal{Q}_0\times\dots\times \mathcal{Q}_s$ of class $Q^k_w$, $k>0$, is the vector space of the equivalence classes $v\equiv[f]$ of $C^1_w$ (or equivalently $C^1$) curves $f:I\to M$, $I\equiv$ open neighborhood of $0\in \mathbb{R}$, $f(0)=p$; two curves $f$, $f'$ are equivalent if for each (equivalently, for some) coordinate system $\mu$ around $p$ the functions $\mu\circ f$, $\mu\circ f':I\to \mathcal{Q}_0^{m_0}\times\dots\times \mathcal{Q}_s^{m_s}$ have the same derivative at $0\in\mathbb{R}$.
\end{definition}

\begin{remark}
Then, derived tangent spaces associated to a quantum hypercomplex manifold $M$ can be naturally defined similarly to what made for quantum manifolds. (For details see \cite{PRA0500, PRA2, PRA3, PRA12, PRA13, PRA14, PRA15, PRA16, PRA18-0, PRA18}.)
\end{remark}

\begin{definition}
We say that a quantum hypercomplex manifold of dimension $(m_0,\dots,m_s)$ is {\em classic regular} if it admits a projection $c:M\to M_C$ on a $n$-dimensional manifold $M_C$. We will call $M_C$ the {\em classic limit} of $M$ and in order to emphasize this structure we say that the dimension of $M$ is $(n\downarrow m_0,\dots,m_s)$.
\end{definition}

\begin{definition}
The category $\mathfrak{Q}_{hyper}$ of quantum hypercomplex manifolds, is made by $Ob(\mathfrak{Q}_{hyper})$ that contains all quantum hypercomplex manifolds, and morphisms $Hom(\mathfrak{Q}_{hyper})$ are mappings of class $Q^k_w$ between quantum hypercomplex manifolds. In order that $Hom_{\mathfrak{Q}_{hyper}}(M,N)$ should be non empty, it is necessary that, whether $M$, (resp. $N$), is modeled on the quantum hypercomplex algebra $A$, (resp. $A'$), and $A'$ should be also a $Z(A)$-module, where $Z\equiv Z(A)$ is the centre of $A$.
\end{definition}

\begin{remark}
Let $\pi:W\to M$ be a fiber bundle in the category $\mathfrak{Q}_{hyper}$, such that $M$ is of dimension $m$ on $A$ and $W$ of dimension $(m,s)$ on $B\equiv A\times E$, where $E$ is also a $Z(A)$-module.\footnote{In the following we shall consider fiber bundles of this type.} Then we can define the $k$-order jet-derivative space for sections of $\pi$, $J\hat{\it D}^k(W)$ as an object in $\mathfrak{Q}_{hyper}$, similarly to what made for jet-derivaive spaces in the category $\mathfrak{Q}$. So in the following we will formally resume the definition of PDE's in the category $\mathfrak{Q}_{hyper}$, by adapting the language for analogous geometric structures just previously considered in the category $\mathfrak{Q}$. (See Refs. \cite{PRA0500, PRA2, PRA3, PRA12, PRA13, PRA14, PRA15, PRA16, PRA18-0, PRA18}.)
\end{remark}

A {\em quantum PDE} (QPDE) of order $k$ on the fibre bundle $\pi:W\to M$, defined in the category of quantum hypercomplex manifolds, is a subset $\hat E_k\subset J\hat{\it D}^k(W)$ of the jet-quantum derivative space $J\hat{\it D}^k(W)$ over $M$. We can formally extend the geometric theory for quantum PDEs, previously considered in \cite{PRA0500, PRA2}, to PDE's in the category $\mathfrak{Q}_{hyper}$, since the intrinsic formulation therein is not influenced by the non-associativity of the underlying quantum hyoercomplex algebra.\footnote{This is, instead, important in the local writing and meaning of PDE's.} In the following we shall emphasize some important definitions and results about. A QPDE $\hat E_k$ is {\em quantum regular} if the $r$-quantum prolongations $\hat E_{k+r}\equiv J\hat{\it D}^r(\hat E_k)\cap J\hat{\it D}^{k+r}(W)$ are subbundles of $\pi_{k+r,k+r-1}:J\hat{\it D}^{k+r}(W)\to J\hat{\it D}^{k+r-1}(W)$, $\forall r\ge 0$. Furthermore, we say that $\hat E_k$ is {\em formally quantum integrable} if $\hat E_k$ is quantum regular and if the mappings $\hat E_{k+r+1}\to \hat E_{k+r}$, $\forall r\ge 0$, and $\pi_{k,0}:\hat E_k\to W$ are surjective. The {\em quantum symbol} $\dot g_{k+r}$ of $\hat E_{k+r}$ is a family of $Z\equiv Z(A)$-modules over $\hat E_k$ characterized by means of the following short exact sequence of $Z$-modules: $0\to\pi^*_{k+r}\dot g_{k+r}\to vT\hat E_{k+r}\to\pi^*_{k+r,k+r-1}vT\hat E_{k+r-1}$. Then one has the following complex of $Z$-modules over $\hat E_k$ ({\em$\delta$-quantum complex}):
$$\begin{array}{l}
    \xymatrix{0\ar[r]&\dot g_m\ar[r]^(0.3){\delta}&Hom_Z(TM;\dot g_{m-1})\ar[r]^(0.5){\delta}&Hom_Z(\dot\Lambda ^2_0M;\dot g_{m-2})\ar[r]^(0.7){\delta}\ar[r]&\cdots} \\
\xymatrix{\cdots\ar[r]^(0.25){\delta}&Hom_Z(\dot\Lambda^{m-k}_0M;\dot g_{k})\ar[r]^(0.45){\delta}&\delta(Hom_Z(\dot\Lambda ^{m-k}_0M;\dot g_{k}))\ar[r]& 0}
  \end{array}$$
where $\dot\Lambda ^{s}_0M$ is the skewsymmetric subbundle of $\dot T^r_0M\equiv TM\otimes_{Z}\dots_r\dots\otimes_{Z}TM$. We call {\em Spencer quantumcohomology} of $\hat E_k$ the homology of such complex. We denote by $\{H_q^{m-j,j}\}_{q\in\hat E_k}$ the homology at $(Hom_Z(\dot\Lambda ^{j}_0M;\dot g_{m-j}))_q$. We say that $\hat E_k$ is {\em $r$-quantumacyclic} if $H_q^{m,j}=0$, $m\ge k$, $0\le j\le r$, $\forall q\in\hat E_k$. We say that $\hat E_k$ is {\em quantuminvolutive} if $H_q^{m,j}=0$, $m\ge k$, $j\ge 0$. We say that $\hat E_k$ is {\em $\delta$-regular} if there exists an integer $\kappa_0\ge \kappa$, such that $\dot g_{\kappa_0}$ is quantum involutive or $2$-quantumacyclic.

\begin{theorem}[$\delta$-Poincar\'e lemma for quantum PDE's in the category  $\mathfrak{Q}_{hyper}$]\label{delta-poincare-lemma-pdes-category-quantum-hypercomplex-manifolds}
Let $\hat E_k\subset J\hat{\it D}^k(W)$ be a quantum regular QPDE. If $Z$ is a Noetherian $\mathbb{K}$-algebra, then $\hat E_k$ is a $\delta$-regular QPDE.
\end{theorem}

\begin{proof}
The proof can be copied by the analogous theorem formulated in the category of $\mathfrak{Q}$ \cite{PRA0500, PRA2}. In fact that proof is given there in intrinsic way, thus it does not depend on the particular coordinates representation. Really, from the formal point of view the difference between a quantum manifold $M$ of dimension $m$ over a quantum algebra $B$, and a quantum hypercomplex manifold $N$ of dimension $m$ over a quantum hypercomplex algebra $\mathcal{Q}_r=B\otimes_{\mathbb{R}}\mathbb{A}_r$, is that the quantum coordinates $x^A$ on $M$ take values in the associative algebra $B$ and the quantum coordinates $y^A$ on $N$ take values in the algebra $\mathcal{Q}_r$, that is not associative for $r\ge 3$. But the eventual non-associativity does not influence the geometric intrinsic formal properties of (nonlinear) PDE's built in the category $\mathfrak{Q}_{hyper}$. (For complementary informations on nonassociative algebras see, e.g., \cite{SCHAFER}.)
\end{proof}

\begin{theorem}[Criterion of formal quantum integrability in the category  $\mathfrak{Q}_{hyper}$]\label{criterion-formal-quantum-integrability-pdes-category-quantum-hypercomplex-manifolds}
Let $\hat E_k\subset J\hat{\it D}^k(W)$ be a quantum regular, $\delta$-regular QPDE. Then if $\dot g_{k+r+1}$ is a bundle of $Z$-modules over $\hat E_k$, and $\hat E_{k+r+1}\to\hat E_{k+r}$ is surjective for $0\le r\le m$, then $\hat E_k$ is formally quantum-integrable.
\end{theorem}

\begin{proof}
The proof can be copied by the analogous theorem formulated in the category of $\mathfrak{Q}$ \cite{PRA0500, PRA2}. (Considerations similar to the proof in Theorem \ref{delta-poincare-lemma-pdes-category-quantum-hypercomplex-manifolds} can be made here too.
\end{proof}

An {\em initial condition} for QPDE $\hat E_k\subset J\hat{\it D}^k(W)$ is a point $q\in\hat E_k$. A {\em solution} of $\hat E_k$ passing for the initial condition $q$ is a $m$-dimensional quantum hypercomplex manifold $N\subset \hat E_k$ such that $q\in N$ and such that $N$ can be represented in a neighborood of any of its points $q'\in N$, except for a nowhere dense subset $\Sigma(N)\subset N$ of dimension $\le m-1$, as image of the $k$-derivative $D^ks$ of some $Q^k_w$-section $s$ of $\pi:W\to M$. We call $\Sigma(N)$ the set of {\em singular points} (of Thom-Bordman type) of $N$. If $\Sigma(N)\not=\varnothing$ we say that $N$ is a {\em regular solution} of $\hat E_k\subset J\hat{\it D}^k(W)$. Furthermore, let us denote by $\hat J^k_m(W)$ the $k$-jet of $m$-dimensional quantum manifolds (over $A$) contained into $W$. One has the natural embeddings $\hat E_k\subset J\hat{\it D}^k(W)\subset \hat J^k_m(W) $. Then, with respect to the embedding $\hat E_k\subset \hat J^k_m(W) $ we can consider solutions of $\hat E_k$ as $m$-dimensional (over $A$) quantum hypercomplex manifolds $V\subset\hat E_k$ such that $V$ can be represented in the neighborhood of any of its points $q'\in V$, except for a nowhere dense subset $\Sigma(V)\subset V$, of dimension $\le m-1$, as $N^{(k)}$, where $N^{(k)}$ is the $k$-quantum prolongation of a $m$-dimensional (over $A$) quantum hypercomplex manifold $N\subset W$. In the case that $\Sigma(V)=\varnothing$, we say that $V$ is a {\em regular solution} of $\hat E_k\subset \hat J^k_m(W)$. Of course, solutions $V$ of $\hat E_k\subset \hat J^k_m(W)$, even if regular ones, are not, in general diffeomorphic to their projections $\pi_k(V)\subset M$, hence are not representable by means of sections of $\pi:W\to M$.

Therefore, from Theorem \ref{delta-poincare-lemma-pdes-category-quantum-hypercomplex-manifolds} and Theorem \ref{criterion-formal-quantum-integrability-pdes-category-quantum-hypercomplex-manifolds} we are able to obtain existence theorems of local solutions. Now, in order to study the structure of global solutions it is necessary to consider the integral bordism groups of QPDEs. In \cite{PRA0500, PRA2} we extended to QPDEs our previous results on the determination of integral bordism groups of PDEs \cite{PRA01, PRA1, PRA2}.\footnote{See also Refs. \cite{PRA4, PRA5, PRA6, PRA7, PRA8, PRA9, PRA10, PRA11}.} Let us denote by $\Omega^{\hat E_k}_p$, $0\le p\le m-1$, the integral bordism groups of a QPDE $\hat E_k\subset \hat J^k_m(W)$ for closed integral quantum hypercomplex submanifolds of dimension $p$, over a quantum hypercomplex algebra $A$, of $\hat E_k$. The structure of smooth global solutions of $\hat E_k$ are described by the integral bordism group $\Omega_{m-1}^{\hat E_\infty}$ corresponding to the $\infty$-quantumprolongation
$\hat E_\infty$ of $\hat E_k$. Beside the groups $\Omega^{\hat E_k}_p$, $0\le p\le m-1$, we can also introduce the {\em integral singular $p$-bordism groups} ${}^B\Omega^{\hat E_k}_{p,s}$, $0\le p\le m-1$, where $B$ is a quantum hypercomplex algebra. Then one can prove that ${}^B\Omega^{\hat E_k}_{p,s}\cong \Omega^{\hat E_k}_{p,s}\otimes_\mathbb{K}B$, where $\Omega^{\hat E_k}_{p,s}$ are the integral singular bordism groups for $B=\mathbb{K}$. Furthermore, the equivalence classes in the groups ${}^B\Omega^{\hat E_k}_{p,s}$ are characterized by means of suitable characteristic numbers (belonging to $B$), similarly to what happens for PDEs in the category of commutative manifolds and quantum (super) manifolds. (For details see \cite{PRA0500, PRA2, PRA3, PRA12, PRA13, PRA14, PRA15, PRA16, PRA18-0, PRA18}.)

\begin{example}[Quantum quaternionic manifolds]\label{quantum-quaternionic-manifolds}
Since the category of $\mathfrak{Q}_r$, for $0\le r\le 1$, coincides with the category of quantum manifolds over $\mathbb{K}=\mathbb{R},\, \mathbb{C}$, that we have just explicitly considered in some our previous works, let us first concentrate our attention on the category $\mathfrak{Q}_2$, i.e., with the category of quantum quaternionic manifolds. Really, also an algebraic topology of PDE's in this category has been considered by us in some previous works (see \cite{PRA3} and references therein). However, may be useful to recall here these resultsin some details.

Let us first recall some algebraic definitions and results on quaternionic and Cayley algebra \cite{{BOURBAKI}}. Let $R$ be a commutative ring. Let $\alpha,\beta\in R$, $(e_1,e_2)$ the canonical basis of the $R$-module $R^2$. We say {\em quadratic algebra of type} $(\alpha,\beta)$ over $R$ the $R$-module $R^2$ endowed with the structure of algebra defined by means of the following multiplication:
\begin{equation}\label{multiplication-quadratic-algebra-type-alpha-beta}
    e_1^2=e_1, e_1e_2=e_2e_1=e_2, e_2^2=\alpha e_1+\beta e_2.
\end{equation}

  Any $R$-algebra $E$, isomorphic to a quadratic algebra is called a quadratic algebra too. (Any $R$-algebra $E$ that admits a basis of two elements (one being the identity) is a quadratic algebra.) Then the basis is called a {\em basis of type} $(\alpha,\beta)$. A quadratic algebra $E$ is associative and commutative. Let $E$ be a quadratic $R$-algebra, $e$ its unit. Let $u\in E$ and $T(u)$ the trace of the endomorphism $x\mapsto ux$ of the free $R$-module $E$. Then the application $s:E\to E$, $s(u)=T(u).e-u$, is an endomorphism of the $R$-algebra $E$ and one has $s^2(u)=u$, $\forall u\in E$. A {\em Cayley algebra} on $R$ is a couple $(E,s)$, where $E$ is a $R$-algebra, with unit $e\in E$, and $s$ is a skewendomorphism of $E$ such that: {\em(a)} $u+\bar u\in Re$, {\em(b)} $u.\bar u\in R e$, with $\bar u\equiv s(u)$, $\forall u\in E$. $s$ is called {\em conjugation} of the Cayley algebra $E$ and $s(u)\equiv\bar u$ is the {\em conjugated} of $u$. From the condition {\em(a)} it follows that $u\bar u=\bar u u$. One defines {\em Cayley trace} and {\em Cayley norm} respectively the following maps: $T:E\to Re$, $u\mapsto T(u)=u+\bar u$; $N:E\to Re$, $u\mapsto N(u)=u.\bar u$. One has the following properties:

  {\em(1)} $\bar e=e$;

  {\em(2)} $s(u+s(u))=u+s(u)\Rightarrow s(u)+s^2(u)=u+s(u)\Rightarrow s^2(u)=u\Rightarrow s^2=id_E$;

  {\em(3)} $T(\bar u)=T(u)$;

 {\em(4)} $N(\bar u)=N(u)$;

 {\em(5)} $(u-u)(u-\bar u)=0\Rightarrow u^2-T(u).u+N(u)=0$;

 {\em(6)} Let $E$ be a $R$-algebra and let $s,s'$ be skewendomorphisms of $E$ such that $(E,s)$ and $(E,s')$ are Cayley algebras. If $E$ admits a basis containing $e$, one has $s=s'$;

 {\em(7)} $\overline{u+v}=\bar u+\bar v$; $\overline{\alpha u}=\alpha\bar u$; $\overline{u.v}=\bar v.\bar u$, $\forall\alpha\in R$, $\forall u,v\in E$;

 {\em(8)} $T(e)=2e$; $N(e)=e$;

 {\em(9)} $T(uv)=T(vu)$;

 {\em(10)} $T(v\bar u)=T(u\bar v)=N(u+v)-N(u)-N(v)=T(u)T(v)-T(uv)$;

 {\em(11)} $N(\alpha u)=\alpha^2N(u)$;

 {\em(12)} $(T(u))^2-T(u^2)=2N(u)$;

 {\em(13)} $T$ is a linear form on $E$ and $N$ is a quadratic form on $E$.

$\bullet$\hskip 2pt{\em(Cayley extension of a Cayley algebra $(E,s)$ defined by an element $\gamma\in R$).}
Let $(E,s)$ be a Cayley algebra and let $\gamma\in R$. Let $F$ be the $R$-algebra with underlying module $E\times E$ and with multiplication $(x,y)(x',y')=(xx'+\gamma\bar y'y,y\bar x'+y'x)$. Then $ (e,0)$ is the unit of $ F$ and $ E\times\{0\}$ is a subalgebra of $ F$ isomorphic to $ E$ that can be identified with $ E$. Let $ t$ be the permutation of $ F$ defined by $ t(x,y)=(\bar x,-y)$, $ \forall x, y\in E$. Then the couple $ (F,t)$ is a Cayley algebra over $ R$. Set $ j=(0,e)$. So we can write $ (x,y)=(x,0)(e,0)+(0,y)(0,e)=xe+yj$. One has $ yj=j\bar y$, $ x(yj)=(yx)j-(xj)y=(x\bar y)j$, $ (xj)(yj)=\bar y xe$, $ j^2=e$. Furthermore, one has $ T_F(xe+yj)=T(x)$, $ N_F(xe+yj)=N(x)-\gamma N(y)$. $ F$ is associative iff $ E$ is associative and commutative.

As a particular case one has: If $ E=R$ (hence $ s=id_R$), the Cayley extension of $ (R.id_R)$ by an element $ \gamma\in R$ is a {\em quadratic $ R$-algebra} with basis $ (e,j)$ with $ j^2=\gamma e$.

Another particular case is the following. Let $ E$ be a quadratic algebra of type $ (\alpha,\beta)$ such that the underlying module is $ R^2$ with multiplication rule given by means of {\em(\ref{multiplication-quadratic-algebra-type-alpha-beta})} for the canonical basis. Let the conjugation $ s$ be the conjugation in $ E$. Then for any $ \gamma\in R$, the Cayley extension $ F$ of $ (E,s)$ by means of $ \gamma$ is called {\em quaternionic algebra of type} $ (\alpha, \beta,\gamma)$. (This is an associative algebra.) The underlying module is $ R^4$. Let us denote by $ (0,i,j,k)$ the canonical basis of $ R^4$. Then the corresponding multiplication rule is given in Tab. \ref{multiplication-table-trace-norm-formulas-a}. (In the same table it are also reported the trace and norm formulas.)

\begin{table}
\caption{Multiplication table, trace and norm formulas for quaternionic algebra $F\cong R^4$ of type $(\alpha,\beta,\gamma)$ with basis $(0,i,j,k)$.}
\label{multiplication-table-trace-norm-formulas-a}
\begin{tabular}{|l|l|l|l|l|}
\hline
{\rm{\footnotesize }}&{\rm{\footnotesize \hfil$ i$\hfil}}&{\rm{\footnotesize \hfil$ j$\hfil}}&{\rm{\footnotesize \hfil$ k$\hfil}}&{\rm{\footnotesize trace and norm formulas}}\\
\hline
{\rm{\footnotesize \hfil $ i$\hfil}}&{\rm{\footnotesize \hfil $ \alpha e+\beta i$\hfil}}&{\rm{\footnotesize \hfil $ k$\hfil}}&{\rm{\footnotesize \hfil $ \alpha j+\beta\kappa$\hfil}}&{\rm{\footnotesize \hfil $ T_F(u)=2\rho+\beta\xi$\hfil}}\\
\hline
{\rm{\footnotesize \hfil $ j$\hfil}}&{\rm{\footnotesize \hfil $ \beta j-\kappa$\hfil}}&{\rm{\footnotesize \hfil $ \gamma e$\hfil}}&{\rm{\footnotesize \hfil $ \beta\gamma-\gamma i$\hfil}}&{\rm{\footnotesize \hfil $ N_F(u)=\rho^2+\beta\rho\xi-\alpha\xi^2-\gamma(\eta^2+\beta\eta\zeta-\alpha\zeta^2)$\hfil}}\\
\hline
{\rm{\footnotesize \hfil $ k$\hfil}}&{\rm{\footnotesize \hfil $ -\alpha j$\hfil}}&{\rm{\footnotesize \hfil $ \gamma i$\hfil}}&{\rm{\footnotesize \hfil $ -\alpha\gamma e$\hfil}}&{\rm{\footnotesize \hfil $ N_F(uv)=N_F(u)N_F(v)$\hfil}}\\
\hline
\multicolumn{5}{l}{\rm{\footnotesize $ u=\rho e+\xi i+\eta j+\zeta\kappa,\quad \rho,\xi,\eta,\zeta\in\mathbb{K};\quad
\bar u=(\rho+\beta\xi)e-\xi i-\eta j-\zeta\kappa$.}}\\
\multicolumn{5}{l}{\rm{\footnotesize $R$ commutative ring. $e$ unit of the quadratic algebra $(E,s)$ of type $(\alpha,\beta)$.}}\\

\end{tabular}
\end{table}

An $ R$-algebra isomorphic to a quaternionic algebra is called a {\em quaternionic algebra} of type $(\alpha,\beta,\gamma)$, if it has a basis with multiplication table given in Tab. \ref{multiplication-table-trace-norm-formulas-a}.

If $\beta=0$ we say that the quaternionic algebra is of type $(\alpha,\gamma)$. The corresponding multiplication table is given in Tab. \ref{multiplication-table-trace-norm-formulas-b}.

\begin{table}
\caption{Multiplication table, trace and norm formulas for quaternionic algebra $F\cong R^4$ of type $(\alpha,\gamma)$ with basis $(0,i,j,k)$.}
\label{multiplication-table-trace-norm-formulas-b}
\begin{tabular}{|l|l|l|l|l|}
\hline
{\rm{\footnotesize }}&{\rm{\footnotesize \hfil$ i$\hfil}}&{\rm{\footnotesize \hfil$ j$\hfil}}&{\rm{\footnotesize \hfil$ k$\hfil}}&{\rm{\footnotesize trace and norm formulas}}\\
\hline
{\rm{\footnotesize \hfil $ i$\hfil}}&{\rm{\footnotesize \hfil $\alpha\, e$\hfil}}&{\rm{\footnotesize \hfil $\kappa $\hfil}}&{\rm{\footnotesize \hfil $ \alpha j$\hfil}}&{\rm{\footnotesize \hfil $ T_F(u)=2\rho$\hfil}}\\
\hline
{\rm{\footnotesize \hfil $ j$\hfil}}&{\rm{\footnotesize \hfil $ -\kappa$\hfil}}&{\rm{\footnotesize \hfil $ \gamma e$\hfil}}&{\rm{\footnotesize \hfil $ -\gamma i$\hfil}}&{\rm{\footnotesize \hfil $ N_F(u)=\rho^2-\alpha\xi^2-\gamma\eta^2+\alpha\zeta^2$\hfil}}\\
\hline
{\rm{\footnotesize \hfil $ k$\hfil}}&{\rm{\footnotesize \hfil $-\alpha\, j$\hfil}}&{\rm{\footnotesize \hfil $ \gamma\,  i$\hfil}}&{\rm{\footnotesize \hfil $  -\alpha\, \gamma\,  e$\hfil}}&\\
\hline
\multicolumn{5}{l}{\rm{\footnotesize $ u=\rho e+\xi i+\eta j+\zeta\kappa,\quad \rho,\xi,\eta,\zeta\in\mathbb{K};\quad
\bar u=\rho e-\xi i-\eta j-\zeta\kappa$.}}\\
\end{tabular}
\end{table}

In particular if $ R=\mathbb{K}=\mathbb{R}$, $ \alpha=\gamma=-1$, $ \beta=0$, $ F$ is called the {\em Hamiltonian quaternionic algebra} and is denoted by $ \mathbb{H}$. In this case $ N(u)\not=0$, hence $ u$ admits an inverse $ u^{-1}=N(u)^{-1}\bar u$ in $ \mathbb{H}$, therefore $ \mathbb{H}$ is a noncommutative corp. Any finte $ \mathbb{R}$-algebra that is also a corp (noncommutative) is isomorphic to $ \mathbb{H}$. Any quaternion $ q\in \mathbb{H}$ can be represented by $ q=\rho e+\xi i+\eta j+\zeta k$, where $ i,j,k$ are linearly independent symbols that satisfy the following multiplication rules: $ ij=k=-ji$, $ jk=i=-kj$, $ ki=j=-ik$, $ i^2=j^2=k^2=-1$. One has the following $ \mathbb{R}$-algebras homomorphism: $ A:\mathbb{H}\to M(2;\mathbb{C})$, $ q\mapsto
\left(
  \begin{array}{cc}
    a+bi & c+di \\
    -c+di & a-bi \\
  \end{array}
\right)$,
where $ i$ is the imaginary unity of $ \mathbb{C}$. The matrices $ \sigma_x\equiv -iA(k)$, $ \sigma_y\equiv -iA(j)$, $ \sigma_z\equiv -iA(i)$, where $ A(i)=\left(
  \begin{array}{cc}
    i & 0 \\
    0 & -i \\
  \end{array}
\right)$, $ A(j)=\left(
  \begin{array}{cc}
    0 & 1 \\
    -1 & 0 \\
  \end{array}
\right)$, $ A(k)=\left(
  \begin{array}{cc}
    0 & i \\
    i & 0 \\
  \end{array}
\right)$, are called {\em Pauli matrices} and satisfy $ \sigma_x^2=\sigma_y^2=\sigma_z^2=1$, $ \sigma_x\sigma_y=-\sigma_y=\sigma_x=i\sigma_z$. The set $ \mathbb{H}_1\equiv N_\mathbb{H}^{-1}(1)$ of quaternions of norm $ 1$ is isomorphic to the group $ SU(2)$: $ \mathbb{H}_1\cong SU(2)$. The {\em  $ n$-dimensional quaternionic space} $\mathbb{H}^n$ has a canonical basis $ \{e_k\}_{1\le k\le n}$, $ e_k\in\mathbb{H}$, and any $ v\in\mathbb{H}^n$ can be represented in the form $ v=\sum_{1\le k\le n}q^ke_k$, $ q^k\in\mathbb{H}$, ($ q^k\equiv ${\em  quaternionic components}). As any quaternionic number $ q$ admits the following representation $ q=x+yj=x+j\bar y$, with $ x=\rho e+\xi i$, $ y=\eta +\zeta i$, where $ x$ and $ y$ can be considered complex numbers, then one has the following isomoprphism $ \mathbb{H}^n\cong\mathbb{C}^{2n}$, $ (q^k)\mapsto(x^k,y^k)$, where $ \mathbb{C}^{2n}$ has the following basis $ (e_1,\dots,e_n,je_1,\dots,je_n)$. We write $ \dim_\mathbb{H}\mathbb{H}^n=n$, $ \dim_\mathbb{C}\mathbb{H}^n=2n$. By using different quaternionic bases in $ \mathbb{H}^n$ one has that the quaternionic components of any vector $ v\in\mathbb{H}^n$ transform by means of the following rule $ q'^{k}=\sum_{1\le i\le n}q^i\lambda^k_i$, $ (\lambda^k_i)\in GL(n,\mathbb{H})$. Furthermore, the corresponding complex components transform in the following way:
$$ \left\{ x'^h=x^la^h_l-y^l\bar b^h_l,\quad y'^h=x^lb^h_l+y^l\bar a^h_l\right\}, \quad \lambda^h_l=a^h_l+b^h_lj.$$
Then one has a group-homomorphism $ GL(n,\mathbb{H})\mathop{\to}\limits^{c}GL(2n,\mathbb{C})$, such that if $ \Lambda=A+Bj\in GL(n,\mathbb{H})$, then $ c(\Lambda)=\left(\begin{array}{cc}
                                                       A & B \\
                                                       -\bar B &  \bar A \\
                                                     \end{array}\right)$.
On $ \mathbb{H}^n$ there is a canonical quadratic form $ |v|^2=\sum_{1\le k\le n}|q^k|^2=\sum_{1\le k\le n}q^k\bar q^k=\sum_{1\le k\le n}(|x^k|^2+|y^k|^2)\in\mathbb{R}$, where $ v=q^ke_k$, $ q^k\in\mathbb{H}$, $ q^k=x^k+y^kj$, $ x^k,y^k\in\mathbb{C}$. So such a quadratic form coincides with the ordinary norm of the vector space $ \mathbb{C}^{2n}$. Furthermore one has on $ \mathbb{H}^n$ the following form $ <v_1,v_2>_\mathbb{H}=\sum_{1\le k\le n}q_1^k\bar q_2^k\in\mathbb{H}$, $ v_i=\sum_{1\le k\le n}q_i^ke_k$, $ i=1,2$. The quaternionic transformations of $ \mathbb{H}^n$, that conserve above form a group $ Sp(n)\subset GL(n,\mathbb{H})$. As we can write $ <v_1,v_2>_\mathbb{H}$ in the following way:

$$\scalebox{0.8}{$ <v_1,v_2>_\mathbb{H}=\left\{
\begin{array}{l}
\sum_{1\le k\le n}(x^k_1\bar x^k_2+y_1^k\bar y_2^k)=<v_1,v_2>_\mathbb{C}=\hbox{\rm hermitiam form in} \mathbb{C}^{2n}\\
\sum_{1\le k\le n}(y^k_1 x^k_2-x_1^k y_2^k)j=\sigma(v_1,v_2)_\mathbb{C}=\hbox{\rm skewsymmatric form in } \mathbb{C}^{2n}\\
\end{array}
\right.$}$$
we see that $ \Lambda\in Sp(n)$ preserves the hermitian form and the skewsymmetric form. Therefore $ c(Sp(n))\subset U(2n)$ and it is formed by the unitary transformations of $ \mathbb{C}^{2n}$ that preserves the antisymmetric form $ \sigma(v_1,v_2)_\mathbb{C}$.\footnote{$ Sp(1)\cong SU(2)\subset U(2)$. So all the transformations contained in $ c(Sp(1))$ are unimodular.}

$\bullet$\hskip 2pt Let $B$ be a quantum algebra. We define {\em Cayley $B$-quantum algebra} any quantum algebra $C$ that is obtained from a Cayley $\mathbb{K}$-algebra $A$, by ''extending the scalars'' from $\mathbb{K}$ to $B$, i.e., $C\cong B\otimes_\mathbb{K}A$.

The noncommutative $B$-quaternionic algebra $B\otimes_\mathbb{K}\mathbb{H}$ is a Cayley $B$-quantum algebra over $\mathbb{K}=\mathbb{R}$ endowed with the natural topology of Banach space and considering the $\mathbb{K}$-linear morphism $c=c_B\otimes{1\over 2}T:B\otimes_\mathbb{K}\mathbb{H}\to\mathbb{K}$, where $T$ is the trace of $\mathbb{H}$. (Another possibility is to take $c=c_B\otimes N$, where $N$ is the norm of $\mathbb{H}$. In this last case, whether $B$ is an augmented quantum algebra, then $B\otimes_\mathbb{K}\mathbb{H}$ becomes an augmented quantum algebra too.)

$\bullet$\hskip 2pt A {\em quantum $B$-quaternionic manifold} of dimension $n$ and class $Q^k_w$, $0\le k\le \infty,\omega$, is a quantum manifold $M$ of dimension $n$ and class $Q^k_w$ over the $B$-quantum algebra $C\equiv B\otimes_\mathbb{K}\mathbb{H}$. Then the quantum coordinates in an open coordinate subset $U\subset M$ are called {\em $B$-quaternionic coordinates}, $\{q^k\}_{1\le k\le n}$, $q^k:U\to C$.\footnote{As a particular case we can take $ B=\mathbb{K}$. In this case $ C=\mathbb{H}$ and we call such quantum $ \mathbb{K}$-quaternionic manifolds simply {\em quantum quaternionic manifolds}.}

The {\em category $\mathcal{C}^B_{\mathbb{H}}$, of quantum $B$-quaternionic manifolds} of class $Q^k_w$, is defined by considering as morphisms maps of class $Q^k_w$, between quantum $B$-quaternionic manifolds.
\end{example}

\begin{example}[Quantum quaternionic M\"obius strip]\label{quantum-quaternionic-mobius-strip}
Let us denote $I\equiv[-\pi,\pi]\subset\mathbb{R}$, $N\equiv I\times \mathbb{H}$. Let us introduce the following equivalence relation in $N$: $(x,y)\sim(x,y)$ if $x\not=-\pi,\pi$, $(-\pi,-y)\sim(\pi,y)$. Then $N/\sim\equiv M$ is called noncommutative quaternionic M\"obius strip. One has a natural projection $p:M\to S^1$, given by $p([x,y])=x\in S^1$ if $x\not=-\pi,\pi$, and $p([\pi,y])=*\in S^1$, where $*$ is the point of $S^1\equiv I/\{-\pi,\pi\}$, corresponding to $\{-\pi,\pi\}$. One can recover $M$ with two open sets:
$$\left\{\Omega_1\equiv p^{-1}(U_1), \quad U_1\equiv]-\pi,\pi[;\quad
            \Omega_2\equiv p^{-1}(U_2), \quad U_2\equiv S^1\setminus\{0\}\right\}.$$
We put quaternionic coordinates on $\Omega_i$, $i=1,2$, in the following way. On $\Omega_1$, $\{x^1[x,y]=p[x,y]=x\in\mathbb{R},
                     x^2[x,y]=y\in\mathbb{H}\}$. On $\Omega_2$, if $x\not=-\pi,\pi$, $\bar x^1[x,y]=-\pi+x\in\mathbb{R}$,
                                $\bar x^1[-x,y]=\pi-x\in\mathbb{R}$,
                                $\bar x^2[\pm x,y]=y\in\mathbb{H}$;
                     $\bar x^1[-\pi,-y]=\bar x^1[\pi,y]=0$,
                                       $\bar x^2[-\pi,-y]=\bar x^2[\pi,y]=|y|\in c^{-1}(\mathbb{                                    R}^+)\subset\mathbb{H}$, with $c={1\over 2}T$. The change of coordinates is given by:
$$\forall q\in\Omega_1\cap \Omega_2, \quad (p(q)\in S^1\setminus\{*,0\}\equiv U_+\sqcup U_-),\quad\left\{
\begin{array}{l}
\bar x^1|_{U_-}=\pi-x^1\in\mathbb{R}\\
\bar x^1|_{U_+}=-\pi+x^1\in\mathbb{R}\\
\bar x^2=x^2\in\mathbb{H}\\
\end{array}
\right\}.$$
The structure group, i.e. the group of the jacobian matrix, is isomorphic to $\mathbb{Z}_2$. In fact one has:
$$\left\{(\partial x_j.\bar x^k)=
\left(
  \begin{array}{cc}
    (\partial x_1.\bar x^1) & (\partial x_1.\bar x^2) \\
    (\partial x_2.\bar x^1) & (\partial x_2.\bar x^2) \\
  \end{array}
\right)=
\left(
  \begin{array}{cc}
    1 & 0 \\
    0 & 1\\
  \end{array}
\right)\subset GL(2;\mathbb{H})\right\}.$$
Therefore, $M$ is a quantum quaternionic manifold modelled on $\mathbb{R}\times\mathbb{H}\subset\mathbb{H}^2$, hence $\dim_\mathbb{H}M=2$. Furthermore one has the canonical projection $M\to S^1$, therefore $M$ is a regular quantum manifold of dimension $(1\downarrow 2)$. Finally remark that as $M$ is not covered by a global chart, it is a non trivial example of quantum quaternionic manifold. Of course this can be also seen by means of homological arguments. In fact one has $H_1(M;\mathbb{R})\cong H_1(M_C;\mathbb{R})\cong H_1(S^1;\mathbb{R})\cong\mathbb{R}$. Therefore, $M$ is not homotopy equivalent to $\mathbb{R}^5$, as $H_1(\mathbb{R}^5;\mathbb{R})=0$.
\end{example}

\begin{example}[Quaternionic manifolds \cite{SOMMESE, SUDBERG}] The category $\mathcal{C}_\mathbb{H}$ of quaternionic manifolds is a subcategory of $\mathcal{C}_\mathbb{H}^{B\equiv\mathbb{R}}$, where the morphisms are quaternionic affine maps. Therefore any of such morphisms $f\in Hom_{\mathcal{C}_\mathbb{H}}(M,N)$, where $\dim M=4m$, $\dim N=4n$, are locally represented by formulas like the following: $f^k=A^k_jq^j+r^k$, $A^k_j,q^j,r^k\in\mathbb{H}$, $1\le k\le n$, $1\le j\le m$. $A^k_j$ identify $m\times n$ matrices with entries in $\mathbb{H}$, or equivalently, real matrices of the form
$$\left(A^k_j\right)=
\left(
  \begin{array}{ccc}
     \hat A^1_1 & \dots & \hat A^m_1 \\
    \dots & \dots & \dots \\
    \hat A^1_n & \dots & \hat A^m_n \\
  \end{array}
\right),\quad
\hat A^j_i=
\left(
  \begin{array}{cccc}
    a & b & c & d \\
    { -b}&{ a}&{ d}&{ -c}\\
                                   { -c}&{ -d}&{ a}&{ b}\\
                                    { -d}&{ c}&{ b}&{ a}\\
  \end{array}
\right),\quad a,b,c,d\in\mathbb{R}.$$
The set of such matrices is denoted by $M(n,m;\mathbb{H})$. The structure group of a $4n$-dimensional quaternionic manifold is $GL(n;\mathbb{H})$, (that is the subset of $M(n,m;\mathbb{H})$ of invertible matrices). Therefore, quaternionic manifolds are quantum quaternionic manifolds where the local maps $f:U\subset\mathbb{H}^n\to\bar U\subset\mathbb{H}^n$, change of coordinates, are $\mathbb{H}$-linear. Hence $Df(p)\in \mathbb{H}^{n^2}$, $\forall p\in U$. In fact, one has the following commutative diagram:
$$\xymatrix{{ Hom_\mathbb{H}(\mathbb{H}^n;\mathbb{H}^n)}\ar@{=}[d]^{\wr}\ar@{^{(}->}[r]&{ Hom_\mathbb{R}(\mathbb{H}^n;\mathbb{H}^n)}\ar@{=}[r]^{\sim}&{ Hom_\mathbb{R}(\mathbb{H};\mathbb{H})^{n^2}}\ar@{=}[d]^{\wr}\\
{ Hom_\mathbb{H}(\mathbb{H};\mathbb{H})^{n^2}}\ar@{=}[d]^{\wr}&&{ Hom_\mathbb{R}(\mathbb{R}^4;\mathbb{R}^4)^{n^2}}\ar@{=}[d]^{\wr}\\
{ \mathbb{H}^{n^2}}\ar@{=}[r]^{\sim}&{ \mathbb{R}^{4n^2}}\ar@{^{(}->}[r]&{ \mathbb{R}^{4^2n^2}}\\ }$$
On the other hand the tangent space $T_pM$ has a natural structure of $\mathbb{H}$-module iff $M$ is an affine manifold. (As in this case the action of $\mathbb{H}$ on $T_pM\cong\mathbb{H}^n$ does not depend on the coordinates used to obtain the identification of $T_pM$ with $\mathbb{H}^n$.) Hence the category $\mathcal{C}_\mathbb{H}$ is the subcategory of $\mathcal{C}_\mathbb{H}^\mathbb{R}$ of affine quantum quaternionic manifolds. A trivial example of quaternionic manifold is $\mathbb{R}^{4n}\cong\mathbb{H}^n$. If $\{x^i,y^i,u^i,v^i\}_{1\le i\le n}$ are real coordinates on $\mathbb{R}^{4n}$, then the almost quaternionic structure given by
$$ \left\{
\begin{array}{l}
J(\partial x_i)=\partial y_i,\quad J(\partial y_i)=-\partial x_i,\quad J(\partial u_i)=-\partial v_i,\quad J(\partial v_i)=\partial u_i\\
 K(\partial x_i)=\partial u_i,\quad K(\partial y_i)=\partial v_i,\quad K(\partial u_i)=-\partial x_i,\quad J(\partial v_i)=-\partial y_i\\
 \end{array}
 \right\}$$
is called the {\em standard right quaternionic structure} on $\mathbb{R}^{4n}$.\footnote{An {\em  almost complex structure} on a $ C^\infty$ manifold $ M$ is a fiberwise endomorphism $ J$ of the tangent bundle $ TM$ such that $ J^2=-1$. A {\em  complex analytic map} between almost complex manifolds $ (X,J_1)$ and $ (Y,J_2)$ is a $ C^\infty$ map $ \phi:X\to Y$, such that $ T(\phi)\circ J_1=J_2\circ T(\phi)$. An {\em  almost quaternionic structure} on a $ C^\infty$ manifold $ M$ is a pair of two almost complex structures $ J$ and $ K$ such that $ JK+KJ=0$. A {\em  quaternionic map} $ \phi$ between two almost quaternionic manifolds $ (X,J_1,K_1)$ and $ (Y,J_2,K_2)$ is a map $ \phi:X\to Y$ that is complex analytic from $ (X,J_1)$ to $(Y,J_2)$ and from $ (X,K_1)$ to $(Y,K_2)$. A {\em  quaternionic manifold} is a $ C^\infty$ manifold $ M$ endowed with an atlas $ \{\phi_i:U_i\to\mathbb{R}^{4n}\}$, for some $ n$, such that $ \phi_j\circ\phi_i^{-1}:\phi_i(U_i\cap U_j)\to \phi_j(U_i\cap U_j)$ is a quaternionic function with respect to the standard structure on $ \mathbb{R}^{4n}$. (See also \cite{SOMMESE, SUDBERG}.)} A non trivial example of quaternionic manifold is $\mathbb{R}^{4n}$ with the standard quaternionic structure quotiented by a discrete translation group that gives a torus.\end{example}

\begin{example}[Almost quaternionic manifolds]
The category $\widetilde{\mathcal{C}_\mathbb{H}}$ of almost quaternionic manifolds is a subcategory of $\mathcal{C}_\mathbb{H}^{B\equiv\mathbb{R}}$, where the structure group of a $4n$-dimensional quantum quaternionic manifold is $GL(n;\mathbb{H})Sp(1)\subset GL(4n;\mathbb{R})$. The category $\widetilde{\mathcal{C}_\mathbb{H}}$ properly contains $\mathcal{C}_\mathbb{H}$. An example of almost quaternionic manifold,  that is not contained into $\mathcal{C}_\mathbb{H}$, is the quaternionic projective space $\mathbb{H}P^1$. This cannot be a quaternionic manifold, since it does not admit a structure of complex manifold.\end{example}

\begin{remark}
As the centre $Z(C)$ of $C\equiv B\otimes_\mathbb{K}\mathbb{H}$ is isomorphic to $Z(B)$, that is the centre of $B$, we get that, whether $Z(B)$ is Noetherian,  one can apply above Theorem 1.1 and Theorem 1.2 for QPDEs, in order to state the formal quantum-integrability for quantum $B$-quaternionic PDEs. Note that in such a way we obtain as solutions submanifolds that have natural structures of quantum $B$-quaternionic manifolds. Then applying our theorems on the integral bordism groups for quantum PDEs, we can also calculate theorem of existence of global solutions for quantum $B$-quaternionic PDEs.
\end{remark}

\begin{example}[Quantum $B$-quaternionic heat equation]
Let us consider the fiber bundle $\pi:W\equiv C^3\to C^2\equiv M$ with coordinates $(t,x,u)\mapsto (t,x)$. The quantum $B$-quaternionic heat equation is the following QPDE: $\widehat{(HE)}_{C}\subset \hat J{\it D}^2(W)\subset \hat J^2_2(W)$: $u_{xx}-u_t=0$. This is a formally (quantum)integrable QPDE. Hence, for $\widehat{(HE)}_{C}$ we have the existence of local solutions for any initial condition. This means that in the neighborhood of any point $q\in{\widehat{(HE)}}_{C}$ we can built an integral quantum $B$-quaternionic manifold of dimension $2$ over $C$, $V\subset\widehat{(HE)}_{C}$, such that $V\cong \pi_2(V)\subset M$, where $\pi_2$ is the canonical projection $\pi_2:\hat J{\it D}^2(W)\to M$. Then by using a Theorem 5.6 given in \cite{PRA2} we have that the first integral bordism groups of $\widehat{(HE)}_{C}$ is: $\Omega_1^{{\widehat{(HE)}}_{C}}\cong H_1(W;\mathbb{K})\otimes_\mathbb{K}C\cong 0$. Hence we get that any admissible closed integral $1$-dimensional quantum $B$-quaternionic manifold, $N\subset{\widehat{(HE)}}_{C}$ is the boundary of an integral $2$-dimensional quantum $B$-quaternionic manifold $V$, $\partial V\subset N$, $V\subset{\widehat{(HE)}}_{C}$, such that $V$ is diffeomorphic to its projection into $W$ by means of the canonical projection $\pi_{2,0}:\hat J^2_2(W)\to W$.
\end{example}

\begin{example}[Quantum quaternionic heat equation]
As a particular case of above equation one can take $B\equiv\mathbb{R}$. Then one has:
\begin{equation}\label{quantum-quaternionic-heat-equation-definition}
\scalebox{0.8}{${\left\{
\begin{array}{l}
\pi:W\equiv\mathbb{H}^3\to\mathbb{H}^2\equiv M;\quad (t,x,u)\mapsto (t,x)\\
\widehat{(HE)}_\mathbb{H}\subset J{\it D}^2(W)\subset  J^2_2(W):\quad u_{xx}-u_t=0\\
\end{array}
\right\},\quad
\Omega_1^{{\widehat{(HE)}}_\mathbb{H}}\cong H_1(W;\mathbb{R})\otimes_\mathbb{R}\mathbb{H}\cong 0.}$}
\end{equation}
We can see that the set $\mathcal{S}ol(\widehat{(HE)}_\mathbb{H})$ of solutions of $\widehat{(HE)}_\mathbb{H}$ contains also quaternionic manifolds, i.e., affine quantum quaternionic solutions. For example a torus \footnote{Recall \cite{SOMMESE, SUDBERG} that if $ (X,J,K)$ is a quaternionic manifold, then $ X$ with the complex structure $ aJ+bK+c(JK)$, $ a,b,c\in\mathbb{R}$, is an affine complex manifold, hence has zero rational Pontryagin classes. Furthermore, if $ X$ is compact has zero index and Euler characteristic. Moreover, if $ \dim_\mathbb{H}X=1$ and, for some $ a,b,c$, $ X$ is K\"ahler, then it is a torus.} $X\subset\mathbb{H}^2\equiv M$ can be embedded into ${\widehat{(HE)}}_\mathbb{H}$ by means of the second holonomic prolongation of the zero section $u\equiv 0:M\to W$. In fact, ${\widehat{(HE)}}_\mathbb{H}$ is a linear equation. Therefore, $X^{(2)}\equiv D^2u(X)$ is a $1$-dimensional smooth closed compact admissible integral manifold contained into ${\widehat{(HE)}}_\mathbb{H}$, that is the boundary of a $2$-dimensional integral admissible manifold contained into ${\widehat{(HE)}}_\mathbb{H}$ too. This last is also a quaternionic manifold. Moreover, all the regular solutions of ${\widehat{(HE)}}_\mathbb{H}\subset J{\it D}^2(W)$ are quaternionic manifolds, as they are diffeomorphic to $\mathbb{H}^2$. However, no all the regular solutions of ${\widehat{(HE)}}_\mathbb{H}\subset J^2_2(W)$ are necessarily quaternionic manifolds too.\end{example}

We are ready now to state the following theorem.

\begin{theorem}\label{integral-bordism-groups-quantum-hypercomplex-pdes}
Let $B$ be a quantum algebra such that its centre $Z(B)$ is a Noetherian $\mathbb{R}$-algebra. Let $\hat E_k\subset J\hat{\it D}^k(W)$ be a quantum regular QPDE in the category $\mathcal{C}^B_\mathbb{H}$, where $\pi:W\to M$ is a fibre bundle with $dim_CM=m$, $C\equiv B\otimes_\mathbb{R}\mathbb{H}$. If $\dot g_{k+r+1}$ is a bundle of $Z(C)$-modules over $\hat E_k$, and $\hat E_{k+r+1}\to\hat E_{k+r}$ is surjective for $0\le r\le m$, then $\hat E_k$ is formally quantumintegrable. In such a case, and further assuming that $W$ is $p$-connected, $p\in\{0,\dots,m-1\}$, then the integral bordism groups of $\hat E_k\subset\hat J^k_m(W)$ are given by:
$$\Omega_p^{\hat E_k}\cong H_p(W;\mathbb{R})\otimes_\mathbb{R}C,\quad 0\le p\le m-1.$$
All the regular solutions of $\hat E_k\subset \hat J^k_m(W)$ are quantum $B$-quaternionic submanifolds of $\hat E_k$ of dimension $m$, over $C$, identified with $m$-dimensional quantum $B$-quaternionic submanifolds of $W$.
\end{theorem}

\begin{proof} It follows directly from above definitions and remarks by specializing Theorem 1.1, Theorem 1.2 and our results in \cite{PRA2}, about integral bordism groups in QPDEs, to the category $\mathcal{C}^B_\mathbb{H}$.\end{proof}

\begin{cor}
Let $\hat E_k\subset J{\it D}^k(W)$ be a quantum regular QPDE in the category $\mathcal{C}^{\mathbb{R}_\mathbb{H}}$, (resp. $\widetilde{\mathcal{C}}_\mathbb{H}$), where $\pi:W\to M$ is a fibre bundle with $dim_{\mathbb{H}}M=m$. If $\dot g_{k+r+1}$ is a bundle of $\mathbb{R}$-modules over $\hat E_k$, and $\hat E_{k+r+1}\to\hat E_{k+r}$ is surjective for $0\le r\le m$, then $\hat E_k$ is formally quantumintegrable. In such a case, further assuming that $W$ is $p$-connected, $p\in\{0,\dots,m-1\}$, then the integral bordism groups of $\hat E_k\subset J^k_m(W)$ are given by:
$$\Omega_p^{\hat E_k}\cong H_p(W;\mathbb{R})\otimes_{\mathbb{R}}\mathbb{H},\quad 0\le p\le m-1.$$
All the regular solutions of $\hat E_k\subset J^k_m(W)$ are quantum quaternionic, (resp. almost quaternionic), submanifolds of $\hat E_k$ of dimension $m$, identified with $m$-dimensional quantum quaternionic, (resp. almost quaternionic), submanifolds of $W$.
\end{cor}

Before to pass to the following example on the heat PDE over octonions, let us first recall some useful definitions and results about alternative algebras.\footnote{The algebra $\mathbb{O}$ of octonions is just a distinguished example of alternative algebra.}

\begin{definition}
The {\em associator} of an algebra $A$ is a trilinear map $[,,]:A\times A\times A\to A $, such that $[a,b,c]\equiv (ab)c-a(bc)$.
When $[a,b,c]=0$ we say that the three elements {\em associate}.

A $n$-multilinear mapping $f:\underbrace{A\times\cdots\times A}_{n}\to A$ is alternating if it vanishes whenever two of its arguments are equal.

An {\em alternative algebra} is one where the associator is alternative.
\end{definition}

\begin{proposition}[Properties of alternative algebras]

{\em 1)} Given an algebra $A$ the following propositions are equivalent.

{\em (i)} $A$ is an alternative algebra.

{\em (ii)} The following propositions hold:

{\em (a) (left alternative identity):} $[a,a,b]=0$, $\for a,b\in A$.

{\em (b) (Right alternative identity):} $[a,b,b]=0$, $\for a,b\in A$.

{\em (c) (Flexibility identity):} $a(ba)=a(ba)$, $\for a,b\in A$.

{\em 2)} An alternating associator is always totally skew-symmetric, i.e., $[a_{\sigma(1)},a_{\sigma(2)},a_{\sigma(3)}]=\epsilon(\sigma)[a_1,a_2,a_3]$, where $\epsilon(\sigma)$ is the signature of the permutation $\sigma\in S_3$ and $S_3$ is the group of permutations of three objects.\footnote{The converse holds too when the characteristic of the base field $\mathbb{K}$ of $A$ is not $2$, (e.g. $\mathbb{K}=\mathbb{R}$).}

{\em 3) (Artin's theorem)} In an alternative algebra $A$ the subalgebras generated by any two elements is associative, and vice versa.

{\em 4)} Let $A$ be an alternative algebra, then the sublagebra generated by three elements $a,b,c\in A$, such that $[a,b,c]=0$, is associative.

{\em 5)} Alternative algebras are {\em power-associative}, that is, the sublagebra, generated by a simple element is associative.\footnote{The converse need not hold: the sedenions are power-associative, but not alternative.}

{\em 6) (Moufoang identities)} In any alternative algebra hold the following identities:

{\em (a)} $a(x(ay))=(axa)y$.

{\em (b)} $((xa)y)a=x(aya)$.

{\em (c)} $(ax)(ya)=a(xy)a$.

{\em 7)} If $A$ is a unital alternative algebra, multiplicative inverse are unique whenever they exist. Furthermore one has the identities reported in {\em(\ref{identities-unital-alternative-algebra-and-multiplicative-inverse})}.
\begin{equation}\label{identities-unital-alternative-algebra-and-multiplicative-inverse}
 \left\{
   \begin{array}{l}
b=a^{-1}(ab) \\
{[a^{-1},a,b]}=0\\
(ab)^{-1}=b^{-1}a^{-1}\\
   \end{array}
 \right\}\, \hbox{\rm if $a$ and $b$ are invertible}.
\end{equation}
\end{proposition}

\begin{example}[Quantum octonionic heat equation]
The octonions form a normed division $\mathbb{R}$-algebra that is non-associative, non-commutative, power-associative and alternative. In Tab. \ref{multiplication-table-trace-norm-formulas-c} it is reported the multiplication table with respect to a basis $\{e_\alpha\}_{0\le\alpha\le 7}$, with $e_0=1\in\mathbb{R}$, ({\em unit octonions}), i.e., any $q\in \mathbb{O}$ has the following linear representation $q=\sum_{0\le\alpha\le 7}x^\alpha\, e_\alpha$, $x^\alpha\in\mathbb{R}$.\footnote{There are 480 isomorphic octonionic algebras generated by different multiplication tables. The group $G_2=Aut(\mathbb{O})$ is a simply connected, compact real Lie group, $\dim_{\mathbb{R}}G_2=4$. One has $\mathbb{O}\cong \mathbb{H}\times\mathbb{H}$, with multiplication $(a,b)(c,d)=(ac-d{\bar b},da+b{\bar c})$. The corresponding representation of the basis is the following: $e_0=(1,0)$, $e_1=(i,0)$, $e_2=(j,0)$, $e_3=(k,0)$, $e_4=(0,1)$, $e_5=(0,i)$, $e_6=(0,j)$, $e_7=(0,k)$.}

\begin{table}
\caption{Multiplication table, trace and norm formulas for octonionic algebra $\mathbb{O}\cong \mathbb{R}^8$ with basis $(e_\alpha)_{0\le\alpha\le 7}$.}
\label{multiplication-table-trace-norm-formulas-c}
\begin{tabular}{|l|l|l|l|l|l|l|l|}
\hline
\hfil{\rm{\footnotesize }}\hfil&\hfil{\rm{\footnotesize $ e_1$}}\hfil&\hfil{\rm{\footnotesize $ e_2$}}\hfil&\hfil{\rm{\footnotesize $ e_3$}}&{\rm{\footnotesize $e_4$}}&{\rm{\footnotesize $ e_5$}}&{\rm{\footnotesize $ e_6$}}\hfil&\hfil{\rm{\footnotesize $ e_7$}}\hfil\\
\hline
{\rm{\footnotesize $ e_1$}}\hfil&\hfil{\rm{\footnotesize $ -1$}}\hfil&\hfil{\rm{\footnotesize $ e_3$}}\hfil&\hfil{\rm{\footnotesize $ -e_2$}}\hfil&\hfil{\rm{\footnotesize  $ e_5$}}\hfil&\hfil{\rm{\footnotesize $ -e_4$}}\hfil&\hfil{\rm{\footnotesize $ -e_7$}}\hfil&\hfil{\rm{\footnotesize $ e_6$}}\hfil\\
\hline
{\rm{\footnotesize $ e_2$}}\hfil&\hfil{\rm{\footnotesize $ -e_3$}}\hfil&\hfil{\rm{\footnotesize $ -1$}}\hfil&\hfil{\rm{\footnotesize $ e_1$}}\hfil&\hfil{\rm{\footnotesize $ e_6$}}\hfil&\hfil{\rm{\footnotesize $ e_7$}}\hfil&\hfil{\rm{\footnotesize $ -e_4$}}\hfil&\hfil{\rm{\footnotesize $ -e_5$}}\\
\hline
{\rm{\footnotesize $ e_3$}}\hfil&\hfil{\rm{\footnotesize $ e_2$}}\hfil&\hfil{\rm{\footnotesize $ -e_1$}}\hfil&\hfil{\rm{\footnotesize $ -1$}}\hfil&\hfil{\rm{\footnotesize $ e_7$}}\hfil&\hfil{\rm{\footnotesize $ -e_6$}}\hfil&\hfil{\rm{\footnotesize $ e_5$}}\hfil&\hfil{\rm{\footnotesize $ -e_4$}}\hfil\\
\hline
{\rm{\footnotesize $ e_4$}}\hfil&\hfil{\rm{\footnotesize $ -e_5$}}\hfil&\hfil{\rm{\footnotesize $ -e_6$}}\hfil&\hfil{\rm{\footnotesize $ -e_7$}}\hfil&\hfil{\rm{\footnotesize $ -1$}}\hfil&\hfil{\rm{\footnotesize $ e_1$}}\hfil&\hfil{\rm{\footnotesize $ e_2$}}\hfil&\hfil{\rm{\footnotesize $ e_3$}}\hfil\\
\hline
\hfil {\rm{\footnotesize $ e_5$}}\hfil &\hfil {\rm{\footnotesize $ e_4$}}\hfil &\hfil {\rm{\footnotesize $ -e_7$}}\hfil &\hfil {\rm{\footnotesize $ e_6$}}\hfil &\hfil {\rm{\footnotesize $ -e_1$}}\hfil &\hfil {\rm{\footnotesize $ -1$}}\hfil &\hfil {\rm{\footnotesize $ -e_3$}}\hfil &\hfil {\rm{\footnotesize $ e_2$}}\hfil \\
\hline
{\rm{\footnotesize $ e_6$}}\hfil&\hfil{\rm{\footnotesize $ e_7$}}\hfil&\hfil{\rm{\footnotesize $ e_4$}}\hfil&\hfil{\rm{\footnotesize $ -e_5$}}\hfil&\hfil{\rm{\footnotesize  $ -e_2$}}\hfil&\hfil{\rm{\footnotesize $ e_3$}}\hfil&\hfil{\rm{\footnotesize $ -1$}}\hfil&\hfil{\rm{\footnotesize $ -e_1$}}\hfil\\
\hline
{\rm{\footnotesize $ e_7$}}\hfil&\hfil{\rm{\footnotesize $ -e_6$}}\hfil&\hfil{\rm{\footnotesize $ e_5$}}&{\rm{\footnotesize $ e_4$}}\hfil&\hfil{\rm{\footnotesize $ -e_3$}}\hfil&\hfil{\rm{\footnotesize $ -e_2$}}\hfil&\hfil{\rm{\footnotesize $ e_1$}}\hfil&\hfil{\rm{\footnotesize $ -1$}}\hfil\\
\hline
\multicolumn{8}{l}{\rm{\footnotesize $e_ie_j=-\delta_{ij}e_0+\epsilon_{ijk}e_k $, $1\le i,j,k\le 7$.}}\\
\multicolumn{8}{l}{\rm{\footnotesize $e_ie_j=-e_je_i$. $(e_ie_j)e_k=-e_i(e_je_k)$.}}\\
\multicolumn{8}{l}{\rm{\footnotesize $x=\sum_{0\le\alpha\le 7}x^\alpha e_\alpha=x^0+\sum_{1\le k\le 7}x^k e_k$.}}\\
\multicolumn{8}{l}{\rm{\footnotesize $\bar x=x^0-\sum_{1\le k\le 7}x^k e_k$.}}\\
\multicolumn{8}{l}{\rm{\footnotesize $\Re(x)\equiv \frac{1}{2}(x+\bar x)$; $\Im(x)\equiv \frac{1}{2}(x-\bar x)$.}}\\
\multicolumn{8}{l}{\rm{\footnotesize Norm: $N(x)\equiv||x||=\sqrt{\bar x x}$.}}\\
\multicolumn{8}{l}{\rm{\footnotesize Square root: $||x||^2=\bar x x=\sum_{0\le\alpha\le 7}(x^\alpha)^2>0$.}}\\
\multicolumn{8}{l}{\rm{\footnotesize Inverse: $x^{-1}=\bar x/||x||^2$; $x^{-1}x=x^{-1}=1$.}}\\
\end{tabular}
\end{table}
An example of quantum octonionic manifold is the {\em quantum octonionic M\"obius strip}, that can be obtained similarly to the quantum quaternionic manifold, just considered in Example \ref{quantum-quaternionic-mobius-strip}. In fact it is enough to copy the construction given in Example \ref{quantum-quaternionic-mobius-strip}, after substitueted $\mathbb{H}$, with $\mathbb{O}$, to get a quantum octonionic manifold. Similarly we can obtain the quantum octonionic heat equation as given in (\ref{quantum-octonionic-heat-equation}).

\begin{equation}\label{quantum-octonionic-heat-equation}
\scalebox{0.8}{${\left\{
\begin{array}{l}
\pi:W\equiv\mathbb{O}^3\to\mathbb{O}^2\equiv M;\quad (t,x,u)\mapsto (t,x)\\
\widehat{(HE)}_{\mathbb{O}}\subset J{\it D}^2(W)\subset  J^2_2(W):\quad u_{xx}-u_t=0\\
\end{array}
\right\},\quad
\Omega_1^{{\widehat{(HE)}}_{\mathbb{O}}}\cong H_1(W;\mathbb{R})\otimes_\mathbb{R}\mathbb{O}\cong 0.}$}
\end{equation}

We can see that the set $\mathcal{S}ol(\widehat{(HE)}_{\mathbb{O}})$ of solutions of $\widehat{(HE)}_{\mathbb{O}}$ contains also octonionic manifolds, i.e., affine quantum octonionic solutions. $X\subset\mathbb{O}^2\equiv M$ can be embedded into ${\widehat{(HE)}}_{\mathbb{O}}$ by means of the second holonomic prolongation of the zero section $u\equiv 0:M\to W$. In fact, ${\widehat{(HE)}}_{\mathbb{O}}$ is a linear equation. Therefore, $X^{(2)}\equiv D^2u(X)$ is a $1$-dimensional smooth closed compact admissible integral manifold contained into ${\widehat{(HE)}}_{\mathbb{O}}$, that is the boundary of a $2$-dimensional integral admissible manifold contained into ${\widehat{(HE)}}_{\mathbb{O}}$ too. This last is also a octonionic manifold. Moreover, all the regular solutions of ${\widehat{(HE)}}_{\mathbb{O}}\subset J{\it D}^2(W)$ are octonionic manifolds, as they are diffeomorphic to $\mathbb{O}^2$. However, no all the regular solutions of ${\widehat{(HE)}}_{\mathbb{O}}\subset J^2_2(W)$ are necessarily octonionic manifolds too.

We are ready now to state one of main results of this paper.

\begin{theorem}\label{theorem-quantum-octonionic-heat-equation}
Let $B$ be a quantum algebra such that its centre $Z(B)$ is a Noetherian $\mathbb{R}$-algebra. Let $\hat E_k\subset J\hat{\it D}^k(W)$ be a quantum regular QPDE in the category $\mathcal{C}^B_\mathbb{O}$, where $\pi:W\to M$ is a fibre bundle with $dim_CM=m$, $C\equiv B\otimes_\mathbb{R}\mathbb{O}$. If $\dot g_{k+r+1}$ is a bundle of $Z(C)$-modules over $\hat E_k$, and $\hat E_{k+r+1}\to\hat E_{k+r}$ is surjective for $0\le r\le m$, then $\hat E_k$ is formally quantumintegrable. In such a case, and further assuming that $W$ is $p$-connected, $p\in\{0,\dots,m-1\}$, then the integral bordism groups of $\hat E_k\subset\hat J^k_m(W)$ are given in {\em(\ref{integral-bordism-groups-quantum-octonionic-heat-equation})}.

\begin{equation}\label{integral-bordism-groups-quantum-octonionic-heat-equation}
\Omega_p^{\hat E_k}\cong H_p(W;\mathbb{R})\otimes_\mathbb{R}C,\quad 0\le p\le m-1.
\end{equation}

All the regular solutions of $\hat E_k\subset \hat J^k_m(W)$ are quantum $B$-octonionic submanifolds of $\hat E_k$ of dimension $m$, over $C$, identified with $m$-dimensional quantum $B$-octonionic submanifolds of $W$.
\end{theorem}
\end{example}

\begin{example}[Quantum sedenionic heat equation]
{\em Sedenionic algebra}, $\mathbb{S}$, is a $16$-dimensional non-commutative and non-associative $\mathbb{R}$-algebra, that is power-associative (but not even alternative), where any $q\in\mathbb{S}$ has the following linear representation $q=\sum_{0\le\alpha\le 15}x^\alpha\, e_\alpha$, $x^\alpha\in\mathbb{R}$, where the basis $\{e_\alpha\}_{0\le\alpha\le 15}$, ({\em unit sedenions}), has $e_0=1\in\mathbb{R}$, with multiplication table reported in Tab. \ref{multiplication-table-trace-norm-formulas-d}. $\mathbb{S}$ is not a division algebra since there are zero divisors, i.e., there are non-zero $a,b\in \mathbb{S}$ such that their product is zero: $ab=0$.\footnote{For example $(e_3+e_{10})(e_6-e_{15})=0$. Note that all algebras $\mathbb{A}_r$ in the Cayley-Dickson construction contain zero divisors, when $r\ge 4$. Recall that all the Cayley-Dickson algebras $\mathbb{A}_r$, with $0\le r\le 2$, are associative, division algebras, hence have no zero divisors. [A finite-dimensional unital associative algebra (over a field) is a division algebra iff has no zero divisors.] Let us emphasize that a quantum algebra $A$ is not, in general, a division algebra. Therefore, in general, a quantum hypercomplex algebra $\mathcal{Q}_r$, $r\ge 0$, is not a division algebra.}

\begin{table}
\caption{Multiplication table, trace and norm formulas for sedenionic algebra $\mathbb{S}\cong \mathbb{R}^{16}$ with basis $(e_\alpha)_{0\le\alpha\le 15}$.}
\label{multiplication-table-trace-norm-formulas-d}
\scalebox{0.7}{$\begin{tabular}{|l|l|l|l|l|l|l|l|l|l|l|l|l|l|l|l|l|}
\hline
\hfil{\rm{\footnotesize }}\hfil&\hfil{\rm{\footnotesize $ e_0$}}\hfil&\hfil{\rm{\footnotesize $ e_1$}}\hfil&\hfil{\rm{\footnotesize $ e_2$}}\hfil&\hfil{\rm{\footnotesize $ e_3$}}\hfil&\hfil{\rm{\footnotesize $e_4$}}\hfil&\hfil{\rm{\footnotesize $ e_5$}}\hfil&\hfil{\rm{\footnotesize $ e_6$}}\hfil&\hfil{\rm{\footnotesize $ e_7$}}\hfil&\hfil{\rm{\footnotesize $e_8$}}\hfil&\hfil{\rm{\footnotesize $ e_9$}}\hfil&\hfil{\rm{\footnotesize $ e_{10}$}}\hfil&\hfil{\rm{\footnotesize $ e_{11}$}}\hfil&\hfil{\rm{\footnotesize $e_{12}$}}\hfil&\hfil{\rm{\footnotesize $ e_{13}$}}\hfil&\hfil{\rm{\footnotesize $ e_{14}$}}\hfil&\hfil{\rm{\footnotesize $ e_{15}$}}\hfil\\
\hline
{\rm{\footnotesize  $ e_0$}}\hfil&\hfil{\rm{\footnotesize  $ 1$}}\hfil&\hfil{\rm{\footnotesize  $ e_1$}}\hfil&\hfil{\rm{\footnotesize  $ e_2$}}\hfil&\hfil{\rm{\footnotesize  $ e_3$}}\hfil&\hfil{\rm{\footnotesize $ e_4$}}\hfil&\hfil{\rm{\footnotesize $ e_5$}}\hfil&\hfil{\rm{\footnotesize $ e_6$}}\hfil&\hfil{\rm{\footnotesize $e_7$}}\hfil&\hfil{\rm{\footnotesize $ e_8$}}\hfil&\hfil{\rm{\footnotesize $ e_9$}}\hfil&\hfil{\rm{\footnotesize $ e_{10}$}}\hfil&\hfil{\rm{\footnotesize $e_{11}$}}\hfil&\hfil{\rm{\footnotesize $ e_{12}$}}\hfil&\hfil{\rm{\footnotesize $ e_{13}$}}\hfil&\hfil{\rm{\footnotesize $ e_{14}$}}\hfil&\hfil{\rm{\footnotesize $ e_{15}$}}\hfil\\
\hline
{\rm{\footnotesize $ e_1$}}\hfil&\hfil{\rm{\footnotesize $ e_1$}}\hfil&\hfil{\rm{\footnotesize $ -1$}}\hfil&\hfil{\rm{\footnotesize $ e_3$}}\hfil&\hfil{\rm{\footnotesize $ -e_2$}}\hfil&\hfil{\rm{\footnotesize $ e_5$}}\hfil&\hfil{\rm{\footnotesize $ -e_4$}}\hfil&\hfil{\rm{\footnotesize $ -e_7$}}\hfil&\hfil{\rm{\footnotesize $e_6$}}\hfil&\hfil{\rm{\footnotesize $ e_9$}}\hfil&\hfil{\rm{\footnotesize $ -e_8$}}\hfil&\hfil{\rm{\footnotesize $ -e_{11}$}}\hfil&\hfil{\rm{\footnotesize $e_{10}$}}\hfil&\hfil{\rm{\footnotesize $ -e_{13}$}}\hfil&\hfil{\rm{\footnotesize $ e_{12}$}}\hfil&\hfil{\rm{\footnotesize $ e_{15}$}}\hfil&\hfil{\rm{\footnotesize $-e_{14}$}}\hfil\\
\hline
{\rm{\footnotesize $ e_2$}}\hfil&\hfil{\rm{\footnotesize $ e_2$}}\hfil&\hfil{\rm{\footnotesize $ -e_3$}}\hfil&\hfil{\rm{\footnotesize $ -1$}}\hfil&\hfil{\rm{\footnotesize $ e_1$}}\hfil&\hfil{\rm{\footnotesize $ e_6$}}\hfil&\hfil{\rm{\footnotesize $ e_7$}}\hfil&\hfil{\rm{\footnotesize $ -e_4$}}\hfil&\hfil{\rm{\footnotesize $-e_5$}}\hfil&\hfil{\rm{\footnotesize $ e_{10}$}}\hfil&\hfil{\rm{\footnotesize $ e_{11}$}}\hfil&\hfil{\rm{\footnotesize $ -e_3$}}\hfil&\hfil{\rm{\footnotesize $-e_9$}}\hfil&\hfil{\rm{\footnotesize $ -e_{14}$}}\hfil\hfil&{\rm{\footnotesize $ -e_{15}$}}\hfil&\hfil{\rm{\footnotesize $ e_{12}$}}\hfil&\hfil{\rm{\footnotesize $ e_{13}$}}\hfil\\
\hline
{\rm{\footnotesize $ e_3$}}\hfil&\hfil{\rm{\footnotesize $ e_3$}}\hfil&\hfil{\rm{\footnotesize $ e_2$}}\hfil&\hfil{\rm{\footnotesize $ -e_1$}}\hfil&\hfil{\rm{\footnotesize $ -1$}}\hfil&\hfil{\rm{\footnotesize $ e_7$}}\hfil&\hfil{\rm{\footnotesize $ -e_6$}}\hfil&\hfil{\rm{\footnotesize $ e_5$}}\hfil&\hfil{\rm{\footnotesize $-e_4$}}\hfil&\hfil{\rm{\footnotesize $ e_{11}$}}\hfil&\hfil{\rm{\footnotesize $ -e_{10}$}}\hfil&\hfil{\rm{\footnotesize $ e_9$}}\hfil&\hfil{\rm{\footnotesize $-e_8$}}\hfil&\hfil{\rm{\footnotesize $ -e_{15}$}}\hfil&\hfil{\rm{\footnotesize $ e_{14}$}}\hfil&\hfil{\rm{\footnotesize $ -e_{13}$}}\hfil&\hfil{\rm{\footnotesize $ e_{12}$}}\hfil\\
\hline
 {\rm{\footnotesize $ e_4$}}\hfil &\hfil {\rm{\footnotesize $ e_4$}}\hfil &\hfil {\rm{\footnotesize $ -e_5$}}\hfil &\hfil {\rm{\footnotesize $ -e_6$}}\hfil &\hfil {\rm{\footnotesize $ -e_7$}}\hfil &\hfil {\rm{\footnotesize $ -1$}}\hfil &\hfil {\rm{\footnotesize $ e_1$}}\hfil &\hfil {\rm{\footnotesize $ e_2$}}\hfil&\hfil{\rm{\footnotesize $ e_3$}}\hfil&\hfil{\rm{\footnotesize $ e_{12}$}}\hfil&\hfil{\rm{\footnotesize $ e_{13}$}}\hfil&\hfil{\rm{\footnotesize $e_{14}$}}\hfil&\hfil{\rm{\footnotesize $ e_{15}$}}\hfil&\hfil{\rm{\footnotesize $ -e_8$}}\hfil&\hfil{\rm{\footnotesize $ -e_9$}}\hfil&\hfil{\rm{\footnotesize $ -e_{10}$}}\hfil&\hfil{\rm{\footnotesize $ -e_{11}$}}\hfil\\
\hline
{\rm{\footnotesize  $ e_5$}}\hfil&\hfil{\rm{\footnotesize  $ e_5$}}\hfil&\hfil{\rm{\footnotesize  $ e_4$}}\hfil&\hfil{\rm{\footnotesize  $ -e_7$}}\hfil&\hfil{\rm{\footnotesize  $ e_6$}}\hfil&\hfil{\rm{\footnotesize  $ -e_1$}}\hfil&\hfil{\rm{\footnotesize  $ -1$}}\hfil&\hfil{\rm{\footnotesize  $ e_3$}}\hfil&\hfil{\rm{\footnotesize $e_2$}}\hfil&\hfil{\rm{\footnotesize $ e_{13}$}}\hfil&\hfil{\rm{\footnotesize $ -e_{12}$}}\hfil&\hfil{\rm{\footnotesize $ e_{15}$}}\hfil&\hfil{\rm{\footnotesize $-e_{14}$}}\hfil&\hfil{\rm{\footnotesize $ e_9$}}\hfil&\hfil{\rm{\footnotesize $ -e_8$}}\hfil&\hfil{\rm{\footnotesize $ e_{11}$}}\hfil&\hfil{\rm{\footnotesize $ -e_{10}$}}\hfil\\
\hline
{\rm{\footnotesize  $ e_6$}}\hfil&\hfil{\rm{\footnotesize  $ e_6$}}\hfil&\hfil{\rm{\footnotesize  $ e_7$}}\hfil&\hfil{\rm{\footnotesize  $ e_4$}}\hfil&\hfil{\rm{\footnotesize $ -e_5$}}\hfil&\hfil{\rm{\footnotesize $ -e_2$}}\hfil&\hfil{\rm{\footnotesize $ e_3$}}\hfil&\hfil{\rm{\footnotesize $ -1$}}\hfil&\hfil{\rm{\footnotesize $-e_1$}}\hfil&\hfil{\rm{\footnotesize $ e_{14}$}}\hfil&\hfil{\rm{\footnotesize $ -e_{15}$}}\hfil&\hfil{\rm{\footnotesize $ -e_{12}$}}\hfil&\hfil{\rm{\footnotesize $e_{13}$}}\hfil&\hfil{\rm{\footnotesize $ e_{10}$}}\hfil&\hfil{\rm{\footnotesize $ -e_{11}$}}\hfil&\hfil{\rm{\footnotesize $ -e_3$}}\hfil&\hfil{\rm{\footnotesize $ e_9$}}\hfil\\
\hline
{\rm{\footnotesize $e_7$}}\hfil&\hfil{\rm{\footnotesize $ e_7$}}\hfil&\hfil{\rm{\footnotesize $ -e_6$}}\hfil&\hfil{\rm{\footnotesize $ e_5$}}\hfil&\hfil{\rm{\footnotesize $e_{4}$}}\hfil&\hfil{\rm{\footnotesize $ -e_{3}$}}\hfil&\hfil{\rm{\footnotesize $ -e_{2}$}}\hfil&\hfil{\rm{\footnotesize $ e_{1}$}}\hfil&\hfil{\rm{\footnotesize $-1$}}\hfil&\hfil{\rm{\footnotesize $ e_{15}$}}\hfil&\hfil{\rm{\footnotesize $ e_{14}$}}\hfil&\hfil{\rm{\footnotesize $ -e_{13}$}}\hfil&\hfil{\rm{\footnotesize $-e_{12}$}}\hfil&\hfil{\rm{\footnotesize $ e_{11}$}}\hfil&\hfil{\rm{\footnotesize $ e_{10}$}}\hfil&\hfil{\rm{\footnotesize $ -e_9$}}\hfil&\hfil{\rm{\footnotesize $ -e_8$}}\hfil\\
\hline
{\rm{\footnotesize $ e_8$}}\hfil&\hfil{\rm{\footnotesize $ e_{8}$}}\hfil&\hfil{\rm{\footnotesize $ -e_9$}}\hfil&\hfil{\rm{\footnotesize $ -e_{10}$}}\hfil&\hfil{\rm{\footnotesize $ -e_{11}$}}\hfil&\hfil{\rm{\footnotesize $ -e_{12}$}}\hfil&\hfil{\rm{\footnotesize $ -e_{13}$}}\hfil&\hfil{\rm{\footnotesize $ -e_{14}$}}\hfil&\hfil{\rm{\footnotesize $-e_{15}$}}\hfil&\hfil{\rm{\footnotesize $ -1$}}\hfil&\hfil{\rm{\footnotesize $ e_1$}}\hfil&\hfil{\rm{\footnotesize $ e_2$}}\hfil&\hfil{\rm{\footnotesize $e_3$}}\hfil&\hfil{\rm{\footnotesize $ e_4$}}\hfil&\hfil{\rm{\footnotesize $ e_5$}}\hfil&\hfil{\rm{\footnotesize $ e_6$}}\hfil&\hfil{\rm{\footnotesize $ e_7$}}\hfil\\
\hline
{\rm{\footnotesize $ e_9$}}\hfil&\hfil{\rm{\footnotesize $ e_9$}}\hfil&\hfil{\rm{\footnotesize $ e_{8}$}}\hfil&\hfil{\rm{\footnotesize $ -e_{11}$}}\hfil&\hfil{\rm{\footnotesize $ e_{10}$}}\hfil&\hfil{\rm{\footnotesize $ -e_{13}$}}\hfil&\hfil{\rm{\footnotesize $ e_{12}$}}\hfil&\hfil{\rm{\footnotesize $ e_{15}$}}\hfil&\hfil{\rm{\footnotesize $-e_{14}$}}\hfil&\hfil{\rm{\footnotesize $ -e_{1}$}}\hfil&\hfil{\rm{\footnotesize $-1$}}\hfil&\hfil{\rm{\footnotesize $ -e_3$}}\hfil&\hfil{\rm{\footnotesize $e_{2}$}}\hfil&\hfil{\rm{\footnotesize $ -e_5$}}\hfil&\hfil{\rm{\footnotesize $ e_4$}}\hfil&\hfil{\rm{\footnotesize $ e_7$}}\hfil&\hfil{\rm{\footnotesize $ -e_6$}}\hfil\\
\hline
{\rm{\footnotesize $ e_{10}$}}\hfil&\hfil{\rm{\footnotesize $ e_{10}$}}\hfil&\hfil{\rm{\footnotesize $ e_{11}$}}\hfil&\hfil{\rm{\footnotesize $ e_3$}}\hfil&\hfil{\rm{\footnotesize $ -e_{9}$}}\hfil&\hfil{\rm{\footnotesize $ -e_{14}$}}\hfil&\hfil{\rm{\footnotesize $ -e_{15}$}}\hfil&\hfil{\rm{\footnotesize  $ e_{12}$}}\hfil&\hfil{\rm{\footnotesize $ e_{13}$}}\hfil&\hfil{\rm{\footnotesize $-e_2$}}\hfil&\hfil{\rm{\footnotesize $ e_3$}}\hfil&\hfil{\rm{\footnotesize $ -1$}}\hfil&\hfil{\rm{\footnotesize $ -e_1$}}\hfil&\hfil{\rm{\footnotesize $-e_6$}}\hfil&\hfil{\rm{\footnotesize $ -e_7$}}\hfil&\hfil{\rm{\footnotesize $ e_4$}}\hfil&\hfil{\rm{\footnotesize $ e_5$}}\hfil\\
\hline
{\rm{\footnotesize $ e_{11}$}}\hfil&\hfil{\rm{\footnotesize $ e_{11}$}}\hfil&\hfil{\rm{\footnotesize $ -e_{10}$}}\hfil&\hfil{\rm{\footnotesize $ e_9$}}\hfil&\hfil{\rm{\footnotesize $ e_8$}}\hfil&\hfil{\rm{\footnotesize $ -e_{15}$}}\hfil&\hfil{\rm{\footnotesize $ e_{14}$}}\hfil&\hfil{\rm{\footnotesize $ -e_{13}$}}\hfil&\hfil{\rm{\footnotesize $ e_{12}$}}\hfil&\hfil{\rm{\footnotesize $-e_{3}$}}\hfil&\hfil{\rm{\footnotesize $ -e_{2}$}}\hfil&\hfil{\rm{\footnotesize $ e_1$}}\hfil&\hfil{\rm{\footnotesize $ -1$}}\hfil&\hfil{\rm{\footnotesize $-e_7$}}\hfil&\hfil{\rm{\footnotesize $ e_6$}}\hfil&\hfil{\rm{\footnotesize $ -e_5$}}\hfil&\hfil{\rm{\footnotesize $ e_4$}}\hfil\\
\hline
\hfil {\rm{\footnotesize $ e_{12}$}}\hfil &\hfil {\rm{\footnotesize $ e_{12}$}}\hfil &\hfil {\rm{\footnotesize $ e_{13}$}}\hfil &\hfil {\rm{\footnotesize $ e_{14}$}}\hfil &\hfil {\rm{\footnotesize $ e_{15}$}}\hfil &\hfil {\rm{\footnotesize $ e_{8}$}}\hfil &\hfil {\rm{\footnotesize $ -e_9$}}\hfil &\hfil {\rm{\footnotesize $ e_{10}$}}\hfil&\hfil{\rm{\footnotesize $-e_{11}$}}\hfil&\hfil{\rm{\footnotesize $ -e_4$}}\hfil&\hfil{\rm{\footnotesize $ e_5$}}\hfil&\hfil{\rm{\footnotesize $ e_6$}}\hfil&\hfil{\rm{\footnotesize $e_7$}}\hfil&\hfil{\rm{\footnotesize $ -1$}}\hfil&\hfil{\rm{\footnotesize $ -e_1$}}\hfil&\hfil{\rm{\footnotesize $ -e_2$}}\hfil&\hfil{\rm{\footnotesize $ -e_3$}}\hfil\\
\hline
{\rm{\footnotesize $ e_{13}$}}\hfil&\hfil{\rm{\footnotesize $ e_{13}$}}\hfil&\hfil{\rm{\footnotesize $ -e_{12}$}}\hfil&\hfil{\rm{\footnotesize $ e_{15}$}}\hfil&\hfil{\rm{\footnotesize $ -e_{14}$}}\hfil&\hfil{\rm{\footnotesize $ e_9$}}\hfil&\hfil{\rm{\footnotesize $ e_{8}$}}\hfil&\hfil{\rm{\footnotesize $ e_{11}$}}\hfil&\hfil{\rm{\footnotesize$-e_{10}$ }}\hfil&\hfil{\rm{\footnotesize $ -e_5$}}\hfil&\hfil{\rm{\footnotesize $ -e_4$}}\hfil&\hfil{\rm{\footnotesize $ e_7$\hfil}}\hfil&\hfil{\rm{\footnotesize $-e_6$}}\hfil&\hfil{\rm{\footnotesize $ e_1$}}\hfil&\hfil{\rm{\footnotesize $ -1$}}\hfil&\hfil{\rm{\footnotesize $ -e_3$}}\hfil&\hfil{\rm{\footnotesize $ -e_2$}}\hfil\\
\hline
{\rm{\footnotesize  $ e_{14}$}}\hfil&\hfil{\rm{\footnotesize $ e_{14}$}}\hfil&\hfil{\rm{\footnotesize $ -e_{15}$}}\hfil&\hfil{\rm{\footnotesize $ -e_{12}$}}\hfil&\hfil{\rm{\footnotesize $ e_{13}$}}\hfil&\hfil{\rm{\footnotesize $ e_{10}$}}\hfil&\hfil{\rm{\footnotesize $ -e_{11}$}}\hfil&\hfil{\rm{\footnotesize $ e_{8}$}}\hfil&\hfil{\rm{\footnotesize$e_{9}$}}\hfil&\hfil{\rm{\footnotesize $ -e_6$}}\hfil&\hfil{\rm{\footnotesize $ -e_7$}}\hfil&\hfil{\rm{\footnotesize $ -e_4$}}\hfil&\hfil{\rm{\footnotesize $e_5$}}\hfil&\hfil{\rm{\footnotesize $ e_2$}}\hfil&\hfil{\rm{\footnotesize $ -e_3$}}\hfil&\hfil{\rm{\footnotesize $-1$}}\hfil&\hfil{\rm{\footnotesize $e_{1}$}}\hfil\\
\hline
{\rm{\footnotesize  $ e_{15}$}}\hfil&\hfil{\rm{\footnotesize $ e_{15}$}}\hfil&\hfil{\rm{\footnotesize $ e_{14}$}}\hfil&\hfil{\rm{\footnotesize $ -e_{13}$}}\hfil&\hfil{\rm{\footnotesize $ -e_{12}$}}\hfil&\hfil{\rm{\footnotesize $ e_{11}$}}\hfil&\hfil{\rm{\footnotesize $ e_{10}$}}\hfil&\hfil{\rm{\footnotesize $ -e_{9}$}}\hfil&\hfil{\rm{\footnotesize$e_{8}$}}\hfil&\hfil{\rm{\footnotesize $ -e_7$}}\hfil&\hfil{\rm{\footnotesize $ e_6$}}\hfil&\hfil{\rm{\footnotesize $ -e_5$}}\hfil&\hfil{\rm{\footnotesize $-e_4$}}\hfil&\hfil{\rm{\footnotesize $ e_3$}}\hfil&\hfil{\rm{\footnotesize $ e_2$}}\hfil&\hfil{\rm{\footnotesize $-e_{1}$}}\hfil&\hfil{\rm{\footnotesize $-1$}}\hfil\\
\hline
\multicolumn{8}{l}{\rm{\footnotesize $\mathbb{S}=\mathbb{A}_4$ in the Cayley-Dickson construction.}}\\
\end{tabular}$}
\end{table}
Similarly to the two above examples, we can define and characterize {\em quantum sedenionic manifolds}, and quantum sedenionic PDE's. In particular we get a natural extension, in the category of quantum sedenionic manifolds, of Theorem \ref{theorem-quantum-octonionic-heat-equation}. In fact, we get the following theorem.

\begin{theorem}\label{theorem-quantum-sedenionic-heat-equation}
Let $B$ be a quantum algebra such that its centre $Z(B)$ is a Noetherian $\mathbb{R}$-algebra. Let $\hat E_k\subset J\hat{\it D}^k(W)$ be a quantum regular QPDE in the category $\mathcal{C}^B_\mathbb{S}$, where $\pi:W\to M$ is a fibre bundle with $dim_CM=m$, $C\equiv B\otimes_\mathbb{R}\mathbb{S}$. If $\dot g_{k+r+1}$ is a bundle of $Z(C)$-modules over $\hat E_k$, and $\hat E_{k+r+1}\to\hat E_{k+r}$ is surjective for $0\le r\le m$, then $\hat E_k$ is formally quantumintegrable. In such a case, and further assuming that $W$ is $p$-connected, $p\in\{0,\dots,m-1\}$, then the integral bordism groups of $\hat E_k\subset\hat J^k_m(W)$ are given in {\em(\ref{integral-bordism-groups-quantum-sedenionic-heat-equation})}.

\begin{equation}\label{integral-bordism-groups-quantum-sedenionic-heat-equation}
\Omega_p^{\hat E_k}\cong H_p(W;\mathbb{R})\otimes_\mathbb{R}C,\quad 0\le p\le m-1.
\end{equation}

All the regular solutions of $\hat E_k\subset \hat J^k_m(W)$ are quantum $B$-sedenionic submanifolds of $\hat E_k$ of dimension $m$, over $C$, identified with $m$-dimensional quantum $B$-sedenionic submanifolds of $W$.
\end{theorem}

\end{example}

\section{Quantum hypercomplex singular (super) PDE's}\label{quantum-hypercomplex-singular-pdes-section}

In this section we shall consider singular quantum super PDE's
extending our previous theory of singular PDE's,\footnote{See Refs.\cite{PRA3, PRA17, PRA18-0}. See also \cite{AG-PRA2} where some interesting applications are considered.} i.e., by
considering singular quantum (super) PDE's as singular quantum
sub-(super)manifolds of jet-derivative spaces in the category
$\mathfrak{Q}$ or $\mathfrak{Q}_S$ or $\mathfrak{Q}_{hyper}$. In fact, our previous formal
theory of quantum (super) PDE's works well on quantum smooth or
quantum analytic submanifolds, since these regularity conditions are
necessary to develop such a theory. However, in many mathematical
problems and physical applications, it is necessary to work with
less regular structures, so it is useful to formulate a general
geometric theory for such more general quantum PDE's in the category
$\mathfrak{Q}_S$ or $\mathfrak{Q}_{hyper}$. Therefore, we shall assume that quantum singular
super PDE's are subsets of jet-derivative spaces where are presents
regular subsets, but also other ones where the conditions of
regularity are not satisfied. So the crucial point to investigate is
to obtain criteria that allow us to find existence theorems for
solutions crossing ''singular points'' and study their stability
properties. In some previous works we have considered quantum singular PDE's in the categories $\mathfrak{Q}$ and $\mathfrak{Q}_S$. Here we shall specialize on quantum singular PDE's in the category $\mathfrak{Q}_{S-hyper}$ of quantum hypercomplex supermanifolds. There the fundamental non-commutative algebras considered are of the type $\mathcal{Q}_r\equiv B\bigotimes_{\mathbb{R}}\mathbb{A}_r$, where $B$ is a quantum superalgebra in the sense of A. Pr\'astaro, and $\mathbb{A}_r$ is a Cayley-Dickson algebra, as just considered in the prevoius sections.

The main result of this section is Theorem
\ref{main-quantum-singular1} that relates singular integral bordism
groups of singular qunatum PDE's to global solutions passing through
singular points. Some example are explicitly considered.

Let us, now, first begin with a generalization of algebraic
formulation of quantum super PDE's, starting with the following
definitions. (See also Refs.\cite{PRA13, PRA19, PRA20, PRA21}.)

\begin{definition}
The {\em general category of quantum hypercomplex superdifferential equations},
${\frak Q}_{S-hyper}^{\frak E}$, is defined by the following: {\em 1)}
$\hat{\frak B}\in Ob({\frak Q}_{S-hyper}^{\frak E})$ iff $\hat{\frak B}$ is
a filtered quantum hypercomplex superalgebra $\hat{\frak B}\equiv \{ \hat{\frak
B}_i\} , \hat{\frak B}_i\subset \hat{\frak B}_{i+1}$, such that in
the differential calculus in the category ${\frak
Q}^{FG}_{S-hyper}(\hat{\frak B})$ over $\hat{\frak B}$ is defined a natural
operation ${\it C}$ that satisfies ${\it C}\hat \Omega ^1 \wedge
\hat\Omega^\bullet={\it C}\hat\Omega^\bullet$ , where
$\hat\Omega^i\equiv \hat{\frak B}\wedge \cdots_i\cdots
\wedge\hat{\frak B}$ are the representative objects of the functor
$\hat D_i$ in the category ${\frak Q}^{FG}_S(\hat{\frak B})$ over
$\hat{\frak B}$, where $\hat D_i \equiv\hat D\cdots_i\cdots\hat D$,
being $\hat D(P)$ the $\hat{\frak B}$-module of all quantum
superdifferentiations of algebra $\hat{\frak B}$ with values in
module $P$. Furthermore, $\hat\Omega^\bullet\equiv\bigoplus_{i\ge
0}\hat\Omega^i$, $\hat\Omega^0\equiv A$. {\em 2)} $f\in Hom({\frak
Q}_{S-hyper}^{\frak E})$ iff $f$ is a homomorphism of filtered quantum hypercomplex
superalgebras preserving operation ${\it C}$.
\end{definition}

\begin{remark}
In practice we shall take $\hat{\frak B}\equiv \{ \hat{\frak
B}_i\equiv Q^\infty _w(M_i;A)\}$, where $M_i$ is a quantum hypercomplex
supermanifold and $A$ is a quantum hypercomplex superalgebra. Then, we have a
canonical inclusion: $j_i:M_i\rightarrow Sp(\hat{\frak B}_i),
x\mapsto j_i(x)\equiv e_x\equiv$ evaluation map at $x\in M_i$. To
the inclusion $\hat{\frak B}_i\subset \hat{\frak B}_{i+1}$
corresponds the quantum smooth map $M_{i+1}\rightarrow M_i$. So we
set $M_\infty =\mathop{lim}\limits_{\leftarrow }M_i$. One has
${\overline M}_\infty =Sp(\hat{\frak B}_\infty )$. However, as
$M_\infty $ contains all the ''nice'' points of $Sp(\hat{\frak
B}_\infty )$, we shall use the space $M_\infty $ to denote an object
of the {\em category of quantum hypercomplex superdifferential equations}.
\end{remark}

\begin{definition}
The {\em category of quantum hypercomplex superdifferential equations}
$\underline{\frak Q}^{\frak E}_{S-hyper}$ is defined by the Frobenius full
quantum superdistribution $\widehat C(X)\subset \widehat TX\equiv
Hom_{Z}(A;TX)$, which is locally the same as
 $\widehat{\mathbf{E}}_\infty$, i.e., the
Cartan quantum superdistribution of $\hat E_\infty $ for some
quantum hypercomplex super PDE $\hat E_k\subset J\hat D^k(W)$. We set: $s\dim
X\equiv \dim\widehat C(X)=(m+n|m+n)$, i.e., the {\it Cartan quantum
superdimension} of $X\in Ob(\underline{\frak Q}^{\frak E}_{S-hyper})$. $f\in
Hom(\underline{\frak Q}^{\frak E}_{S-hyper})$ iff it is a quantum
supersmooth map $f:X\rightarrow Y$, where $X,Y\in
Ob(\underline{\frak Q}^{\frak E}_{S-hyper})$, such that conserves the
corresponding Frobenius full superdistributions: $\widehat
T(f):\widehat C(X)\rightarrow \widehat C(Y)$,
           $f\in Hom_{\underline{\frak Q}^{\frak E}_{S-hyper}}(X,Y)$,
           $sdimX=(m+n|m+n)$, $sdimY=(m'+n'|m'+n')$, $srankf=(r|s)=\dim(\hat T(f)_x
           (\widehat C(X)_x))$, $x\in X$.
Then the fibers $f^{-1}(y)$, $y\in im(f)\subset Y$, are
$(m+n-r|m+n-s)$-quantum superdimensional objects of
$\underline{\frak Q}^{\frak E}_{S-hyper}$. {\em Isomorphisms} of
$\underline{\frak Q}^{\frak E}_{S-hyper}$: quantum supermorphisms with
fibres consisting of separate points. {\em Covering maps} of
$\underline{\frak Q}^{\frak E}_{S-hyper}$: quantum supermorphims with
zero-quantum superdimensional fibres.
\end{definition}
\begin{example}[Some quantum hypercomplex singular PDE's]\label{examples-quantum-singular-PDEs}
\begin{table}[t]\centering
\caption{Some quantum hypercomplex singular PDE's defined by differential polynomials.}
\label{some-quantum-singular-pdes-defined-by-differential-polynomials}
\scalebox{0.8}{$\begin{tabular}{|l|l|}
\hline
{\rm{\footnotesize Name}}&{\rm{\footnotesize Singular PDE}}\\
\hline
{\rm{\footnotesize PDE with node and triple point}}&{\rm{\footnotesize $p_1\equiv(u^1_x)^4+(u^2_y)^4-(u^1_x)^2=0$}}\\
{\rm{\footnotesize $\hat R_1\subset \hat J{\it D}(E)$}}&{\rm{\footnotesize $p_2\equiv(u^2_x)^6+(u^1_y)^6-u^2_xu^1_y=0$}}\\
\hline
{\rm{\footnotesize PDE with cusp and tacnode}}&{\rm{\footnotesize $q_1\equiv(u^1_x)^4+(u^2_y)^4-(u^1_x)^3+(u^2_y)^2=0$}}\\
{\rm{\footnotesize such that $\hat S_1\subset \hat J{\it D}(E)$.}}&{\rm{\footnotesize such that $q_2\equiv(u^2_x)^4+(u^1_y)^4-(u^2_x)^2(u^1_y)-(u^2_x)(u^1_y)^2=0$.}}\\
\hline
{\rm{\footnotesize PDE with conical double point,}}&{\rm{\footnotesize $r_1\equiv(u^1)^2-(u^1_x)(u^2_y)^2=0$}}\\
{\rm{\footnotesize double line and pinch point}}&{\rm{\footnotesize $r_2\equiv(u^2)^2-(u^2_x)^2-(u^1_y)^2=0$}}\\
{\rm{\footnotesize $\hat T_1\subset J{\it D}(F)$}}&{\rm{\footnotesize $r_3\equiv(u^3)^3+(u_y^3)^3+(u^2_x)(u^3_y)=0$}}\\
\hline
\multicolumn{2}{l}{\rm{\footnotesize $\pi:E\equiv A^4\to
A^2$, $(x,y,u^1,u^2)\mapsto(x,y)$. $\bar\pi:F\equiv A^5\to A^2$, $(x,y,u^1,u^2,u^3)\mapsto(x,y)$.}}\\
\multicolumn{2}{l}{\rm{\footnotesize $\mathfrak{a}\equiv<p_1,p_2>\subset
\hat{\frak B}_1, \mathfrak{b}\equiv<q_1,q_2>\subset \hat{\frak B}_1,\mathfrak{c}\equiv<r_1,r_2,r_3>\subset\hat{\frak P}_1$.}}\\
\multicolumn{2}{l}{\rm{\footnotesize Here $A$ is a quantum hypercomplex algebra.}}\\
\end{tabular}$}
\end{table}

In Tab. \ref{some-quantum-singular-pdes-defined-by-differential-polynomials} we report some quantum singular PDE's having  some
algebraic singularities.  For the first two equations these are
quantum singular PDE's of first order defined on the quantum fiber
bundle $\pi:E\to M$, with $E\equiv A^4$, $M\equiv A^2$, where $A$ is
a quantum algebra. Then $\hat J{\it D}(E)\cong
B_1^{4,4}=A^4\times\widehat{A}^4$. Furthermore, for the third
equation one has the quantum fiber bundle $\bar\pi:F\to M$, with
$F\equiv A^5$, $M\equiv A^2$, and $\hat J{\it D}(F)\cong
B_1^{5,6}=A^5\times\widehat{A}^6$. We follow our usual notation
introduced in some previous works on the same subject. In particular
for a given quantum (super)algebra $A$, we put
\begin{equation}\label{dot-tensor-products}
\left\{\begin{array}{l}
\dot T^r_0(H)\equiv
\underbrace{H\otimes_Z\cdots\otimes_ZH}_r,\quad
r\ge0\\
\mathop{\widehat{A}}\limits^{r}\equiv Hom_Z(\dot T^r_0(A);A),\quad
r\ge0\\
\mathop{\widehat{A}}\limits^{0}\equiv Hom_Z(\dot T^0_0(A);A)\equiv
Hom_Z(A;A)\equiv \widehat{A}\\
\end{array}\right.
\end{equation}
with $Z$ the centre of $A$ and $H$ any $Z$-module.
Furthermore, we denote also by $\dot S^r_0(H)$ and $\dot
\Lambda^r_0(H)$ the corresponding symmetric and skewsymmetric
submodules of $\dot T^r_0(H)$. To the ideals $\mathfrak{a}\equiv
<p_1,p_2>\subset \hat{\frak B}_1$,
$\mathfrak{b}\equiv<q_1,q_2>\subset \hat{\frak B}_1$ and
$\mathfrak{c}\equiv<r_1,r_2,r_3>\subset \hat{\frak P}_1$, where
$\hat{\frak B}_1\equiv Q^\infty_w(\hat J{\it D}(E),B_2)$, with
$B_2\equiv A\times \widehat{A}\times\mathop{\widehat{A}}\limits^2$, and $\hat{\frak P}_1\equiv
Q^\infty_w(\hat J{\it D}(F),B_1)$, one associates the corresponding
algebraic sets $\hat R_1=\{q\in B_1^{4,4}|f(q)=0, \forall f\in
\mathfrak{a}\}\subset B_1^{4,4}$, $\hat S_1=\{q\in B_1^{4,4}|f(q)=0,
\forall f\in \mathfrak{b}\}\subset B_1^{4,4}$ and $\hat T_1=\{q\in
B_1^{5,6}|f(q)=0, \forall f\in \mathfrak{c}\}\subset B_1^{5,6}$.

Let us consider in some details, for example, the first equation in
Tab.5. There the node and the triple point refer to the singular
points in the planes $(u^1_x,u^2_y)$ and $(u^2_x,u^1_y)$
respectively, with respect to the $\mathbb{R}$-restriction. However, the equation $\hat R_1$ has a
set $\Sigma(\hat R_1)\subset \hat R_1$ of singular points that is larger than one reported in {\em(\ref{singular-points-eq-1-tab3})}.
\begin{equation}\label{singular-points-eq-1-tab3}
    \Sigma(\hat R_1)_0=\left\{
      q_0=(x,y,u^1,u^2,0,0,0,0)\right\}\cong A^4
    \subset \Sigma(\hat R_1)\subset\hat R_1.
\end{equation}
In fact the jacobian $(j(F)_{ij})$, $i=1,2$, $j=1,\cdots,8$, with
$(F_i)\equiv(p_1,p_2):\hat J{\it D}(E)\to B_2$, is given by the
following matrix with entries in the quantum algebra $B_2$:
\begin{equation}\label{jacobian}
(j(F)_{ij})=\left(
  \begin{array}{cccccccc}
    0 & 0 & 0 & 0 & 2u^1_x[2(u^1_x)^2-1] & 0 & 0 & 4(u^2_y)^3 \\
    0 & 0 & 0 & 0 & 0 & 6(u^1_y)^5-u_x^2 &6(u^2_x)^5-u_y^1  & 0 \\
  \end{array}
\right).
\end{equation}
Since, in general, the quantum hypercomplex algebra $A$ can have non-empty the set $Z_{ero}(A)$ of zero-divisors,\footnote{Let us emphasize that $Z_{ero}(A)=\bigcup_i\mathfrak{p}_i$, where $\mathfrak{p}_i$ is an {\em associated prime ideal} of $A$, i.e., a prime ideal that is the annihilator of some element of $A$. (Let us recall that an ideal $\mathfrak{p}$ of $A$ is called {\em prime} if $\mathfrak{p}\varsubsetneq A$ and whenever two ideals $\mathfrak{a},\, \mathfrak{b}\subset A$, are such that $\mathfrak{a}\mathfrak{b}\subset \mathfrak{p}$, then at least one of $\mathfrak{a}$ and $\mathfrak{b}$ is contained in $A$.) The $\mathbb{R}$-algebra $A$ is called {\em tame} if $Z_{ero}(A)=\bigoplus_{1\le i\le r<\infty}S_i$, where $S_i$ is a subspace of $A$. Note that even if $Z_{ero}(A)=\varnothing$, $Z_{ero}(A^2)\not=\varnothing$, since $(0,a)(b,0)=(0,0)=0\in A^2$.} in order $\hat Y_1\equiv \hat R_1\setminus \Sigma(\hat R_1)$ should represent $\hat R_1$ without singular points, i.e., in order to apply the implicit quantum function theorem (see Theorem 1.38 in \cite{PRA3}), it is enough to take the points $q\in\hat R_1$, where there are $2\times 2$ minors in {\em(\ref{jacobian})} with invertible determinant. In other words, we shall add the condition that at least one of determinants $\det_i$, $1\le i\le 4$, reported in {\em(\ref{determinants-minors-jacobian-first-example})} should be invertible elements of the quantum algebra $\mathop{\widehat{A}}\limits^s\equiv Hom_{Z(A)}(\dot T^s_0(A);A)$, for suitable $s\in\mathbb{N}$.\footnote{See the implicit quantum function theorem. (Theorem 1.38 in \cite{PRA3}.)}

\begin{equation}\label{determinants-minors-jacobian-first-example}
\begin{array}{l}
  \det_1\equiv\det\left(
                    \begin{array}{cc}
                      2u^1_x[2(u^1_x)^2-1] & 0\\
                      0 & 6(u^1_y)^5-u^2_x \\
                    \end{array}
                  \right)=2u^1_x[2(u^1_x)^2-1][6(u^1_y)^5-u^2_x]\in\mathop{\widehat{A}}\limits^3\\
  \det_2\equiv\det\left(
                    \begin{array}{cc}
                      2u^1_x[2(u^1_x)^2-1] & 0\\
                      0 & 6(u^2_x)^5-u^1_y \\
                    \end{array}
                  \right)=2u^1_x[2(u^1_x)^2-1][6(u^2_x)^5-u^1_y]\in\mathop{\widehat{A}}\limits^3\\
 \det_3\equiv\det\left(
                    \begin{array}{cc}
                     0&4(u^2_y)^3\\
                     6(u^1_y)^5-u^2_x&0\\
                    \end{array}
                  \right)=-[6(u^1_y)^5-u^2_x]4(u^2_y)^3\in\mathop{\widehat{A}}\limits^2\\
  \det_4\equiv\det\left(
                    \begin{array}{cc}
                     0&4(u^2_y)^3\\
                     6(u^2_x)^5-u^1_y&0\\
                    \end{array}
                  \right)=-[6(u^2_x)^5-u^1_y]4(u^2_y)^3\in\mathop{\widehat{A}}\limits^2.\\
  \end{array}
\end{equation}
Therefore, we must identify the sets ${}_i\hat X_1\equiv\{q\in\hat J{\it D}(E)\, |\,  \det_i(q)\in\mathop{\widehat{A}}\limits^{s(i)}(G(A))\}$, $1\le i\le 4$, with $s(1)=s(2)=3$ and $s(3)=s(4)=2$. Here $G(A)$ denotes the abelian group of unities in $A$ and $\mathop{\widehat{A}}\limits^{s(i)}(G(A))\equiv  Hom_{Z(A)}(\dot T^{s(i)}_0(A);G(A))\subset \mathop{\widehat{A}}\limits^{s(i)}$. Then taking into account that $\mathop{\widehat{A}}\limits^{s(i)}(G(A))$ is an open set in the quantum hypercomplex algebra $\mathop{\widehat{A}}\limits^{s(i)}$,\footnote{This a direct consequence of Lemma 3.32 in \cite{PRA18}. It is useful to emphasize that the abelian group $G(A)$ of a quantum (hypercomplex) algebra $A$, has no zero divisors, $Z_{ero}(A)=\varnothing$, hence it is a division algebra. (Left or right divisors can never be units.) Furthermore, any $b\in G(A)$, with $b\not=1$, cannot be idempotent, since an idempotent element must be a zero divisor: $b^2=b\, \Rightarrow\, b(b-1)=0$.} we get that $\hat X_1\equiv\bigcup_{1\le i\le 4}\det^{-1}_i(\mathop{\widehat{A}}\limits^{s(i)}(G(A)))=\bigcup_{1\le i\le 4}{}_i\hat X_1\subset \hat J{\it D}(E)$ is an open submanifold of $\hat J{\it D}(E)$. We call $\hat Y_1\equiv\hat R_1\bigcap\hat X_1$ the {\em regularized quantum PDE} corresponding to $\hat R_1$. Since $\hat X_1$ is an open submanifold of $\hat J{\it D}(E)$, it follows that $\hat Y_1$ is a quantum submanifold of $\hat J{\it D}(E)$ of dimension as reported in formulas {\em(\ref{isomorphisms-eq-1-tab3})}.
Let us define the following subsets of $\hat R_1$:
\begin{equation}\label{subsets-eq-1-tab3}
    \left\{
    \begin{array}{l}
      \hat Y_1\equiv\hat R_1\setminus \Sigma(\hat R_1)\subset \hat R_1\\
      {}_2\hat R_1\equiv\left\{q\in\hat R_1|u^1_y(q)=0,u^2_x(q)=0\right\}\subset\hat R_1\\
{}_3\hat R_1\equiv\left\{q\in\hat R_1|u^1_x(q)=0,u^2_y(q)=0\right\}\subset\hat R_1.\\
\end{array}
    \right.
\end{equation}
One has ${}_2\hat R_1\bigcap{}_3\hat R_1\not=\varnothing$, $\hat Y_1\bigcap{}_2\hat R_1\not=\varnothing$,  $\hat Y_1\bigcap{}_3\hat R_1\not=\varnothing$. Furthermore the set of singular points $\Sigma({}_2\hat R_1)$ (resp. $\Sigma({}_3\hat R_1)$) of ${}_2\hat R_1$ (resp. ${}_3\hat R_1$) is contained in $\Sigma(\hat R_1)$ and contains $\Sigma({}_2\hat R_1)_0\equiv{}_2\hat R_1\bigcap\Sigma(\hat R_1)_0\cong A^4$ (resp. $\Sigma({}_3\hat R_1)_0\equiv{}_3\hat R_1\bigcap\Sigma(\hat R_1)_0\cong A^4$). We can write:
\begin{equation}\label{split-subsets-eq-1-tab3}
\hat R_1=\hat Y_1\bigcup\Sigma(\hat R_1)\cong{}_2\hat R_1\times{}_3\hat R_1\cong[\hat Y_2\bigcup\Sigma({}_2\hat R_1)]\times[\hat Y_3\bigcup\Sigma({}_3\hat R_1)]\subset\hat J{\it D}(E),
\end{equation}
where $\hat Y_2\equiv {}_2\hat R_1\setminus\Sigma({}_2\hat R_1)$ (resp. $\hat Y_3\equiv {}_3\hat R_1\setminus\Sigma({}_3\hat R_1$). Then we can see that $\hat Y_1$ is a formally quantum integrable and completely quantum
integrable quantum PDE of first order. (For the theory of formal
integrability of quantum PDE's, see Refs.\cite{PRA7, PRA11, PRA19,
PRA20, PRA21}.) In fact $\hat Y_1$ and its prolongations $(\hat
Y_1)_{+r}\subset \hat J{\it D}^{r+1}(E)$, are subbundles of $\hat
J{\it D}^{r+1}(E)\to\hat J{\it D}^{r}(E)$, $r\ge 0$. One can also
see that the canonical maps $\pi_{r+1,r}:(\hat Y_1)_{+r}\to(\hat
Y_1)_{+(r-1)}$, are surjective mappings. For example, for $r=1$, one
has the following isomorphisms:
\begin{equation}\label{isomorphisms-eq-1-tab3}
    \left\{
    \begin{array}{l}
    \dim_{B_1}\hat J{\it D}(E)=(4,4)\\
    \\
      \hat Y_1\cong A^4\times\widehat{A}{}^2\Rightarrow\dim_{B_1}\hat Y_1=(4,2)\\
\hat J{\it D}^2(E)\cong
A^4\times\widehat{A}{}^4\times(\mathop{\widehat{A}}\limits^{2})^8\Rightarrow\dim_{B_2}\hat J{\it D}^2(E)=(4,4,8)\\
      (\hat Y_1)_{+1}\cong
      A^4\times\widehat{A}{}^2\times(\mathop{\widehat{A}}\limits^{2})^4
      \Rightarrow\dim_{B_2}(\hat Y_1)_{+1}=(4,2,4)\\
      Hom_Z(\dot S^2_0(T_pM);vT_{\bar
      q}E)\cong(\mathop{\widehat{A}}\limits^{2})^8\Rightarrow\dim_{B_2}Hom_Z(\dot S^2_0(T_pM);vT_{\bar
      q}E)=(0,0,8)\\
      ((\hat g_1)_{+1})_{q\in\hat Y_1}\cong (\mathop{\widehat{A}}\limits^{2})^4 \Rightarrow\dim_{B_2}
((\hat g_1)_{+1})_{q\in\hat Y_1}=(0,0,4)\\
\\
\left[\dim_{B_2}(\hat Y_1)_{+1}\right]=\left[\dim_{B_2}\hat
Y_1\right]+\left[\dim_{B_2}((\hat
g_1)_{+1})_{q\in\hat Y_1}\right].\\
\end{array}
    \right.
\end{equation}
Therefore, $(\hat Y_1)_{+1}\to(\hat Y_1)$, is surjective, and by
iterating this process, we get that also the mappings $(\hat
Y_1)_{+r}\to(\hat Y_1)_{+(r-1)}$, $r\ge 0$, are surjective. We put
$(\hat Y_1)_{+(-1)}\equiv E$. Thus $\hat Y_1$ is a quantum regular
quantum PDE, and under the hypothesis that $A$ has a Noetherian
centre, it follows that $\hat Y$ is quantum $\delta-$regular too.
Then, from Theorem \ref{delta-poincare-lemma-pdes-category-quantum-hypercomplex-manifolds}  and Theorem \ref{criterion-formal-quantum-integrability-pdes-category-quantum-hypercomplex-manifolds}, it follows that $\hat Y_1$ is
formally quantum integrable. Since it is quantum analytic, it is
completely quantum integrable too.

\end{example}

\begin{definition}\label{quantum-extended-crystal-singular-super-PDE}
We define {\em quantum extended crystal hypercomplex singular super PDE}, a
singular quantum hypercomplex super PDE $\hat E_k\subset \hat J^k_{m|n}(W)$ that
splits in irreducible components $\hat A_i$, i.e., $\hat
E_k=\bigcup_i \hat A_i$, where each $\hat A_i$ is a quantum extended
crystal hypercomplex super PDE. Similarly we define {\em quantum extended
$0$-crystal hypercomplex singular PDE}, (resp. {\em quantum $0$-crystal hypercomplex singular
PDE}), a quantum extended crystal hypercomplex singular PDE where each component
$\hat A_i$ is a quantum extended $0$-crystal hypercomplex PDE, (resp. quantum
$0$-crystal hypercomplex PDE).
\end{definition}

\begin{definition}{\em(Algebraic singular solutions of quantun singular hypercomplex super PDE's)}.\label{singular-algebraic-solution}
Let $\hat E_k\subset \hat J^k_{m|n}(W)$ be a quantum singular super
PDE, that splits in irreducible components $\hat A_i$, i.e., $\hat
E_k=\bigcup_i\hat A_i$. Then, we say that $\hat E_k$ admits an {\em
algebraic singular solution} $V\subset \hat E_k$, if $V\bigcap \hat A_r\equiv V_r$ is a solution  (in the usual sense) in $\hat A_r$ for
at least two different components, say $\hat A_i$, $\hat
A_j$, $i\not=j$, and such that one of following conditions are
satisfied: {\em(a)} ${}_{(ij)}\hat E_k\equiv \hat A_i\bigcap \hat
A_j\not=\varnothing$; {\em(b)} ${}^{(ij)}\hat E_k\equiv \hat
A_i\bigcup \hat A_j$ is a connected set, and ${}_{(ij)}\hat
E_k=\varnothing$. Then we say that the algebraic singular solution $V$
is in the case {\em(a)}, {\em weak}, {\em singular} or {\em smooth},
if it is so with respect to the equation ${}_{(ij)}\hat E_k$. In the
case {\em(b)}, we can distinguish the following situations:
{\em(weak solution):} There is a discontinuity in $V$, passing from
$V_i$ to $V_j$; {\em(singular solution):} there is not discontinuity
in $V$, but the corresponding tangent spaces $TV_i$ and $TV_j$ do
not belong to a same $n$-dimensional Cartan sub-distribution of
$\hat J^k_{m|n}(W)$, or alternatively $TV_i$ and $TV_j$ belong to a
same $(m|n)$-dimensional Cartan sub-distribution of $\hat
J^k_{m|n}(W)$, but the kernel of the canonical projection
$(\pi_{k,0})_*:T\hat J^k_{m|n}(W)\to TW$, restricted to $V$ is
larger than zero; {\em(smooth solution):} there is not discontinuity
in $V$ and the tangent spaces $TV_i$ and $TV_j$ belong to a same
$(m|n)$-dimensional Cartan sub-distribution of $\hat J^k_{m|n}(W)$
that projects diffeomorphically on $W$ via the canonical projection
$(\pi_{k,0})_*:T\hat J^k_{m|n}(W)\to TW$. Then we say that a
solution passing through a critical zone {\em
bifurcate}.\footnote{Note that the bifurcation does
 not necessarily imply that the tangent planes in the points of $V_{ij}\subset
 V$ to the components $V_i$ and $V_j$, should be different.}
\end{definition}

\begin{definition}{\em(Integral bordism for quantum singular hypercomplex super PDE's)}.\label{integral-bordism-quantum-singular-super-PDE}
Let $\hat E_k\subset \hat J^k_{m|n}(W)$ be a quantum super PDE on
the fiber bundle $\pi:W\to M$, $\dim_B W=(m|n,r|s)$, $\dim_A M=m|n$,
$B=A\times E$, $E$ a quantum hypercomplex superalgebra that is also a $Z$-module,
with $Z=Z(A)$ the centre of $A$. Let $N_1, N_2\subset \hat
E_k\subset \hat J^k_{m|n}(W)$ be two
 $(m-|n-1)$-dimensional, (with respect to $A$), admissible closed integral
quantum hypercomplex supermanifolds. We say that $N_1$ {\em algebraic integral
bords} with $N_2$, if $N_1$ and $N_2$ belong to two different
irreducible components, say $N_1\subset \hat A_i$, $N_2\subset \hat
A_j$, $i\not=j$, such that there exists an algebraic singular
solution $V\subset \hat E_k$ with $\partial V=N_1\varnothing N_2$.

In the integral bordism group $\Omega_{m-1|n-1}^{\hat E_k}$ (resp.
$\Omega_{m-1|n-1,s}^{\hat E_k}$, resp. $\Omega_{m-1|n-1,w}^{\hat
E_k}$) of a quantum singular hypercomplex super PDE $\hat E_k\subset \hat
J^k_{m|n}(W)$, we call {\em algebraic class} a class
$[N]\in\Omega_{m-1|n-1}^{\hat E_k}$, (resp. $[N]\in\Omega_{m-1|n-1,
s}^{\hat E_k}$, resp. $[N]\in\Omega_{m-1|n-1}^{\hat E_k}$,), with
$N\subset A_j$, such that there exists a closed
$(m-1|n-1)$-dimensional, (with respect to $A$), admissible integral
quantum hypercomplex supermanifolds $X\subset \hat A_i\subset \hat E_k$,
algebraic integral bording with $N$, i.e., there exists a smooth
(resp. singular, resp. weak) algebraic singular solution $V\subset
\hat E_k$, with $\partial V=N\varnothing X$.
\end{definition}

\begin{theorem}[Singular integral bordism group of quantum hypercomplex singular
super PDE]\label{main-quantum-singular1} Let $\hat
E_k\equiv\bigcup_i \hat A_i\subset \hat J^k_{m|n}(W)$ be a quantum
singular super PDE. Then under suitable conditions, {\em algebraic
singular solutions integrability conditions}, we can find (smooth)
algebraic singular solutions bording assigned admissible closed
smooth $(m-1|n-1)$-dimensional, (with respect to $A$), integral
quantum hypercomplex supermanifolds $N_0$ and $N_1$ contained in some component
$\hat A_i$ and $\hat A_j$, $i\not= j$.
\end{theorem}

\begin{proof}
In fact, we have the following lemmas.
\begin{lemma}\label{lemma-main-quantum-singular1}
Let $\hat E_k\equiv\bigcup_i \hat A_i\subset \hat J^k_{m|n}(W)$ be a
quantum singular super PDE with ${}_{(ij)}\hat E_k\equiv \hat
A_i\bigcap \hat A_j\not=\varnothing$. Let us assume that $\hat
A_i\subset \hat J^k_{m|n}(W)$, $\hat A_j\subset \hat J^k_{m|n}(W)$
and ${}_{(ij)}\hat E_k\subset \hat J^k_{m|n}(W)$ be formally
integrable and completely integrable quantum hypercomplex super PDE's with
nontrivial symbols. Then, one has the following isomorphisms:
\begin{equation}\label{singular-bordism-groups-quantum-singular-PDE}
\left\{\begin{array}{ll}
          \Omega_{m-1|n-1,w}^{\hat A_i}&\cong\Omega_{m-1|n-1,w}^{\hat A_j}\cong\Omega_{m-1|n-1,w}^{{}_{(ij)}\hat E_k}\\
         &\cong \Omega_{m-1|n-1,s}^{\hat A_i}\cong\Omega_{m-1|n-1,s}^{\hat A_j}\\
        & \cong\Omega_{m-1|n-1,s}^{{}_{(ij)}\hat E_k}.\\
       \end{array}
  \right.
\end{equation}

So we can find a weak or singular algebraic singular solution
$V\subset \hat E_k$ such that $\partial V=N_0\varnothing N_1$, $N_0\subset
\hat A_i$,  $N_1\subset \hat A_j$, iff $N_1\in[N_0]$.
\end{lemma}

\begin{proof}
In fact, under the previous hypotheses one has that we can apply
Theorem 2.1 in \cite{PRA10} to each component $\hat A_i$, $\hat A_j$
and ${}_{(ij)}\hat E_k$ to state that all their weak and singular
integral bordism groups of dimension $(m-1|n-1)$ are isomorphic to
$H_{m-1|n-1}(W;A)$.
\end{proof}

\begin{lemma}\label{lemma-main-quantum-singular1}
Let $\hat E_k=\bigcup_i\hat A_i$ be a quantum $0$-crystal hypercomplex singular
PDE. Let ${}^{(ij)}\hat E_k\equiv \hat A_i\bigcup \hat A_j$ be
connected, and ${}_{(ij)}\hat E_k\equiv \hat A_i\bigcap \hat
A_j\not=\varnothing$. Then $\Omega_{m-1|n-1,s}^{{}^{(ij)}\hat
E_k}=0$.\footnote{But, in general, one has
$\Omega_{m-1|n-1}^{{}^{(ij)}\hat E_k}\not=0$.}
\end{lemma}

\begin{proof}
In fact, let $Y\subset{}_{(ij)}\hat E_k$ be an admissible closed
$(m-1|n-1)$-dimensional closed integral quantum hypercomplex supermanifold, then
there exists a smooth solution $V_i\subset \hat A_i$ such that
$\partial V_i=N_0\varnothing Y$ and a solution $V_j\subset \hat A_j$ such
that $\partial V_j=Y\varnothing N_1$. Then, $V=V_i\bigcup_Y V_j$ is an
algebraic singular solution of $\hat E_k$. This solution is singular
in general.
\end{proof}

After above lemmas the proof of the theorem can be considered done
besides the algebraic singular solutions integrability conditions.
\end{proof}

\begin{example}[Quantum sedenionic d'Alembert equation]\label{quantum-sedenionic-d-alembert-equation}
Above examples about heat equations in the categories of quantum hypercomplex manifolds are linear equations, however the general theory works too also for non-linear equations. For example, we can consider the {\em quantum sedenionic d'Alembert equation} are reported in {\em(\ref{quantum-sedenionic-d-alembert-equation-definition})}.
\begin{equation}\label{quantum-sedenionic-d-alembert-equation-definition}
\left\{
\begin{array}{l}
\pi:W\equiv\mathbb{S}^3\to\mathbb{S}^2\equiv M;\quad (x,y,u)\mapsto (x,y)\\
\widehat{(d'A)}_\mathbb{S}\subset J\hat{\it D}^2(W)\subset  \hat J^2_2(W):\quad F\equiv uu_{xy}-u_xu_y=0\\
\end{array}
\right\}.
\end{equation}
In this case, we can apply the same machinary, for the quantum smooth submanifold of $\hat X_2\subset J\hat{\it D}^2(W)$, where $\hat X_2$ is the open quantum hypercomplex submanifold of $J\hat{\it D}^2(W)$, identified with the condition $u\not=0$. In fat, there the jacobian $$(\partial\xi_j.F)=\left(
                                         \begin{array}{cccccccc}
                                           0 & 0 & u_{xy} & -u_{y} & -u_{x} & 0 & u & 0 \\
                                         \end{array}
                                       \right)$$
has rank $1$. Here $(\xi^j)=(x,y,u,u_x,u_y,u_{xx},u_{xy},u_{yy})$ are quantum hypercomplex coordinates on $J\hat{\it D}^2(W)$. Thus, in the case, in order to regularize $\widehat{(d'A)}_\mathbb{S}$, it is enough to consider $\widehat{(d'A)}_\mathbb{S}\bigcap\hat X_2$, $\hat X_2\equiv u^{-1}(0)$, with $u:J\hat{\it D}^2(W)\to A$. (Compare with singular quantum hypercomplex PDE's considered in Example \ref{examples-quantum-singular-PDEs}.)

\end{example}

\section{Quantum hypercomplex exotic PDE's}\label{quantum-exotic-pdes-section}

In this section we extend to PDE's in the category $\mathfrak{Q}_{hyper}$, a previous definition, and some results, about ''exotic PDE's'' given in the category of commutative manifolds \cite{PRA19, PRA20, PRA21, PRA21-1}. This allows us the opportunity to discuss about a classification of smooth solutions of PDE's in the category $\mathfrak{Q}_{hyper}$ and to prove a smooth generalized version of the quantum generalized Poincar\'e conjecture, (previously proved in \cite{PRA15} in the category $\mathfrak{Q}_S$ of quantum supermanifolds). There it is proved that a closed quantum supermanifold $M$, of dimension $(m|n)$, with respect to a quantum superalgebra $A$, homotopy equivalent to the quantum $(m|n)$dimensional supersphere, $\hat S^{m|n}$, is homeomorphic (but not necassarily diffeomorphic) to $\hat S^{m|n}$, if $M$ is classic regular, with classic limit a $m$-dimensional manifold $M_C$, identified by means of a fiber structure $\bar\pi_C:M\to M_C$, such that the homotopy equivalence is realized by means of a couple of mappings $(f,f_C)$ such that the diagram (\ref{commutative-diagram-homotopy-equivalence-quantum-m-n-supersphere}) is commutative.
\begin{equation}\label{commutative-diagram-homotopy-equivalence-quantum-m-n-supersphere}
\xymatrix{M\ar[d]_{\bar\pi_C}\ar[r]^{f}&\hat S^{m|n}\ar[d]^{\pi_C}\\
M_C\ar[r]_{f_C}&S^m\\}
\end{equation}
Thus in this section we will assume quantum manifolds with classic regular quantum manifold structures $\bar\pi_C:M\to M_C$, $\dim_A M=\dim_{\mathbb{R}}M_C=n$, that are compact, closed and homotopy equivalent to the quantum $n$-sphere $\hat S^n$, (with respect to the same quantum (hypercomplex) algebra $A$), are homeomorphic to $\hat S^n$ too.

\begin{definition}[Quantum homotopy $n$-sphere]\label{quantum-homotopy-n-sphere}
We call quantum homotopy $n$-sphere (with respect to a quantum algebra $A$) a smooth, compact, closed $n$-dimensional quantum manifold $\hat\Sigma^n$, that is homeomorphic to $\hat S^n$, with classic regular structure $\bar\pi_C:M\to M_C$, where $M_C$ is a homotopy $n$-sphere, and such that the homotopy equivalence between $M$ and $\hat S^n$ is realized by a commutative diagram {\em(\ref{commutative-diagram-homotopy-equivalence-quantum-m-n-supersphere})}.
\end{definition}

\begin{remark}\label{remark-quantum-spheres-classic-limits}
Let $\hat\Sigma^{n}_1$ and $\hat\Sigma^{n}_2$ be two quantum diffeomorphic, quantum homotopy $n$-spheres: $\hat\Sigma^{n}_1\cong\hat\Sigma^{n}_2$. Then the corresponding classic limits $\hat\Sigma^{n}_{1,C}$ and $\hat\Sigma^{n}_{2,C}$ are diffeomorphic too: $\hat\Sigma^{n}_{1,C}\cong\hat\Sigma^{n}_{2,C}$. This remark is the natural consequence of the fact that quantum diffeomorphisms here considered respect the fiber bundle structures of quantum homotopy $n$-spheres with respect their classic limits: $\pi_C:\hat\Sigma^n\to\hat\Sigma^n_C$. Therefore quantum diffeomorphisms between quantum homotopy $n$-spheres are characterized by a couple $(f,f_C):(\hat\Sigma_1^n,\hat\Sigma_{1,C}^n)\to(\hat\Sigma_2^n,\hat\Sigma_{2,C}^n)$ of mappings related by the commutative diagram in {\em(\ref{commutative-diagram-a-lemma-quantum-spheres-classic-limits-a})}.
\begin{equation}\label{commutative-diagram-a-lemma-quantum-spheres-classic-limits-a}
 \xymatrix{\hat\Sigma^{n}_1\ar[d]_{\pi_{1,C}}\ar[r]^{f}&\hat\Sigma^{n}_2\ar[d]^{\pi_{2,C}}\\
 \hat\Sigma^{n}_{1,C}\ar[r]_{f_C}&\hat\Sigma^{n}_{2,C}\\}
\end{equation}
There $f$ is a quantum diffeomorphism between quantum manifolds and $f_C$ is a diffeomorphism between manifolds. Note that such diffeomorphisms of quantum homotopy $n$-spheres allow to recognize that $\hat\Sigma^n_1$ has also $\hat\Sigma^{n}_{2,C}$ as classic limit, other than $\hat\Sigma^{n}_{1,C}$. (See commutative diagram in {\em(\ref{commutative-diagram-b-lemma-quantum-spheres-classic-limits-a})}.)
\begin{equation}\label{commutative-diagram-b-lemma-quantum-spheres-classic-limits-a}
 \xymatrix{&\hat\Sigma^{n}_1\ar[ddl]^(0.7){\pi_{1,C}}\ar[ddr]_(0.7){\pi'_{1,C}}\ar[dr]^{f}&\\
 \hat\Sigma^{n}_2\ar[ur]^{f^{-1}}\ar[d]_{\pi'_{2,C}}\ar[rr]^{1_{\hat\Sigma^n_2}}&&\hat\Sigma^{n}_2\ar[d]^{\pi_{2,C}}\\
 \hat\Sigma^{n}_{1,C}\ar[rr]_{f_C}&&\hat\Sigma^{n}_{2,C}\\}
\end{equation}

This clarifies that the classic limit of a quantum homotopy $n$-sphere is unique up to diffeomorphisms.
\end{remark}
Let us also emphasize that (co)homology properties of quantum homotopy $n$-spheres are related to the ones of $n$-spheres, since we here consider classic regular objects only.

\begin{lemma}\label{cohomology-quantum-homotopy-n-spheres}
In {\em(\ref{cohomology-properties-quantum-homotopy-spheres})} are reported the cohomology spaces for quantum homotopy $n$-spheres.

\begin{equation}\label{cohomology-properties-quantum-homotopy-spheres}
  H^p(\hat\Sigma^{n};\mathbb{Z})\cong  H^p(\hat S^{n};\mathbb{Z})\cong  H^p(S^{n};\mathbb{Z})=\left\{\begin{array}{ll}
                                                                                                       0 & p\not=0,\, n\\
                                                                                                       \mathbb{Z}& p=0,\, n.\\
                                                                                                     \end{array}
  \right.
\end{equation}
\end{lemma}

\begin{proof}
Let us first calculate the homology groups in integer coefficients $\mathbb{Z}$, of quantum $n$-spheres. In (\ref{homology-properties-quantum-n-spheres}) are reported the homology spaces for quantum $n\not=0$-spheres.
\begin{equation}\label{homology-properties-quantum-n-spheres}
  H_p(\hat S^{n};\mathbb{Z})
  \cong  H_p(S^{n};\mathbb{Z})=\left\{\begin{array}{ll}
  0 & p\not=0,\, n\\
  \mathbb{Z}& p=0,\, n.\\
  \end{array}
  \right.
\end{equation}
Furthermore, for $n=0$ we get
$$ H_p(\hat S^{0};\mathbb{Z})
  \cong  H_p(S^{0};\mathbb{Z})=\left\{\begin{array}{ll}
  0 & p\not=0\\
  \mathbb{Z}\bigoplus\mathbb{Z}& p=0.\\
  \end{array}\right.$$
  Above formulas can be obtained by the reduced Mayer-Vietoris sequence applied to the triad $(\hat S^n,\hat D^n_+,\hat D^n_-)$ since we can write $\hat S^n=\hat D^n_+\bigcup\hat D^n_-$, where $\hat D^n_+$ and $\hat D^n_-$ are respectively the north quantum $n$-disk and south quantum $n$-disk that cover $\hat S^n$. Taking into account that $\hat D^n_+\bigcap\hat D^n_-=\hat S^{n-1}$, we get the long exact sequence (\ref{reduced-quantum-n-sphere-mayer-vietoris-exact-sequence}).

  \begin{equation}\label{reduced-quantum-n-sphere-mayer-vietoris-exact-sequence}
   \scalebox{0.7}{$ \xymatrix{\cdots\ar[r]&\widetilde{ H}_{p}(\hat S^{n-1};\mathbb{Z})\ar[r]&\widetilde{ H}_{p}(\hat D_+^n;\mathbb{Z})\bigoplus \widetilde{ H}_{p}(\hat D_-^n;\mathbb{Z})\ar[r]&\widetilde{H}_{p}(\hat S^n;\mathbb{Z})\ar[d]_{\partial}\\
    &\widetilde{ H}_{p-1}(\hat S^{n};\mathbb{Z})\ar[d]_{\partial}&\ar[l]\widetilde{ H}_{p-1}(\hat D_+^n;\mathbb{Z})\bigoplus \widetilde{\hat H}_{p-1}(\hat D_-^n;\mathbb{Z})&\ar[l]\widetilde{ H}_{p-1}(\hat S^{n-1};\mathbb{Z})\\
    &\widetilde{ H}_{p-2}(\hat S^{n-1};\mathbb{Z})\ar[r]&\widetilde{ H}_{p-2}(\hat D_+^n;\mathbb{Z})\bigoplus \widetilde{ H}_{p-2}(\hat D_-^n;\mathbb{Z})\ar[r]&\widetilde{H}_{p-2}(\hat S^{n};\mathbb{Z})\ar[d]_{\partial}\\
    &\vdots\ar[d]_{\partial}&&\\
    &\widetilde{ H}_{0}(\hat S^{n-1};\mathbb{Z})\ar[r]&\widetilde{ H}_{0}(\hat D_+^n;\mathbb{Z})\bigoplus \widetilde{ H}_{0}(\hat D_-^n;\mathbb{Z})\ar[r]&\widetilde{ H}_{0}(\hat S^n;\mathbb{Z})\ar[r]&0.\\}$}
  \end{equation}
Taking into account that $\widetilde{ H}_{0}(\hat D_-^n;\mathbb{Z})=0$ e get $\widetilde{ H}_{p}(\hat S^n;\mathbb{Z})\cong\widetilde{H}_{p-1}(\hat S^{n-1};\mathbb{Z})$ and $\widetilde{ H}_{0}(\hat S^n;\mathbb{Z})\cong 0$.
Therefore, we get
\begin{equation}\label{reduced-homology-properties-quantum-n-spheres}
  \widetilde{ H}_p(\hat S^{n};\mathbb{Z})
  \cong  \left\{\begin{array}{ll}
  \mathbb{Z} & \hbox{\rm if $p=n$}\\
  0& \hbox{\rm if $p\not=n$}\\
  \end{array}
  \right\}\Rightarrow\, \left\{\begin{array}{l}
   H_p(\hat S^{0};\mathbb{Z}) =\left\{ \begin{array}{ll}
  \mathbb{Z}\bigoplus\mathbb{Z} & \hbox{\rm if $p=0$}\\
  0& \hbox{\rm if $p\not=0$}\\
  \end{array}\right. \\
   H_p(\hat S^{n};\mathbb{Z}) =\left\{ \begin{array}{ll}
  \mathbb{Z} & \hbox{\rm if $p=0,\, n$}\\
  0& \hbox{\rm if $p\not=0,\, n$.}\\
  \end{array}\right. \\
  \end{array}
  \right.
\end{equation}

Therefore we get formulas (\ref{homology-properties-quantum-n-spheres}). To conclude the proof we shall consider that $ H^p(\hat S^n;\mathbb{Z})\cong Hom_{\mathbb{Z}}( H_p(\hat S^n;\mathbb{Z});\mathbb{Z})$. Furthermore, quantum homotopy $n$-spheres have same (co)homology of quantum spheres since are homotopy equivalent to these last ones.

\end{proof}

\begin{lemma}\label{quantum-euler-characteristic-quantum-spheres-classic-limits-b}
The (quantum) Euler characteristic numbers for quantum homotopy $n$-spheres are reported in {\em(\ref{euler-characteristic-properties-quantum-homotopy-spheres})}.

\begin{equation}\label{euler-characteristic-properties-quantum-homotopy-spheres}
\hat \chi(\hat\Sigma^{n})= \hat \chi(\hat S^{n})=\chi(S^{n})=(-1)^0\beta_0+(-1)^n\beta_n=1+(-1)^n=\left\{\begin{array}{ll}
                                                                                                       0 & \hbox{\rm $n=$ odd}\\
                                                                                                       2& \hbox{\rm $n=$ even.}\\
                                                                                                     \end{array}
  \right.
\end{equation}
\end{lemma}

\begin{proof}
We have considered that $\hat S^{n}$ admits the following quantum-cell decomposition: $\hat S^n=\hat e^n\bigcup \hat e^0$, where $\hat e^n=\hat D^n$ is a $n$-dimensional quantum cell, with respect to the quantum algebra $A$, and $\hat e^0=\hat D^0$ is the $0$-dimensional quantum cell with respect to $A$. Therefore we can consider the {\em quantum homological Euler characteristic} $\hat\chi(\hat S^{n})$ of $\hat S^{n}$, given by formulas (\ref{quantum-homology-euler-characteristic-properties-quantum-n-spheres}).
\begin{equation}\label{quantum-homology-euler-characteristic-properties-quantum-n-spheres}
\begin{array}{ll}
\hat\chi(\hat S^{n})&=(-1)^0\dim_A H_0(\hat S^n;A)+(-1)^n\dim_A H_n(\hat S^n;A)\\
&=(-1)^0\dim_A A+(-1)^n\dim_A A\\
&=1+(-1)^n\\
&=\left\{\begin{array}{ll}
0 & \hbox{\rm $n=$ odd}\\
2& \hbox{\rm $n=$ even.}\\
\end{array}\right.\\
\end{array}
\end{equation}
 So the homological quantum Euler characteristic of the quantum $n$-sphere is the same of the homological Euler characteristic of the usual $n$-sphere. Furthermore, since quantum homotopy $n$-spheres are homotopy equivalent to quantum $n$-spheres, it follows that the quantum Euler characteristic of a quantum homotopy $n$-sphere is equal to the one of $\hat S^n$.\footnote{In general, quantum characteristic numbers for quantum manifolds with quantum algebras $A$, are $A$-valued characteristic forms. (For details see \cite{PRA0500, PRA2, PRA3, PRA12, PRA13, PRA14, PRA15, PRA16}.) However, taking into account the canonical ring homomorphism $\epsilon:\mathbb{R}\to Z(A)\subset A$, we can consider also above characteristic numbers for quantum homotopy $n$-spheres as belonging to the corresponding quantum algebra $A$.}
 \end{proof}

\begin{theorem}\label{theorem-quantum-spheres-classic-limits-b}
Let $\hat\Theta_{n}$ be the set of equivalence classes of quantum diffeomorphic quantum homotopy $n$-spheres over a quantum (hypercomplex) algebra $A$ (and with Noetherian centre $Z(A)$).\footnote{Quantum diffeomorphisms are meant in the sense specified in Remark \ref{remark-quantum-spheres-classic-limits}.} In $\hat\Theta_{n}$ it is defined an additive commutative and associative composition map such that $[\hat S^n]$ is the zero of the composition. Then one has the exact commutative diagram reported in {\em(\ref{short-exact-sequence-lemma-quantum-spheres-classic-limits-b})}.\footnote{For the definition of the groups $\Theta_n$, see \cite{PRA20}. Let us emphasize here that $\Theta_n$ is a finite abelian group (Kervaire-Milnor). These are particular cases of finitely generated abelian groups $G$ that admit a finite direct sum decomposition like in (\ref{cyclic-decomposition-finitely-generated-abelian-group}).
\begin{equation}\label{cyclic-decomposition-finitely-generated-abelian-group}
  G\cong \mathbb{Z}_{p_1}\bigoplus\cdots\bigoplus\mathbb{Z}_{p_s}\bigoplus\mathbb{Z}^r.  \end{equation}
The partial sum $T\equiv\oplus_{1\le i\le s}\mathbb{Z}_{p_i}$ is called {\em torsion subgroup} of $G$. It is a finite group and consists of all elements of $G$ of finite order. The quotient $G/T\cong\mathbb{Z}^r$ is called {\em free part} of $G$. The number $r$ of summands $\mathbb{Z}$ in $G/T$ is called the {\em rank} (or {\em Betti number}) of $G$. It does not depend on the particular direct sum decomposition (\ref{cyclic-decomposition-finitely-generated-abelian-group}). In fact, $\RANKP(G)$ is the maximal number of linearly independent elements in $G$. The numbers $p_j$ ({\em torsion coefficients} in $G$) that occur in (\ref{cyclic-decomposition-finitely-generated-abelian-group}) are not unique. However, they can be chosen as powers of prime numbers $p_j=q_j^{\rho_j}$, $q_j$ prime, $\rho_j>0$, and then they are unique (independent of the decomposition) up to permutation coefficients. Two finitely generated abelian groups are isomorphic iff they have the same rank and the same system of torsion coefficients. $\Theta_n$ as a finite abelian group has rank $r=0$. Let us emphasize that the {\em fundamental theorem of finite abelian group} states that any such a group $H$ can be written in a direct product of cyclic groups: $H=\mathbb{Z}_{p_1}\bigoplus\cdots\bigoplus\mathbb{Z}_{p_s}$, such that $(p_i,\cdots,p_s)$ are powers of primes, or $p_i$ divides $p_{i+1}$. For example we get the isomorphism: $\mathbb{Z}_{15}\cong\mathbb{Z}_3\bigoplus\mathbb{Z}_5$, but $\mathbb{Z}_8\not\cong\mathbb{Z}_4\bigoplus\mathbb{Z}_2\not\cong \mathbb{Z}_2\bigoplus\mathbb{Z}_2\bigoplus\mathbb{Z}_2$. (In other words $\mathbb{Z}_{pq}\cong\mathbb{Z}_p\bigoplus\mathbb{Z}_q$ iff $p$ and $q$ are coprime.)
Let us emphasize that after above representation of finite abelian groups, we can for abuse of notation denote $\dim_{\mathbb{Z}}H={\rm order}(H)$. (A notation for the order of a group $H$ is also $|H|$.) This is justified since the integral group ring $\mathbb{Q}H$ of $H$ is just a vector space of dimension equal to $|H|$. (See Tab. 4 in \cite{PRA20}.)}

\begin{equation}\label{short-exact-sequence-lemma-quantum-spheres-classic-limits-b}
 \xymatrix{&&&0\ar[d]&\\
 0\ar[r]&\hat\Upsilon_n\ar@{^{(}->}[r]&\hat\Theta_n\ar[r]^{j_C}&\Theta_n\ar[d]\ar[r]&0\\
 &&&\hat\Theta_n/\hat\Upsilon_n\ar[d]&\\
 &&&0&\\}
\end{equation}
where $\Theta_n$ is the set of equivalence classes for diffeomorphic homotopy $n$-spheres and $j_C$ is the canonical mapping $j_C:[\hat\Sigma^n]\mapsto[\hat\Sigma^n_C]$. One has the canonical isomorphisms:
\begin{equation}\label{canonical-isomorphisms-in-short-exact-sequence-lemma-quantum-spheres-classic-limits-b}
   \mathbb{Z}\bigotimes_{\hat\Upsilon_n} \mathbb{Z} \hat\Theta_n\cong\mathbb{Z} \Theta_n,\, \hbox{\rm as right $\hat\Theta_n$-modules}.
\end{equation}
\end{theorem}

\begin{proof}
After above Remark \ref{remark-quantum-spheres-classic-limits} we can state that the mapping $j_C$ is surjective. In other words we can write
$$\hat\Theta_n=\bigcup_{[\hat\Sigma^n_C]\in\Theta_n}(\hat\Theta_n)_{[\hat\Sigma^n_C]}.$$
The fiber $(\hat\Theta_n)_{[\hat\Sigma^n_C]}$ is given by all classes $[\hat\Sigma^n]$ such their classic limits are diffeomorphic, hence belong to the same class in $\Theta_n$. Furthermore one has $\ker(j_C)=j_C^{-1}([S^n])\equiv\hat\Upsilon_n\subset\hat\Theta_n$. Therefore, we can state that $\hat\Theta_n$ is an extension of $\Theta_n$ by $\hat\Upsilon_n$. Such extensions are classified by $H^2(\Theta_n;\hat\Upsilon_n)$.\footnote{In Tab. \ref{homologies-ciclic-group} are reported useful formulas to explicitly calculate these groups.}

\begin{table}[h]
\caption{Homology of finite cyclic group $\mathbb{Z}_i$ of order $i$.}
\label{homologies-ciclic-group}
\begin{tabular}{|c|c|}
  \hline
  \hfil{\rm{\footnotesize $ r$}}\hfil& \hfil{\rm{\footnotesize $H_r(\mathbb{Z}_i;\mathbb{Z})$}}\hfil\\
  \hline
 \hfil{\rm{\footnotesize $0$}}\hfil& \hfil{\rm{\footnotesize $\mathbb{Z}$}}\hfil\\
 \hline
  \hfil{\rm{\footnotesize $r$ odd}}\hfil& \hfil{\rm{\footnotesize $\mathbb{Z}_i$}}\hfil\\
\hline
  \hfil{\rm{\footnotesize $r>0$ even}}\hfil& \hfil{\rm{\footnotesize $0$}}\hfil\\
\hline
\end{tabular}
\end{table}
The composition map in $\hat\Theta_n$ is defined by {\em quantum fibered connected sum}, i.e., a connected sum on quantum manifolds that respects the connected sum on their corresponding classic limits. More precisely let $M\to M_C$ and $N\to N_C$ be connected $n$-dimensional classic regular quantum manifolds. We define quantum fibered connected sum of $M$ and $N$ the classic regular $n$-dimensional quantum manifold $M\sharp N\to M_C\sharp N_C$, where
\begin{equation}\label{quantum-fibered-connected-sum}
 \left\{
 \begin{array}{ll}
   M\sharp N&=(M\setminus\hat D^n)\bigcup(\hat S^{n-1}\times\hat D^1)\bigcup(N\setminus\hat D^n)\\
   &\\
   M_C\sharp N_C&=(M_C\setminus D^n)\bigcup( S^{n-1}\times D^1)\bigcup(N_C\setminus D^n).\\
 \end{array}
 \right.
\end{equation}
Then the additive composition law is $+:\hat\Theta^n\times\hat\Theta^n\to\hat\Theta^n$, $[M]+[N]=[M\sharp N]$. $[\hat S^n]$ is the zero of this addition. In fact, since $\overline{\hat S^n\setminus\hat D^n}\bigcup_{\hat S^{n-1}}(\hat S^{n-1}\times\hat D^1)\cong\hat D^n$, we get
\begin{equation*}
    \begin{array}{ll}
      M\sharp\hat S^n&\cong\overline{M\setminus\hat D^n}\bigcup_{\hat S^{n-1}}(\overline{\hat S^n\setminus\hat D^n}\bigcup_{\hat S^{n-1}}(\hat S^{n-1}\times\hat D^1))\\
      &\\
      &\cong\overline{M\setminus\hat D^n}\bigcup_{\hat S^{n-1}}\hat D^n\cong M.\\
    \end{array}
\end{equation*}
Analogous calculus for $M_C$ completes the proof.
\end{proof}

In the following remark we will consider some examples and further results to better understand some relations between quantum homotopy spheres and their classic limits.

\begin{example}[Quantum homotopy $7$-sphere]\label{quantum-homotopy-7-spheres-with-quantum-algebra-a} Let us calculate the extension classes
\begin{equation}\label{short-exact-sequence-lemma-quantum-7-spheres-classic-limits-c}
 \xymatrix{0\ar[r]&\hat\Upsilon_7\ar@{^{(}->}[r]&\hat\Theta_7\ar[r]^{j_C}&\Theta_7\ar[r]&0}
\end{equation}
 These are given by $H^2(\Theta_7;\hat\Upsilon_7)\cong H^2(\mathbb{Z}_{28};\hat\Upsilon_7)$. We get
 \begin{equation}\label{calculation-short-exact-sequence-lemma-quantum-7-spheres-classic-limits-c}
    H^2(\mathbb{Z}_{28};\hat\Upsilon_7) =Hom_{\mathbb{Z}}(H_2(\mathbb{Z}_{28};\mathbb{Z});\hat\Upsilon_7)
    =Hom_{\mathbb{Z}}(0;\hat\Upsilon_7)\bigoplus Ext_{\mathbb{Z}}(H_1(\mathbb{Z}_{28};\mathbb{Z});\hat\Upsilon_7).
 \end{equation}
We shall prove that $Ext_{\mathbb{Z}}(H_1(\mathbb{Z}_{28};\mathbb{Z});\hat\Upsilon_7)=\hat\Upsilon_7/28\cdot \hat\Upsilon_7$.\footnote{We have used the fact that $H_2(\mathbb{Z}_{28};\mathbb{Z})=0$.} Let us look in some detail to this $\mathbb{Z}$-module. By using
the projective resolution of $\mathbb{Z}_{28}$ given in {\em(\ref{projective-resolution-z28})},
\begin{equation}\label{projective-resolution-z28}
    \xymatrix{0\ar[r]&\mathbb{Z}\ar[r]^{\mu=.28}&\mathbb{Z}\ar[r]^{\epsilon}&\mathbb{Z}_{28}\ar[r]&0}
\end{equation}
we get the exact sequence {\em(\ref{projective-resolution-z28-derived})}.

\begin{equation}\label{projective-resolution-z28-derived}
    \xymatrix{0\ar[r]&Hom_{\mathbb{Z}}(\mathbb{Z}_{28};\hat\Upsilon_7)\ar@{=}[d]\ar[r]^{\epsilon_*}&
    Hom_{\mathbb{Z}}(\mathbb{Z};\hat\Upsilon_7)\ar@{=}[d]^{\wr}\ar[r]^{\mu_*}&Hom_{\mathbb{Z}}(\mathbb{Z};\hat\Upsilon_7)\ar@{=}[d]^{\wr}\\
    0\ar[r]&Hom_{\mathbb{Z}}(\mathbb{Z}_{28};\hat\Upsilon_7)\ar[r]^(0.6){\epsilon_*}&
    \hat\Upsilon_7\ar[r]^{\mu_*}&\hat\Upsilon_7\\}
\end{equation}
Therefore we get
$$Ext_{\mathbb{Z}}(\mathbb{Z}_{28};\hat\Upsilon_7)=\hat\Upsilon_7/\IM(\mu_*).$$
In order to see what is $\IM(\mu_*)$ let us consider that $\mu_*$ is defined by the commutative diagram in {\em(\ref{commutative-diagram-multiplication-28})}.
\begin{equation}\label{commutative-diagram-multiplication-28}
    \xymatrix{\mathbb{Z}\ar[r]^{\alpha}&\hat\Upsilon_7\\
    \mathbb{Z}\ar[u]^{\mu=.28\, }\ar@/_1pc/[ur]_{\mu_*(\alpha)}\\}
\end{equation}
Since $\alpha$ is identified by means of the image of $1\in\mathbb{Z}$, i.e., $\alpha(1)\in\hat\Upsilon_7$, similarly also $\mu_*(\alpha)$ is determined by image of $1\in\mathbb{Z}$, i.e., $\mu_*(\alpha)(1)=\alpha(\mu(1))=\alpha(28)=28\cdot\alpha(1)\in\hat\Upsilon_7$. Therefore $\IM(\mu^*)=28.\hat\Upsilon_7$. Thus we get $Ext_{\mathbb{Z}}(\mathbb{Z}_{28};\hat\Upsilon_7)=\hat\Upsilon_7/28\cdot\hat\Upsilon_7$. The particular structure of this module, depends on the particular quantum algebra considered. For example, take $A=\mathbb{C}$. Since $\hat S^7\to S^7$, is just the fiber bundle $S^{14}\to S^7$, we can easily see that $\hat\Upsilon_7=\Theta_{14}=\mathbb{Z}_2$. In this case the extensions {\em(\ref{short-exact-sequence-lemma-quantum-7-spheres-classic-limits-c})} are classified by $Ext_{\mathbb{Z}}(\mathbb{Z}_{28},\mathbb{Z}_2)=\mathbb{Z}_2/(28\cdot\mathbb{Z}_2)=\mathbb{Z}_{2}/(14\cdot(2\cdot\mathbb{Z}_2))=
\mathbb{Z}_{2}/(14\cdot 0)=\mathbb{Z}_{2}$.\footnote{$\mathbb{Z}_2$ is a field of characteristic $2$, hence $2\cdot\mathbb{Z}_2=0$.} Therefore we get that all the extensions in {\em(\ref{short-exact-sequence-lemma-quantum-7-spheres-classic-limits-c})} are in correspondence one-to-one with $\Theta_{14}$.
Let us explicate in some more details this result. Since $\hat S^7\cong S^{14}$ and $\Theta_{14}=\mathbb{Z}_2$, it follows that other $S^{14}$ there is another class of homotopy $14$-spheres, non-diffeomorphic to $S^{14}$. Let us denote one of these spheres by $\Sigma^{14}$. Let us denote by $f:\Sigma^{14}\to S^{14}$ the homotopy map existing between $S^{14}$ and $\Sigma^{14}$. Then one has also a continuous fibration $\pi=\pi_C\circ f:\Sigma^{14}\to S^{14}\to S^{7}$. Since $C^s(\Sigma^{14},S^7)$ is dense in $C^r(\Sigma^{14},S^7)$, for $0\le s<r$, we can approximate $\pi$ by a smooth mapping, yet denoted $\pi$. Therefore we can consider the identification $\hat\Upsilon_7=\mathbb{Z}_2$. Furthermore, since there exists homeomorphisms $\phi:S^7\to\Sigma^7$, where $\Sigma^7$ are homotopy $7$-spheres belonging to different classes in $\Theta_7$, we get also continuous fibered structures $\Sigma^{14}\to \Sigma^7$, that again can be approximated by smooth mappings. In this way we can have the identification $\hat\Theta_7/\hat\Upsilon_7\cong\Theta_7\cong\mathbb{Z}_{28}$, and all the possible diffeomorphic classes of quantum homotopy $7$-spheres, with respect to the quantum algebra $A=\mathbb{C}$ are $\hat\Theta_7=\mathbb{Z}_2\rtimes_{\alpha}\mathbb{Z}_{28}$, where $\alpha$ are homomorphisms $\alpha:\mathbb{Z}_{28}\to Aut(\mathbb{Z}_{2})$. On the other hand, one has $Aut(\mathbb{Z}_{2})=\mathbb{Z}_{2}^{*}=\{1\}$, where $\mathbb{Z}_{2}^{*}$ is the group of units of $\mathbb{Z}_{2}$, that coincide with all numbers $0\le p\le 2$ coprime to $2$. Therefore the quantum homotopy $7$-spheres diffeomorphic classes are $2\times 28=56$, i.e., the order of $\hat\Theta_7=\mathbb{Z}_2\bigoplus\mathbb{Z}_{28} $.
\end{example}

\begin{example}[Quantum homotopy $n$-spheres for the limit case $A=\mathbb{R}$]\label{quantum-homotopy-n-spheres-with-real-numbers-quantum-algebra}
In the limit case where the quantum algebra is $A=\mathbb{R}$, then the classic limit of a quantum homotopy $n$-sphere $\hat\Sigma^n$ is just $\hat\Sigma^n_C=\hat\Sigma^n$, hence $\pi_C=id_{\hat\Sigma^n}$. Furthermore $\hat\Theta_n=\Theta_n$ and $\hat\Upsilon_n=0=[S^n]\in\Theta_n$. In particular if $n=\{1,2,3,4,5,6\}$, we get $\hat\Theta_n=\Theta_n=\hat\Upsilon_n=0$. (For the smooth case $n=4$ see \cite{PRA21-1}.)
\end{example}

\begin{theorem}[Homotopy groups of quantum $n$-sphere]\label{homotopy-groups-quantum-n-sphere}
Quantum homotopy $n$-spheres cannot have, in general, the same homotopy groups of $n$-spheres:\footnote{In other words, quantum homotopy $n$-spheres are not homotopy equivalent to the $n$-sphere.}
\begin{equation}\label{non-isomorphism-homotopy-groups-quantum-n-sphere-sphere}
    \pi_{k}(\hat\Sigma^n)\cong\pi_{k}(\hat S^n)\not=\pi_{k}(S^n).
\end{equation}

Furthermore, $S^n$ can be identified with a contractible subspace, yet denoted $S^n$, of $\hat S^n$. There exists a mapping $\hat S^n\to S^n$, but this is not a retraction, and the inclusion $S^n\hookrightarrow\hat S^n$, cannot be a homotopy equivalence.\footnote{For any convenience, in Tab. \ref{homotopy-groups-n-sphere} are reported some homotopy groups for $S^n$.}
\end{theorem}
\begin{table}[h]
\caption{Homotopy groups of $n$-sphere.}
\label{homotopy-groups-n-sphere}
\begin{tabular}{|c|c|}
  \hline
  \hfil{\rm{\footnotesize $ k$}}\hfil& \hfil{\rm{\footnotesize $\pi_k(S^n)$}}\hfil\\
  \hline
 \hfil{\rm{\footnotesize $k<n$}}\hfil& \hfil{\rm{\footnotesize $0$}}\hfil\\
 \hline
  \hfil{\rm{\footnotesize $n$ }}\hfil& \hfil{\rm{\footnotesize $\mathbb{Z}$}}\hfil\\
\hline
\multicolumn{2}{|c|}{\rm{\footnotesize Examples for $k>n$.}}\\
\hline
\multicolumn{2}{|l|}{\rm{\footnotesize $\pi_k(S^0)=0,\, k\ge 0$.}}\\
\multicolumn{2}{|l|}{\rm{\footnotesize $\pi_k(S^1)=0,\, k>1$.}}\\
\multicolumn{2}{|l|}{\rm{\footnotesize $\pi_3(S^2)=\mathbb{Z},\, \pi_4(S^2)=\pi_5(S^2)=\mathbb{Z}_2,\, \pi_6(S^2)=\mathbb{Z}_{12},\, \pi_7(S^2)=\mathbb{Z}_{2}$.}}\\
\multicolumn{2}{|l|}{\rm{\footnotesize $\pi_4(S^3)=\pi_5(S^3)=\mathbb{Z}_2,\, \pi_6(S^3)=\mathbb{Z}_{12},\, \pi_7(S^3)=\mathbb{Z}_{2}$.}}\\
\multicolumn{2}{|l|}{\rm{\footnotesize $\pi_5(S^4)=\pi_6(S^4)=\mathbb{Z}_2,\, \pi_7(S^4)=\mathbb{Z}\times\mathbb{Z}_{12}$.}}\\
\multicolumn{2}{|l|}{\rm{\footnotesize $\pi_6(S^5)=\pi_7(S^5)=\mathbb{Z}_2,\, \pi_8(S^5)=\mathbb{Z}_{24}$.}}\\
\multicolumn{2}{|l|}{\rm{\footnotesize $\pi_7(S^6)=\pi_8(S^6)=\mathbb{Z}_2$.}}\\
\multicolumn{2}{|l|}{\rm{\footnotesize $\pi_8(S^7)=\mathbb{Z}_2$.}}\\
\hline
\end{tabular}
\end{table}

\begin{proof}
Since must necessarily be $\pi_k(\hat \Sigma^n)\cong \pi_k(\hat S^n)$, $k\ge 0$, it is enough prove theorem for $\hat S^n$. We shall first recall some useful definitions and results of Algebraic Topology, here codified as lemmas.
\begin{definition}
A pair $(X,A)$ has the {\em homotopy extension property} if a homotopy $f_t:A\to Y$, $t\in I$, can be extended to  homotopy $f_t:X\to Y$ such that $f_0:X\to Y$ is a given map.
\end{definition}
\begin{lemma}
If $(X,A)$ is a CW pair, then it has the homotopy extension property.
\end{lemma}
\begin{lemma}\label{contractibility-and-homotopy-equivalence}
If the pair $(X,A)$ satisfies the homotopy extension property and $A$ is contractible, then the quotient map $q:X\to X/A$ is a homotopy equivalence.
\end{lemma}
Let us consider that we can represent $S^n$ into $\hat S^n$ by a continuous mapping $s:S^n\to\hat S^n$, defined by means of the commutative diagram in (\ref{commutative-diagram-section-quantum-n-sphere}).
\begin{equation}\label{commutative-diagram-section-quantum-n-sphere}
  \xymatrix{\hat S^n\ar[d]_{\pi_C}\ar@{=}[r]&A^n\bigcup\{\infty\}\\
  S^n\ar@{=}[r]&\mathbb{R}^n\bigcup\{\infty\}\ar[u]_{s\equiv(\epsilon^n,id_{\infty})}\\}
\end{equation}
where $\epsilon^n:\mathbb{R}^n\to A^n$ is induced by the canonical ring homomorphism $\epsilon:\mathbb{R}\to A$. $s$ is a section of $\pi$: $\pi\circ s=id_{S^n}$. Let us yet denote by $S^n$ the image of $s$. So we can consider the canonical couple $(\hat S^n,S^n)$ as a CW pair, hence it has the homotopy extension property. $S^n$ is not a contractible subcomplex of $\hat S^n$, so in general the quotient map $\hat q:\hat S^n\to \hat S^n/S^n$ is not a homotopy equivalence.
We have the following lemma.

\begin{lemma}
The couple $(S^n,\infty)$ can be deformed into $(\hat S^n,\infty)$ to the base point $\{\infty\}$.
\end{lemma}
\begin{proof}
In fact, let $p\in\hat S^n\setminus S^n$. Then the inclusion $i:S^n\hookrightarrow\hat S^n$ is nullhomotopic since $\hat S^n\setminus\{\infty\}\thickapprox A^n$ (homeomorphism).
\end{proof}
Since $S^n$ is contractible into $\hat S^n$, to the point $\infty\in\hat S^n$, the quotient map $\hat q:\hat S^n\to \hat S^n/S^n$ can be deformed into quotient mapping $\hat q_t$ over deformed quotient spaces $X_t\equiv\hat S^n/S^n_t$, with $S^n_t\equiv f_t(S^n)\subset\hat S^n$, for some homotopy $f:I\times S^n\to \hat S^n$, such that $X_0=\hat S^n/S^n$, $X_1=\hat S^n$ and $\hat q_1=id_{\hat S^n}$. (See diagram (\ref{diagram-deformation-quotient-mappings}).)

\begin{equation}\label{diagram-deformation-quotient-mappings}
\xymatrix{\hat S^n\ar[rdd]_(0.5){\hat q_1}\ar[rd]^(0.4){\hat q_t}\ar[r]^(0.4){\hat q_0=\hat q}&\hat S^n/S^n\equiv X_0\ar@{.}[d]\\
&\hat S^n/S^n_t\equiv X_t\ar@{.}[d]\\
&\hat S^n/\{\infty\}=\hat S^n\equiv X_1}
\end{equation}
But this does not assure that $\hat q$ is a homotopy equivalence.\footnote{Rally $S^n$ is contractible in $\hat S^n$, but is not a contractible sub-complex of $\hat S^n$. This clarifies the meaning of Lemma \ref{contractibility-and-homotopy-equivalence}. For example, in the case $A=\mathbb{C}$, one has that $\hat S^1/S^1$ is not homotopy equivalent to $S^2$. In fact $\pi_2(S^2)=\mathbb{Z}$ and $\pi_2(S^2/S^1)\cong\pi_2(S^2\vee S^2)\cong H_2(S^2\vee S^2;\mathbb{Z})=\mathbb{Z}\bigoplus\mathbb{Z}$.} Let us, now, consider also some further lemmas.
\begin{lemma}
If $(X,A)$ is a CW pair and we have attaching maps $f,\, g:A\to X_0$ that are homotopic, then $X_0\bigcup_fX_1\backsimeq X_0\bigcup_gX_1$ $\hbox{\rm rel $X_0$}$ (homotopy equivalence).
\end{lemma}

\begin{lemma}
If $(X,A)$ satisfies the homotopy extension property and the inclusion $A\hookrightarrow X$ is a homotopy equivalence, then $A$ is a deformation retract of $X$.
\end{lemma}

\begin{lemma}
A map $f:X\to Y$ is a homotopy equivalence iff $X$ is a deformation retract of the mapping cylinder $M_f$.
\end{lemma}

Let us emphasize that we have a natural continuous mapping $\pi_C:\hat S^n\to S^n$, i.e., the surjection between the quantum $n$-sphere and its classic limit, identified by the commutative diagram (\ref{commutative-diagram-section-quantum-n-sphere}). The inclusion $i:S^n\hookrightarrow\hat S^n$ cannot be a deformation retract (and neither a strong deformation retract), otherwise $i$ should be a homotopy equivalence.\footnote{It is enough to consider the counterexample when $A=\mathbb{C}$ and $\hat S^1=\mathbb{C}\bigcup\{\infty\}=\mathbb{R}^2\bigcup\{\infty\}=S^2$. Then $S^1$ cannot be homotopy equivalent to $\hat S^1=S^2$, since $\pi_1(S^1)=\mathbb{Z}$ and $\pi_1(S^2)=0$.}
However, $\pi_C:\hat S^n\to S^n$, cannot be neither a retraction, otherwise their homotopy groups should be related by the split short exact sequence (\ref{split-short-exact-sequence-quantum-homotopy-groups-n-sphere}),
\begin{equation}\label{split-short-exact-sequence-quantum-homotopy-groups-n-sphere}
  \xymatrix{0\ar[r]&\pi_k(S^n,\infty)\ar@<1ex>[r]^{i_*}&\ar@<1ex>[l]^{r_*}\pi_k(\hat S^n,\infty)\ar[r]&\pi_k(\hat S^n,S^n,\infty)\ar[r]&\infty \\}
\end{equation}
hence we should have the splitting given in (\ref{splittings-homotopy-groups}). (For details on relations between homotopy groups and retractions see, e.g. \cite{PRA3}.)

\begin{equation}\label{splittings-homotopy-groups}
   \pi_k(\hat S^n,\infty)\cong \IM(i_*)\bigoplus\ker(r_*)\cong\pi_k(S^n,\infty)\bigoplus\ker(r_*).
\end{equation}
But this cannot work. In fact, in the case $A=\mathbb{C}$, we should have the commutative diagram (\ref{commutative-diagram-section-quantum-n-sphere-complex-case}) with exact horizontal lines.
\begin{equation}\label{commutative-diagram-section-quantum-n-sphere-complex-case}
  \xymatrix{0\ar[r]&\pi_1(S^1,\infty)\ar@{=}[d]\ar@<1ex>[r]^{i_*}&\ar@<1ex>[l]^{r_*}\pi_1(\hat S^1,\infty)\ar@{=}[d]\ar[r]&\pi_1(\hat S^1,S^1,\infty)\ar@{=}[d]\ar[r]&\infty \\
  0\ar[r]&\mathbb{Z}\ar[r]&0\ar[r]&\pi_1(\hat S^1,S^1,\infty)\ar[r]&0\\}
\end{equation}
This should imply that $\pi_1(S^1,\infty)=0$, instead that $\mathbb{Z}$, hence the bottom horizontal line in (\ref{commutative-diagram-section-quantum-n-sphere-complex-case}) cannot be an exact sequence, hence $\pi_C:\hat S^1\cong S^2\to S^1$ cannot be a retraction !
\end{proof}

\begin{cor}
Quantum homotopy spheres cannot be homotopy equivalent to $S^n$, except in the case that the quantum algebra $A$ reduces to $\mathbb{R}$.
\end{cor}

Quantum homotopy groups for quantum supermanifolds are introduced in \cite{PRA15}. In Tab. \ref{quantum-homotopy-groups-definitions-properties} are resumed their definitions and properties.

\begin{table}[t]\centering
\caption{Quantum homotopy groups: definitions and some properties.}
\label{quantum-homotopy-groups-definitions-properties}
\begin{tabular}{|l|l|}
\hline
\hfil{\rm{\footnotesize Definition}}\hfil&\hfil{\rm{\footnotesize Sum-law}}\hfil\\
\hline
\hfil{\rm{\footnotesize $\hat\pi_n(X,x_0)=[(\hat S^n,s_0),(X,x_0)]$}}\hfil&\hfil{\rm{\footnotesize $f+g:\hat S^n\to\hat S^n\vee \hat S^n\mathop{\to}\limits^{f\vee g}X$}}\hfil\\
\hline
\hfil{\rm{\footnotesize $\hat\pi_n(X,A,x_0)=[(\hat D^n,\hat S^{n-1},s_0),(X,A,x_0)]$}}\hfil&\hfil{\rm{\footnotesize $f+g:\hat D^n\to\hat D^n\vee \hat D^n\mathop{\to}\limits^{f\vee g}X$}}\hfil\\
\hline
\multicolumn{2}{|c|}{\rm{\footnotesize Some Properties.}}\\
\hline
\multicolumn{2}{|l|}{\rm{\footnotesize If $X$ is path connected one can simply write $\hat\pi_n(X)$, since it does not depend on $x_0$.}}\\
\multicolumn{2}{|l|}{\rm{\footnotesize If $A$ is path connected one can simply write $\hat\pi_n(X,A)$, since it does not depend on $x_0$.}}\\
\hline
\multicolumn{2}{|l|}{\rm{\footnotesize $\hat\pi_n(X,x_0)=\hat\pi_n(X,x_0,x_0)$ is abelian for $n\ge 2$.}}\\
\hline
\multicolumn{2}{|l|}{\rm{\footnotesize $\hat\pi_n(\Pi_\alpha X_\alpha)=\Pi_\alpha\hat\pi_n(X_\alpha)$, ($X_\alpha$ path-connected).}}\\
\hline
\multicolumn{2}{|l|}{\rm{\footnotesize $p:(\widetilde{X},\widetilde{x_0})\to(X,x_0)$ covering in $\mathfrak{Q}_{hyper}$:
 $p_*:\hat\pi_n(\widetilde{X},\widetilde{x_0})\cong  \hat\pi_n(X,x_0)$, $n\ge 2$.}}\\
\multicolumn{2}{|l|}{\rm{\footnotesize $\bullet$\hskip 2ptWhenever $X$ has a contractible universal cover $\hat\pi_n(X,x_0)=0$, $n\ge 2$.}}\\
\multicolumn{2}{|l|}{\rm{\footnotesize $\bullet$\hskip 2pt $A\to\hat S^1$, $\hat\pi_n(\hat S^1)=0$, $n\ge 2$. }}\\
\multicolumn{2}{|l|}{\rm{\footnotesize $\bullet$\hskip 2pt  $A^n\to\hat T^n=\underbrace{\hat S^1\times\hat S^1}_{n}$, $\hat\pi_k(\hat T^n)=0$, $k\ge 1$.}}\\
\hline
\multicolumn{2}{l}{\rm{\footnotesize $[f]=0\in\hat\pi_n(X,x_0)$ iff $\exists\, F_t(\hat S^n)\subset X,\, F_t(s_0)=x_0,\, F_0(\hat S^n)=x_0,\, F_1(\hat S^n)=f$.}}\\
\multicolumn{2}{l}{\rm{\footnotesize $[f]=0\in\hat\pi_n(X,x_0)$ iff $\exists\, F_t(\hat D^n)\subset X,\, F_t(\hat S^{n-1})\subset A,\, F_0(\hat D^n)=x_0,\, F_1(\hat D^n)=f$.}}\\
\end{tabular}
\end{table}

\begin{theorem}[Quantum homotopy groups of quantum $n$-sphere]\label{quantum-homotopy-groups-quantum-n-sphere}
Quantum homotopy $n$-spheres have quantum homotopy groups isomorphic to homotopy groups of $n$-spheres:
\begin{equation}\label{non-isomorphism-homotopy-groups-quantum-n-sphere-sphere}
    \hat\pi_{k}(\hat\Sigma^n)=\hat\pi_k(\hat S^n)\cong\pi_{k}(S^n).
\end{equation}
\end{theorem}
\begin{proof}
In fact, we can prove for examples that $\hat\pi_{k}(\hat S^n)\cong 0$, for $k<n$, and $\hat\pi_{n}(\hat S^n)\cong \mathbb{Z}$. For this it is enough to reproduce the anoalogous proofs for the commutative spheres, by substituting cells with quantum cells.
For example we can have the following quantum versions of analogous propositions for commutative CW complexes.

\begin{lemma}
Let $X$ be a quantum CW-complex admitting a decomposition in two quantum subcomplexes $X=A\bigcup B$, such that $A\bigcap B=C\not=\varnothing$. If $(A,C)$ is $m$-connected and $(B,C)$ is $n$-connected, $m,\, n\ge0$, then the mappings $\hat\pi_k(A,C)\to\hat\pi_k(X,B)$ induced by inclusion is an isomorphism for $k<m+n$, and a surjection for $k=m+n$.
\end{lemma}

\begin{lemma}[Quantum Freudenthal suspension theorem]\label{quantum-freudenthal-suspension-theorem}
The quantum suspension map $\hat\pi_k(\hat S^n)\to\hat\pi_{k+1}(\hat S^{n+1})$ is an isomorphism for $k<2n-1$, and a surjection for $k=2n-1$.\footnote{This holds also for quantum suspension $\hat\pi_k(X)\to\hat\pi_{k+1}(\hat S X$, for an $(n-1)$-connected quantum CW-complex $X$.}
\end{lemma}
As a by-product we get the isomorphism $\hat\pi_n(\hat S^n)\cong\mathbb{Z}$.

\end{proof}

\begin{remark}
Let us emphasize that Theorem \ref{quantum-homotopy-groups-quantum-n-sphere} does not allow to state that $S^n$ is a deformation retract of $\hat S^n$, as one could conclude by a wrong application of the Whitehead's theorem, reported in the following lemma.

\begin{lemma}[Whitehead's theorem]\label{whitehead-theorem}
If a map $f:X\to Y$ between connected CW complexes induces isomorphisms $f_*:\pi_n(X)\to\pi_*(Y)$ for all $n$, then $f$ is a homotopy equivalence. Furthermore, if $f$ is the inclusion of a subcomplex $f:X\hookrightarrow Y$, then $X$ is a deformation retract of $Y$.
\end{lemma}

In fact, in the case $i:S^n\hookrightarrow\hat S^n$ we are talking about different CW structures. One for $S^n$ is the usual one, the other, for $\hat S^n$ is the quantum CW structure. In order to easily understand the difference let us refer again to the case $A=\mathbb{C}$. Here one has $\pi_1(S^1)=\mathbb{Z}=\hat\pi_1(\hat S^1)$, but $\hat\pi_1(\hat S^1)=[\hat S^1,\hat S^1]=[S^2,S^2]=\pi_2(S^2)$. Furthermore, $\pi_1(\hat S^1)=[S^1,S^2]=0\not=\pi_1(S^1)$. Therefore, $\hat S^1$, with respect to the usual CW complex structure, has its first homotopy group zero, hence different from the first homotopy group of its classic limit $S^1$. In fact $S^1$ is not a deformation retract of $S^2=\hat S^1$. (Therefore there is not contradiction with the Whitehead's theorem.)
\end{remark}
Moreover, it is useful to formulate the quantum version of the Whitehead's theorem and some related lemmas. These can be proved by reproducing analogous proofs by substituting CW complex structure with quantum CW complex structure in quantum manifolds.

\begin{theorem}[Quantum Whitehead theorem]\label{quantum-whitehead-theorem}
If a map $f:X\to Y$ between connected quantum CW complexes induces isomorphisms $f_*:\hat\pi_n(X)\to\hat\pi_*(Y)$ for all $n$, then $f$ is a homotopy equivalence. Furthermore, if $f$ is the inclusion of a quantum subcomplex $f:X\hookrightarrow Y$, then $X$ is a quantum deformation retract of $Y$.
\end{theorem}

\begin{lemma}[Quantum compression lemma]\label{quantum-compression-lemma}
Let $(X,A)$ be a quantum CW pair and let $(Y,B)$ be any quantum pair with $B\not=\varnothing$. Let us assume that for each $n$ $\hat\pi_n(Y,B,y_{0})=0$, for all $y_0\in B$, and $X\setminus A$ has quantum cells of dimension $n$. Then, every map $f:(X,A)\to(Y,B)$ is homotopic ${\rm rel}_A$ to a map $X\to B$.\footnote{When $n=0$, the condition $\hat\pi_n(Y,B,y_{0})=0$, for all $y_0\in B$, means that $(Y,B)$ is $0$-connected. Let us emphasize that there is not difference between $0$-connected and quantum $0$-connected. In fact $[\hat S^0,Y]=\hat\pi_0(Y)=\pi_0(Y)=[S^0,Y]$, since $\hat S^0=S^0=(\{a\},\{b\})$, i.e., is a set of two points. However, after Theorem \ref{quantum-homotopy-groups-quantum-n-sphere}, there is not difference between the notion of {\em quantum $p$-connected} (i.e., $\hat\pi_k=0$, $k\le p$), quantum (homotopy) $n$-sphere, and {\em $p$-connected} (i.e., $\pi_k=0$, $k\le p$), (homotopy) $n$-sphere. In other words, a quantum homotopy $n$-sphere is quantum $(n-1)$-connected as well as its classic limit is $(n-1)$-connected.}
\end{lemma}

\begin{lemma}[Quantum extension lemma]\label{quantum-extension-lemma}
Let $(X,A)$ be a quantum CW pair and let $f:A\to Y$ be a mapping with $Y$ a path-connected quantum manifold. Let us assume that $\hat\pi_{n-1}(Y)=0$, for all $n$, such that $X\setminus A$ has quantum cells of dimension $n$. Then, $f$ can be extended to a map $f:X\to Y$.
\end{lemma}
\begin{proof}
The proof can be done inductively. Let us assume that $f$ has been extended over the quantum $(n-1)$-skeleton. Then, an extension over quantum $n$-cells exists iff the composition of the quantum cell's attaching map $\hat S^{n-1} \to \hat X^{n-1}$ with $f:\hat X^{n-1}\to Y$ is null homotopic.
\end{proof}
As a by-product of above results we get also the following theorems that relate quantum homotopy groups and quantum relative homotopy groups.

\begin{theorem}[Quantum exact long homotopy sequence]\label{quantum-exact-long-homotopy-sequence}
One has the exact sequence (\ref{quantum-exact-long-homotopy-sequence-a}).
\begin{equation}\label{quantum-exact-long-homotopy-sequence-a}
   \xymatrix{\cdots&\hat\pi_n(A,x_0)\ar[r]^{\hat i_*}&\hat\pi_n(X,x_0)\ar[r]^{\hat j_*}&\hat\pi_n(X,A,x_0)\ar[d]^{\hat \partial}\\
   &\hat\pi_{0}(X,x_0)&\ar[l]\cdots&\ar[l]\hat\pi_{n-1}(A,x_0)\\ }
\end{equation}
where $\hat i_*$ and $\hat j_*$ are induced by the inclusions $\hat i:(A,x_0)\hookrightarrow(X,x_0)$ and $\hat j:(X,x_0,x_0)\hookrightarrow(X,A,x_0)$ respectively. Furthermore, $\hat\partial$ comes from the following composition $(\hat S^{n-1},s_0)\hookrightarrow(\hat D^n,\hat S^{n-1},s_0)\to(X,A,x_0)$, hence $\hat\partial[f]=[f_{|\hat S^{n-1}}]$.
\end{theorem}
\begin{theorem}[Quantum Hurewicz theorem]\label{quantum-hurewicz-theorem}
The exact commutative diagram in {\em(\ref{commutative-diagram-quantum-hurewicz-theorem})} relates (quantum) homotopy groups and (quantum) homology groups for (quantum) homotopy $n$-spheres, $n\ge 2$. The morphisms $a$ and $b$ are isomorphisms for $p\le n$ and epimorphisms for $p=n+1$.\footnote{Compare with analogous theorem in \cite{PRA15} for quantum supermanifolds.}
\begin{equation}\label{commutative-diagram-quantum-hurewicz-theorem}
   \xymatrix{0\ar[r]&\pi_p(S^n)\ar[d]_{a}\ar@<0.5ex>[r]&\ar@<0.5ex>[l]\hat\pi_p(\hat S^n)\ar[d]_{b}\ar[r]&0\\
   0\ar[r]&H_p(S^n;\mathbb{Z})\ar[d]\ar@<0.5ex>[r]&\ar@<0.5ex>[l] H_p(\hat S^n;\mathbb{Z})\ar[d]\ar[r]&0\\
   &0&0&\\}
\end{equation}
\end{theorem}

The following propositions are also stated as direct results coming from Theorem \ref{quantum-homotopy-groups-quantum-n-sphere} and analogous propositions for topologic spaces.

\begin{proposition}\label{quantum-homotopic-properties-spheres}
The following propositions are equivalent for $i\le n-1$.

{\rm 1)} $S^i\to S^n$ is homotopic to a constant map.

{\rm 2)} $S^i\to S^n$ extends to a map $D^{i+1}\to S^n$.

{\rm 3)} $\hat S^i\to \hat S^n$ is homotopic to a constant map.

{\rm 4)} $\hat S^i\to \hat S^n$ extends to a map $\hat D^{i+1}\to \hat S^n$.
\end{proposition}

\begin{proof}
1) and 2) follow from the fact that $S^n$ is $(n-1)$-connected, and 3) and 4) from the fact that $\hat S^n$ is quantum $(n-1)$-connected. Furthermore, let us recall the following related result of Algebraic Topology.\footnote{There exists also a relative version of Lemma \ref{equivalent-conditions-for-n-connected-space}, saying that $\pi_i(X,A,x_0)=0$, for all $x_0\in A$,  is equivalent to one of the following propositions. (a1) Every map $(D^i,\partial D^i)\to(X,A)$ is homotopic ${\rm rel}\, \partial D^i$, to a map $D^i\to A$. (a2) Every map $(D^i,\partial D^i)\to(X,A)$ is homotopic through such maps to a map $D^i\to A$. (a3) Every map $(D^i,\partial D^i)\to(X,A)$ is homotopic through such maps to a constant map $D^i\to A$.}
\begin{lemma}\label{equivalent-conditions-for-n-connected-space}
The following propositions are equivalent.

{\rm (i)} The space $X$ is $n$-connected.

{\rm (ii)} Every map $f:S^i\to X$ is homotopic to a constant map.

{\rm (ii)} Every map $f:S^i\to X$ extends to a map $D^{i+1}\to X$.
\end{lemma}
\end{proof}

\begin{example}[Quantum-complex cupola]\label{quantum-complex-cupola}
Let us consider again $A=\mathbb{C}$, that offers beautiful examples, (even if commutative), easy to understand. Then, from points {\rm 3)} and {\rm 4)} in Proposition \ref{quantum-homotopic-properties-spheres} we get that $\hat S^1\to \hat S^2$ is homotopic to a constant map and extends to a map $\hat D^{2}\to \hat S^2$. These agree with points {\rm 1)} and {\rm 2)} in Proposition \ref{quantum-homotopic-properties-spheres}, as shown in the commutative diagram {\em(\ref{example-commutative diagram-quantum-homotopic-properties-spheres})}. In the {\em quantum-complex cupola} all the building is nullhomotopic to the base point $\{\infty\}\in \hat D^3$ identified with a $6$-cell. By using Proposition \ref{quantum-homotopic-properties-spheres} similar building can be obtained with any quantum algebra.
\end{example}
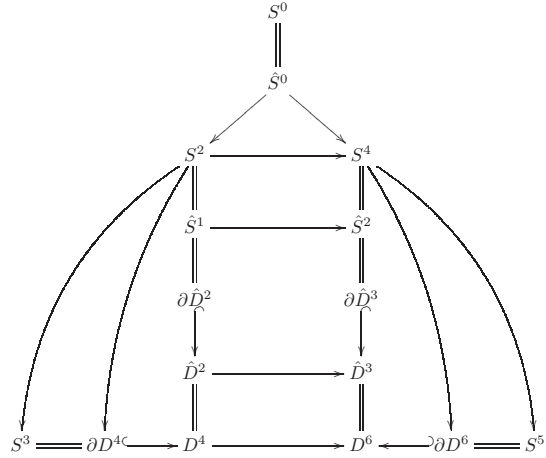
\begin{equation}\label{example-commutative diagram-quantum-homotopic-properties-spheres}
\scalebox{0.7}{$    \xymatrix{&&&S^0\ar@{=}[d]&&&\\
&&&\hat S^0\ar[dl]\ar[dr]&&&\\
&&S^2\ar@/_2pc/[ddddll]\ar@/_1pc/[ddddl]\ar@{=}[d]\ar[rr]&&S^4\ar@{=}[d]\ar@/^1pc/[ddddr]\ar@/^2pc/[ddddrr]&&\\
&&\hat S^1\ar@{=}[d]\ar[rr]&&\hat S^2\ar@{=}[d]&&\\
  &&\partial\hat D^2\ar@{^{(}->}[d]&&\partial\hat D^3\ar@{^{(}->}[d]&&\\
   &&\hat D^2\ar@{=}[d]\ar[rr]&&\hat D^3\ar@{=}[d]&& \\
S^3\ar@{=}[r]&\partial D^4\ar@{^{(}->}[r]&D^4\ar[rr]&&D^6&\ar@{_{(}->}[l]\partial D^6& \ar@{=}[l]S^5\\}$}
\end{equation}
Let us, now, consider quantum PDE's with respect to quantum homotopy $n$-spheres.

\begin{definition}[Quantum hypercomplex exotic PDE's]\label{quantum-hypercomplex-exotic-pdes}
Let $\hat E_k\subset \hat J^k_n(W)$ be a $k$-order PDE on the fiber bundle $\pi:W\to M$ in the category $\mathfrak{Q}_{hyper}$, with $\dim_AM=n$ and $\dim_BW=(n,m)$, where $B=A\times E$ and $E$ is also a $Z(A)$-module. We say that $\hat E_k$ is a {\em quantum exotic PDE} if it admits Cauchy integral manifolds $N\subset \hat E_k$, $\dim N=n-1$, such that one of the following two conditions is verified.

{\em(i)} $\hat\Sigma^{n-2}\equiv\partial N$ is a quantum exotic sphere of dimension $(n-2)$, i.e. $\hat\Sigma^{n-2}$ is homeomorphic to $\hat S^{n-2}$, ($\Sigma^{n-2}\thickapprox \hat S^{n-2}$) but not diffeomorphic to $\hat S^{n-2}$, ($\hat \Sigma^{n-2}\not\cong \hat S^{n-2}$).

{\em(ii)} $\varnothing=\partial N$ and $N\thickapprox \hat S^{n-1}$, but $N\not\cong \hat S^{n-1}$.\footnote{The following Refs. \cite{CERF, HIRSCH1, HIRSCH2, KAWA, KERVAIRE-MILNOR, KIRBY-SIEBENMAN, KLING, MAZUR, MILNOR1, MILNOR3, MOISE1, MOISE2, NASH, OHKAWA, SCHOEN-YAU, SMALE1, SMALE2, SMALE3, SULLIVAN, TOGNOLI} are important background for differential structures and exotic spheres.}
\end{definition}

\begin{definition}[Quantum hypercomplex exotic-classic PDE's]\label{quantum-hypercomplex-exotic-classic-pdes}
Let $\hat E_k\subset \hat J^k_n(W)$ be a $k$-order PDE as in Definition \ref{quantum-hypercomplex-exotic-pdes}. We say that $\hat E_k$ is a {\em quantum exotic-classic PDE} if it is a quantum exotic PDE, and the classic limit of the corresponding Cauchy quantum exotic manifolds are also exotic homotopy spheres.
\end{definition}

 From above results we get also the following one.

\begin{lemma}\label{lemma-quantum-hypercomplex-exotic-classic-pdes}
A quantum PDE $\hat E_k\subset \hat J^k_n(W)$, where $n$ is such that $\Theta_{n-1}=0$, cannot be a quantum exotic-classic PDE, in the sense of Definition \ref{quantum-hypercomplex-exotic-classic-pdes}.
\end{lemma}

\begin{lemma}\label{lemma-quantum-spheres-classic-limits-d}
For $n\in\{1,2,3,4,5,6\}$, one has the isomorphism reported in {\em(\ref{isomorphisms-lemma-quantum-spheres-classic-limits-d})}.

\begin{equation}\label{isomorphisms-lemma-quantum-spheres-classic-limits-d}
   \hat\Theta_n\cong\hat\Upsilon_n.
\end{equation}
In correspondence of such dimensions on $n$ we cannot have quantum exotic-classic PDE's.
\end{lemma}

\begin{proof}
Isomorphisms in (\ref{isomorphisms-lemma-quantum-spheres-classic-limits-d}), follow directly from above lemmas, and the fact that $\Theta_n=0$ for $n\in\{1,2,3,4,5,6\}$. (See Refs. \cite{PRA20, PRA21}.)
\end{proof}

\begin{example}[The quantum hypercomplex Ricci flow equation]\label{the-quantum-hypercomplex-ricci-flow-equation}
Let us recall that in \cite{PRA15} A. Pr\'astaro proved the generalized Poincar\'e conjecture in the category of quantum supermanifolds for classic regular, closed compact quantum supermanifolds $M$, of dimension $(m|n)$, homotopy equivalent to $\hat S^{m|n}$, when the quantum superalgebra $A$ has a Noetherian centre $Z(A)$. More precisely one has proved that $M$ is homeomorphic to $\hat S^{m|n}$ and its classic limit $M_C$ is homeomorphic to $S^m$. As a by-product it follows that
the quantum Ricci flow equation is a quantum exotic PDE for quantum $n$-dimensional Riemannian manifolds. Furthermore, the quantum Ricci flow equation cannot be quantum exotic-classic for $n<7$. (See \cite{PRA4, PRA14, PRA16, PRA17}.) (For complementary informations on the Ricci flow equation see also the following Refs. \cite{HAMIL1, HAMIL2, HAMIL3, HAMIL4, HAMIL5, PER1, PER2}.)
\end{example}

\begin{example}[The quantum hypercomplex Navier-Stokes equation]\label{the-quantum-hypercomplex-navier-stokes-equation}
The quantum Navier-Stokes equation can be encoded on the quantum extension of the affine fiber bundle $\pi:W\equiv M\times\mathbf{I}\times\mathbb{R}^2\to M$, $(x^\alpha,\dot x^i,p,\theta)_{0\le\alpha\le 3,1\le i\le 3}\mapsto(x^\alpha)$. (See Refs. \cite{PRA1, PRA6, PRA7, PRA9, PRA11, PRA14} for the Navier-Stokes equation in the category of commutative manifolds and \cite{PRA2-3, PRA3} for its quantum extension on quantum manifolds.) Therefore, Cauchy manifolds are $3$-dimensional space-like manifolds. For such dimension do not exist exotic spheres. Therefore, the Navier-Stokes equation cannot be a quantum exotic-classic PDE. Similar considerations hold for PDE's of the classical continuum mechanics.
\end{example}

\begin{example}[The quantum hypercomplex $n$-d'Alembert equation]\label{the-quantum-hypercomplex-n-d'alembert-equation}
The quantum $n$-d'Alembert equation on $A^n$ cannot be a quantum exotic-classic PDE for quantum $n$-dimensional Riemannian manifolds of dimension $n<7$ in the category $\mathfrak{Q}_{hyper}$. (See Example \ref{quantum-sedenionic-d-alembert-equation}, and \cite{PRA21} for exotic d'Alembert equations in the category of commutative manifolds.)
\end{example}

\begin{example}[The quantum hypercomplex Einstein equation]\label{the-quantum-hypercomplex-einstein-equation}
The quantum Einstein equation in the category $\mathfrak{Q}_{hyper}$, cannot be a quantum exotic-classic PDE for quantum $n$-dimensional space-times of dimension $n< 7$. Similar considerations hold for generalized quantum Einstein equations like quantum Einstein-Maxwell equation, quantum Einstein-Yang-Mills equation and etc, in the category $\mathfrak{Q}_{hyper}$.
\end{example}

\begin{theorem}[Integral bordism groups in quantum hypercomplex exotic PDE's in the category $\mathfrak{Q}_{hyper}$ and stability]\label{bordism-groups-in-quantum-hypercomplex-exotic-pdes-stability}
Let $\hat E_k\subset \hat J^k_n(W)$ be a quantum exotic formally integrable and completely integrable PDE on the fiber bundle $\pi:W\to M$, in the category $\mathfrak{Q}_{hyper}$, such that $\hat g_k\not=0$ and $\hat g_{k+1}\not=0$.\footnote{The fiber bundle $\pi:W\to M$ is as in Definition \ref{quantum-hypercomplex-exotic-pdes}, hence $\dim_AM=0$, $\dim_BW=(n,m)$, with $E$ endowed with a $Z(A)$-module structure too.}
Then there exists a topologic spectrum $\Xi_s$ such that for the singular integral $p$-(co)bordism groups can be expressed by means of suitable homotopy groups as reported in {\em(\ref{singular-integral-p-co-bordism-groups-exotic-pdes})}.
\begin{equation}\label{singular-integral-p-co-bordism-groups-exotic-pdes}
    \left\{
    \begin{array}{l}
      \Omega_{p,s}^{\hat E_k}=\mathop{\lim}\limits_{r\to\infty}\pi_{p+r}(\hat E_k^+\wedge\Xi_r)\\
      \Omega^{p,s}_{\hat E_k}=\mathop{\lim}\limits_{r\to\infty}[S^r\hat E_k^+,\Xi_{p+r}]\\
      \end{array}
    \right\}_{p\in\{0,1,\cdots,n-1\}}.
\end{equation}

Furthermore, the singular integral bordism group for admissible smooth closed compact Cauchy manifolds, $N\subset \hat E_k$, is given in {\em(\ref{singular-integral-bordism-group-n-1})}.
\begin{equation}\label{singular-integral-bordism-group-n-1}
    \Omega_{n-1,s}^{\hat E_k}\cong H_{n-1}(W;A).
\end{equation}
In the {\em quantum homotopy equivalence full admissibility hypothesis}, i.e., by considering admissible only $(n-1)$-dimensional smooth Cauchy integral manifolds identified with quantum homotopy spheres, and assuming that the space of conservation laws is not trivial,
one has $\Omega^{\hat E_k}_{n-1,s}=0$. Then $\hat E_k$ becomes a quantum extended $0$-crystal PDE. Then, there exists a global singular attractor, in the sense that all Cauchy manifolds, identified with quantum homotopy $(n-1)$-spheres, bound singular manifolds.

Furthermore, if in $W$ we can embed all the quantum homotopy $(n-1)$-spheres, and all such manifolds identify admissible smooth $(n-1)$-dimensional Cauchy manifolds of $\hat E_k$), then two of such Cauchy manifolds bound a smooth solution iff they are diffeomorphic and one has the following bijective mapping: $\Omega^{\hat E_k}_{n-1}\leftrightarrow\hat\Theta_{n-1}$.

Moreover, if in $W$ we cannot embed all quantum homotopy $(n-1)$-spheres, but only $\hat S^{n-1}$, then in the {\em quantum sphere full admissible hypothesis}, i.e., by considering admissible only quantum $(n-1)$-dimensional smooth Cauchy integral manifolds identified with $\hat S^{n-1}$, then $\Omega^{\hat E_k}_{n-1}=0$. Therefore $\hat E_k$ becomes a quantum $0$-crystal PDE and there exists a global smooth attractor, in the sense that two of such smooth Cauchy manifolds, identified with $\hat S^{n-1}$ bound quantum smooth manifolds. Instead, two Cauchy manifolds identified with quantum exotic $(n-1)$-spheres bound by means of quantum singular solutions only.

All above quantum smooth or quantum singular solutions are unstable. Quantum smooth solutions can be stabilized.
\end{theorem}

\begin{proof}
The relations (\ref{singular-integral-p-co-bordism-groups-exotic-pdes}) can be proved by a direct extension of analogous characterizations of integral bordism groups of PDE's in the category of commutative manifolds. (See Theorem 4.14 in \cite{PRA1}.) Furthermore, under the hypotheses of theorem we can apply Theorem 5.5 and Theorem 5.8 in \cite{PRA2} and Theorem 2.2 in \cite{PRA18} (or Theorem 2.3 in \cite{PRA18-0}). Thus we get directly (\ref{singular-integral-bordism-group-n-1}). Furthermore, under the quantum homotopy equivalence full admissibility hypothesis, all admissible quantum smooth $(n-1)$-dimensional Cauchy manifolds of $\hat E_k$, are identified with all possible quantum homotopy $(n-1)$-spheres. Moreover, all such Cauchy manifolds have same quantum integral characteristic numbers. (The proof is similar to the one given for Ricci flow PDE's in \cite{PRA15}.) Therefore, all such Cauchy manifolds belong to the same singular integral bordism class, hence $\Omega^{\hat E_k}_{n-1,s}=0$. Thus in such a case $\hat E_k$ becomes an quantum extended $0$-crystal PDE. When all quantum homotopy $(n-1)$-spheres can  be embedded in $W$ and so that in each smooth integral bordism class of $\Omega^{\hat E_k}_{n-1}$ are contained quantum homotopy $(n-1)$-spheres.\footnote{It is useful to emphasize that the possibility to embed quantum homotopy $(n-1)$-spheres in $W$, are related to the dimension of $W$. For example, in the cases where quantum algebras $A$ and $E$ are finite dimensional, the Whitney's embeddings theorem assures that embeddings exist if $\dim W\ge 2(n-1)+1=2n-1$. Furthermore, by considering that such embeddings should be compatible with $n$-dimensional submanifolds of $W$, we should also require that $\dim W\ge 2n+1$.} Then, since two quantum homotopy $(n-1)$-spheres bound a quantum smooth solution of $\hat E_k$ iff they are diffeomorphic, it follows that one has the bijection (but not isomorphism) $\Omega^{\hat E_k}_{n-1}\cong\hat\Theta_{n-1}$, where $\hat\Theta_{n-1}$ is the set group of equivalent classes of quantum diffeomorphic quantum homotopy $(n-1)$-spheres.\footnote{See Lemma \ref{theorem-quantum-spheres-classic-limits-b}. Therefore $\hat\Theta_{n-1}/\hat\Upsilon_{n-1}\cong\Theta_{n-1}$, where $\Theta_{n-1}$ is the corresponding set group of equivalent classes of diffeomorphic homotopy $(n-1)$-spheres.} In the quantum sphere full admissibility hypothesis we get $\Omega^{\hat E_k}_{n-1}=0$ and $\hat E_k$ becomes a quantum $0$-crystal PDE.

Let us assume now, that in $W$ we can embed only $\hat S^{n-1}$ and not all quantum exotic $(n-1)$-spheres. Then smooth Cauchy $(n-1)$-manifolds identified with quantum exotic $(n-1)$-spheres are necessarily integral manifolds with Thom-Boardman singularities, with respect to the canonical projection $\pi_{k,0}:\hat E_k\to W$. So solutions passing through such Cauchy manifolds are necessarily singular solutions. In such a case smooth solutions bord Cauchy manifolds identified with $\hat S^{n-1}$, and two diffeomorphic Cauchy manifolds identified with two quantum exotic $(n-1)$-spheres belonging to the same class in $\hat\Theta_{n-1}$, cannot bound quantum smooth solutions. Finally, if also $\hat S^{n-1}$ cannot be embedded in $W$, then there are not quantum smooth solutions bording smooth Cauchy $(n-1)$-manifolds in $\hat E_k$, identified with $\hat S^{n-1}$ or $\hat\Sigma^{n-1}$ (i.e., quantum exotic $(n-1)$-sphere). In other words $\Omega^{\hat E_k}_{n-1}$ is not defined in such a case !
\end{proof}

We are ready to state the main result of this paper that extends to the category of quantum manifolds Theorem 4.59 in \cite{PRA20} and Theorem 4.7 in \cite{PRA21-1}, given for homotopy spheres in the category of smooth manifolds.

\begin{theorem}[Integral h-cobordism in quantum hypercomplex Ricci flow PDE's]\label{integral-h-cobordism-in-Ricci-flow-pdes}
The quantum Ricci flow equation for quantum $n$-dimensional Riemannian manifolds, admits that starting from a quantum $n$-dimensional sphere $\hat S^n$, we can dynamically arrive, into a finite time, to any quantum $n$-dimensional homotopy sphere $M$. When this is realized with a smooth solution, i.e., solution with characteristic flow without singular points, then $\hat S^n\cong M$. The other quantum homotopy spheres $\hat\Sigma^n$, that are homeomorphic to $\hat S^n$ only, are reached by means of singular solutions.

For $1\le n\le 6$,  quantum hypercomplex Ricci flow PDE's cannot be quantum exotic-classic ones. In particular, the case $n=4$, is related to the proof that the smooth Poincar\'e  conjecture is true.
\end{theorem}

\begin{proof}
This is a direct consequence of Theorem 3.23 in \cite{PRA15} on the generalized Poincar\'e conjecture for quantum supermanifolds, Theorem 4.59 in \cite{PRA20} and Theorem 4.7 in \cite{PRA21-1} for homotopy spheres, where it is proved also the smooth Poincar\'e conjecture (i.e. in dimension four) for commutative manifolds.
\end{proof}

\end{document}